\documentclass[reqno]{amsart}
\usepackage[margin = 1in]{geometry}
\usepackage{amsmath, amssymb, amsthm, fancyhdr, verbatim, graphicx}
\usepackage{enumerate}
\usepackage[all]{xy}
\usepackage[dvipsnames]{xcolor}
\usepackage{mathrsfs}
\usepackage{tikz-cd}
\usepackage{framed}
\usepackage[hidelinks, pagebackref]{hyperref}
\hypersetup{colorlinks=true, linkcolor=magenta,citecolor=cyan} 
\usepackage[titletoc]{appendix}
\usepackage{bbm}
\usepackage{lipsum}
\usepackage{adjustbox}
\usepackage{mathfyz}
\usepackage{mathzyw}
\usepackage{envi}
\usepackage{stmaryrd}
\usepackage{leftindex}
\usepackage[english]{babel}

\numberwithin{equation}{section}

\makeatletter
\def\th@remark{%
  \thm@headfont{\bfseries}%
  \normalfont 
  \thm@preskip \thm@preskip 
  \thm@postskip\thm@preskip
}
\def\imod#1{\allowbreak\mkern5mu({\operator@font mod}\,\,#1)}
\makeatother

\numberwithin{equation}{section}

\title{Special cycle on Shtukas and categorical trace}
\author{Zeyu Wang}
\address{Massachusetts Institute of Technology, Department of Mathematics, 77 Massachusetts Avenue, Cambridge, MA 02139, USA}
\email{wangzeyu@mit.edu}

\begin{document}

\begin{abstract}
    In this article, we relate the fake special cycle classes $z_{\LL_{\s},r}$ attached to a Hecke eigensheaf $\LL_{\s}\in\Shv_{\Nilp}(\Bun_G)$ introduced in \cite{liu2025higherperiodintegralsderivatives} to the isotypic part of special cycles on Shtukas. As an application, we relate the self-intersection number of the isotypic part of special cycles arising from Rankin--Selberg period to higher derivatives of Rankin--Selberg $L$-functions.
\end{abstract}

\maketitle
\tableofcontents

\section{Introduction}
In \cite{YZ1}\cite{YZ2}, a higher Gross--Zagier formula is proved, which for certain types of cuspidal automorphic representations $\pi$ of $\PGL_2$ over the function field $K(C)$ for some curve $C$ over $\F_q$, it relates the self-intersection numbers of the $\pi$-isotypic part of the Heegner--Drinfeld cycles on the moduli of $\PGL_2$-Shtukas with $r$-legs to the $r$-th derivative of the $L$-functions $L(\pi,s)$ at the central point $s=1/2$. This formula can be regarded as a function field analog of the classical Gross--Zagier formula \cite{gross1986heegner} over the number field.

In \cite{liu2025higherperiodintegralsderivatives}, the authors proved a formula relating a certain norm of some ``fake" special cycle classes arising from Rankin--Selberg period to higher derivatives of Rankin--Selberg $L$-functions. More precisely, for geometrically irreducible Weil local system $\s_{n}, \s_{n-1}$ on $C$ of rank $n,n-1$. Consider $M=H^1(C_{\overline{\FF}_q},\s_{n}\otimes\s_{n-1})\oplus H^1(C_{\overline{\FF}_q},\s_n^*\otimes \s_{n-1}^*)$, which is a vector space naturally equipped with a symmetric bilinear form $\om_M$ arising from cup product. They defined some elements $z_{\LL_{\s}^{\FGV},r}\in (M^{\otimes r})^*$ arising from taking Hecke composed with Frobenius trace of some Hecke operators on a geometric period integral, and proved a formula relating $\om_{(M^{\otimes r})^*}(z_{\LL_{\s}^{\FGV},r},z_{\LL_{\s}^{\FGV},r})$ with the $r$-th derivative of the Rankin--Selberg $L$-function $L(\s_n\otimes\s_{n-1}\oplus\s_n^*\otimes\s_{n-1}^*,s)$ at the central value $s=1/2$. See \cite[Theorem\,1.1]{liu2025higherperiodintegralsderivatives} for the precise formulation.

The formula proved in \cite{liu2025higherperiodintegralsderivatives} was not about the intersection number of special cycles. Therefore, it cannot be regarded as a direct higher-dimensional analog of the higher Gross--Zagier formulas. Also, the fake special cycles considered in \textit{loc.cit} are not seemingly natural from their definition. However, it was claimed in \textit{loc.cit} that these seemingly artificially defined objects are indeed related to special cycle classes on Shtukas. The main subject of this article is to establish this relation. As an application, we prove a formula (Theorem \ref{thm:intro:main}) relating self-intersection numbers of $\s=\s_n\otimes\s_{n-1}$-isotypic part of special cycles on $\GL_n\times\GL_{n-1}$-Shtukas with higher derivatives of the Rankin--Selberg $L$-function $L(\s_n\otimes\s_{n-1}\oplus\s_n^*\otimes\s_{n-1}^*,s)$. This formula can be regarded as a direct generalization of the higher Gross--Zagier formula.

Two features in our formula are different from the higher Gross--Zagier formula in \cite{YZ1}: The first difference is in the spectral decomposition of the cohomology of Shtukas. In \cite{YZ1}, the spectral decomposition was made using the action of Hecke operators on the cohomology of Shtukas, in which the $\s$-isotypic part (actually, $\pi$-isotypic part for the automorphic representation $\pi$ attached to $\s$) was the isotypic part for the associated Hecke character of $\pi$. In this article, we directly construct a subspace of the cohomology of Shtukas (rather than a complete spectral decomposition) and use it as our subspace of the $\s$-isotypic part. However, the geometric Langlands conjecture for $\GL_n$ should imply that our $\s$-isotypic part is indeed the $\s$-isotypic part of the cohomology of Shtukas which can be defined using the (categorical) spectral action constructed in \cite{arinkin2022stacklocalsystemsrestricted} and the categorical trace interpretation of cohomology of Shtukas proved in \cite{arinkin2022automorphicfunctionstracefrobenius}. Therefore, our definition of $\s$-isotypic part should be completely canonical. The second difference lies in the definition of the isotypic part of special cycle classes and the intersection number. One significant difficulty in generalizing the work of \cite{YZ1} to higher-dimensional cases is that the special cycles are rarely compact in general, which makes defining the intersection number a challenge. In \cite{YZ1}, the special cycles are compact, and one can directly take their self-intersection number. However, the special cycles are not compact in our case, and we have to define the intersection number in an ad-hoc way. What we do is the following: Although the special cycles are not compact, one can still talk about their intersection numbers with compactly supported cohomology classes. Therefore, we regard these special cycles as functionals on the $\s$-isotypic part of the (compact support) cohomology of Shtukas (which are finite-dimensional vector spaces), on which the intersection pairing is non-degenerate. Therefore, we can define the ``self-intersection number" as the quadratic norm of these functionals under the dual of the intersection pairings. As far as we know, the non-degeneracy of intersection pairings on the isotypic part, which plays an essential role in our definition, was not previously known. 

\subsection{Main result: higher Rankin--Selberg integrals}\label{sec:intro:mainresult}
In this section, we formulate our main result on the intersection number of Rankin--Selberg cycles on Shtukas. Our main result is Theorem  \ref{thm:intro:main}, which confirms a version of Conjecture \ref{conj:intro:main}.

From \S\ref{sec:intro:sht} to \S\ref{sec:intro:lfunc}, we introduce notations and backgrounds. In \S\ref{sec:intro:mainconj} and \S\ref{sec:intro:mainres}, we formulate the main result.

Throughout the article, we fix a smooth projective geometrically connected curve $C$ over $\FF_q$.

\subsubsection{Moduli space of Shtukas}\label{sec:intro:sht}
We first recall the definition of the moduli space of $\GL_n$-Shtukas. 

Fix a positive integer $r\in\ZZ_{\geq 0}$. Consider $I=\{1,2,\cdots,r\}$. Take $\{\pm 1\}^r_0\sub \{\pm 1\}^r$ to be the subset consisting of sequences $\underline{\e}=(\e_1,\cdots,\e_r)$ such that $\sum_{i=1}^r\e_i=0$. For each $\e\in\{\pm 1\}^r$, we denote the moduli space of (iterated) $\GL_n$-Shtukas with $r$-legs and modification type $\Std_n^{\underline{\e}}$ by $\Sht_{\GL_n,\Std_n^{\underline{\e}}}$. It is is the moduli stack such that for any $\FF_q$-scheme $S$, we have $\Sht_{\GL_n,\Std_n^{\underline{\e}}}(S)$ is the groupoid of tuples \[( (c_i)_{i\in I}, \cE_0\xdashrightarrow{\e_1\cdot c_1}\cE_1 \xdashrightarrow{\e_2\cdot c_2}\cdots \xdashrightarrow{\e_r\cdot c_r}\cE_r, \alpha:\cE_0\cong (\id_C\times\Frob_S)^* \cE_r )\] where:
\begin{itemize}
    \item $c_i\in C(S)$ for $i=1,\cdots,r$;
    \item $\cE_i$ is a vector bundle of rank $n$ on $C\times S$ for $i=0,\cdots,r$;
    \item $\cE_{i-1}\xdashrightarrow{\e_{i}\cdot c_i}\cE_i$ is an isomorphism of vector bundles over $C\times S\backslash \Gamma_{c_i}$ ($\Gamma_{c_i}\sub C\times S$ is the graph of the map $c_i:S\to C$) such that \begin{itemize}
        \item If $\e_i=1$, the inverse of the map induces an inclusion $\cE_{i}\sub\cE_{i-1}$ such that $\cE_{i-1}/\cE_i$ is supported on $\Gamma_{c_i}$ on which it is locally-free of rank $1$;
        \item If $\e_i=-1$, the map induces an inclusion $\cE_{i-1}\sub\cE_i$ such that $\cE_i/\cE_{i-1}$ is supported on $\Gamma_{c_i}$ on which it is locally-free of rank $1$;
        \end{itemize}
    \item $\alpha:\cE_0\cong (\id_C\times\Frob_S)^*\cE_r$ is an isomorphism of vector bundles. Here $\Frob_S:S\to S$ is the relative Frobenius map over $\FF_q$.
\end{itemize}
The moduli space $\Sht_{\GL_n,\Std_n^{\underline{\e}}}$ turns out to be a Deligne-Mumford stack locally of finite type over $C^I$. We use $l_{I}:\Sht_{\GL_n,\Std_n^{\underline{\e}}}\to C^I$ to denote the natural map sending above data to $(c_i)_{i\in I}$. The moduli space $\Sht_{\GL_n,\Std_n^{\underline{\e}}}$ is non-empty if and only if $\underline{\e}\in\{\pm 1\}^r_0$. We have $\dim \Sht_{\GL_n,\Std_n^{\underline{\e}}}=nr$.

When $\underline{\e}\in\{\pm 1\}^r_0$, one has $\pi_0(\Sht_{\GL_n,\Std_n^{\underline{\e}}})\cong \ZZ$ where the isomorphism is given by taking the degree of $\cE_0$ as a vector bundle over $C$. This gives a decomposition into connected components \begin{equation}\label{eq:intro:degreedecompsht}\Sht_{\GL_n,\Std_n^{\underline{\e}}}=\coprod_{d\in \ZZ} \Sht_{\GL_n,\Std_n^{\underline{\e}}}^d.\end{equation}

In our case, we are interested in the moduli space of $\GL_n\times\GL_{n-1}$-Shtukas \begin{equation}
    \Sht_{\GL_n\times\GL_{n-1},(\Std_n\boxtimes\Std_{n-1})^{\underline{\e}}}=  \Sht_{\GL_n,\Std_n^{\underline{\e}}}\times_{C^I}\Sht_{\GL_{n-1},\Std_{n-1}^{\underline{\e}}},
\end{equation} whose decomposition into connected components is \begin{equation}
    \Sht_{\GL_n\times\GL_{n-1},(\Std_n\boxtimes\Std_{n-1})^{\underline{\e}}} =\coprod_{(d_n,d_{n-1})\in\ZZ^2}\Sht_{\GL_n\times\GL_{n-1},(\Std_n\boxtimes\Std_{n-1})^{\underline{\e}}}^{(d_n,d_{n-1})}
\end{equation} in which \[\Sht_{\GL_n\times\GL_{n-1},(\Std_n\boxtimes\Std_{n-1})^{\underline{\e}}}^{(d_n,d_{n-1})}=\Sht_{\GL_n,\Std_n^{\underline{\e}}}^{d_n}\times_{C^I}\Sht_{\GL_{n-1},\Std_{n-1}^{\underline{\e}}}^{d_{n-1}}.\]

\subsubsection{Rankin--Selberg cycles}
The classical story of (everywhere unramified) Rankin--Selberg integrals (over function fields) concerns the integration\footnote{Since we are over function fields, integration here means summation.} of $f\in\Fun_c(\Bun_{\GL_n\times \GL_{n-1}}(\FF_q))$ (the vector space of $\Qlbar$-valued functions with compact support) over the diagonal map $\pi(\FF_q):\Bun_{\GL_{n-1}}(\FF_q)\to \Bun_{\GL_{n}\times \GL_{n-1}}(\FF_q)$ in which the map $\Bun_{\GL_{n-1}}\to \Bun_{\GL_n}$ is given by taking direct sum with the trivial line bundle. Or equivalently speaking, one is interested in the function $\pi(\FF_q)_! 1_{\Bun_{\GL_{n-1}}(\FF_q)}\in \Fun(\Bun_{\GL_n\times \GL_{n-1}}(\FF_q))$ where $\pi(\FF_q)_!$ is summation along fibers.

Note that $\Bun_{\GL_n}(\FF_q)=\Sht_{\GL_n,\Std_n^{\underline{\e}}}$ for $r=0$. The principle of higher integrals is replacing the groupoid $\Bun_G(\FF_q)$ by the moduli stack $\Sht_{G,I}$. In particular, in the Rankin--Selberg case, instead of considering the characteristic function $\pi(\FF_q)_! 1_{\Bun_{\GL_{n-1}}(\FF_q)}\in \Fun(\Bun_{\GL_n\times \GL_{n-1}}(\FF_q))$, one studies the \emph{Rankin--Selberg cycle classes} (or cohomological Rankin--Selberg cycles) \begin{equation}\label{eq:intro:rscycles}\pi_{\Sht,I,!}[\Sht_{\GL_{n-1},\Std_{n-1}^{\underline{\e}}}]\in H^{\BM}_{2(n-1)r}(\Sht_{\GL_n\times\GL_{n-1},(\Std_n\boxtimes\Std_{n-1})^{\underline{\e}}}) \end{equation} where \begin{itemize}
    \item $[\Sht_{\GL_{n-1},\Std_{n-1}^{\underline{\e}}}]\in H^{BM}_{2(n-1)r}(\Sht_{\GL_{n-1},\Std_{n-1}^{\underline{\e}}})$ is the fundamental class of $\Sht_{\GL_{n-1},\Std_{n-1}^{\underline{\e}}}$ as a Borel-Moore homology class;
    \item $\pi_{\Sht,I}:\Sht_{\GL_{n-1},\Std_{n-1}^{\underline{\e}}}\to \Sht_{\GL_n\times\GL_{n-1},(\Std_n\boxtimes\Std_{n-1})^{\underline{\e}}}$ is the diagonal map which is taking direct sum with the trivial vector bundle on the first factor. The map $\pi_{\Sht,I}$ is finite schematic by \cite{yun2022special}.
    \item $\pi_{\Sht,I,!}:H^{BM}_{2(n-1)r}(\Sht_{\GL_{n-1},\Std_{n-1}^{\underline{\e}}})\to H^{\BM}_{2(n-1)r}(\Sht_{\GL_n\times\GL_{n-1},(\Std_n\boxtimes\Std_{n-1})^{\underline{\e}}})$ is the push-forward of Borel-Moore homology class along proper maps.
\end{itemize}
Here, we consider Borel-Moore homology with $\Qlbar$-coefficient for stacks over $\overline{\FF}_q$. For stacks defined over $\FF_q$, if not otherwise specified, we always consider the Borel-Moore homology (or cohomology) of its base change to $\overline{\FF}_q$. In this introduction, we drop all the Tate twist.

\subsubsection{Spectral decomposition of cohomology of Shtukas}
In this article, by cohomology, we by default mean cohomology with compact support in $\Qlbar$-coefficient of stacks over $\overline{\FF}_q$. The cohomology of Shtukas admits a direct sum decomposition into isotypic parts for Langlands parameters.

To apply the machinery of geometric Langlands in positive characteristic, we need to pose the following assumption on the characteristic of the base field:

\begin{assumption}[Assumption on characteristic]\label{assumption:glcchar}
    Throughout the article, we keep the same assumption on the characteristic of the base field as in \cite[\S0.1.9]{GR}.
\end{assumption}

For each split reductive group $G$, we use $\Gc$ to denote the Langlands dual group. In \cite[\S24.1]{arinkin2022stacklocalsystemsrestricted}, the authors define a quasi-compact algebraic stack $\Loc_{\Gc}^{\arith}$ over $\Qlbar$ (denoted $\mathrm{LocSys}_{\Gc}^{\mathrm{arithm}}(X)$ for $X=C$ in \textit{loc.cit}) which is the moduli space of $\Gc$-(Weil) local systems over $C$. This moduli stack has also been considered in \cite{zhu2021coherentsheavesstacklanglands}. 

The ring of global sections $\Gamma(\Loc_{\Gc}^{\arith},\cO_{\Loc_{\Gc}^{\arith}})$ is called the \textit{algebra of excursion operators}. An equivalent form of it was first considered by \cite{genestier2018chtoucasrestreintspourles}. This algebra naturally acts on the cohomology of $G$-Shtukas. In this formulation, the existence of this action is a consequence of \cite[Main Theorem\,0.3.10]{arinkin2022automorphicfunctionstracefrobenius}.

In our case, we get an action of $\Gamma(\Loc_{\GL_n}^{\arith},\cO_{\Loc_{\GL_n}^{\arith}})$ on $\Gamma_c(\Sht_{\GL_n,\underline{\e}},\underline{\Qlbar})$ for each $\underline{\e}\in \{\pm 1\}^r$. It gives rise to a direct sum decomposition \begin{equation}\label{eq:intro:spectraldecompositiongln}\Gamma_c(\Sht_{\GL_n,\underline{\e}},\underline{\Qlbar})=\bigoplus_{s\in \pi_0(\Loc_{\GL_n}^{\arith})} \Gamma_c(\Sht_{\GL_n,\underline{\e}},\underline{\Qlbar})_s\end{equation} where $\Gamma_c(\Sht_{\GL_n,\underline{\e}},\underline{\Qlbar})_s$ is supported on the connected component of $\Loc_{\GL_n}^{\arith}$ indexed by $s$. 

For a Weil local system $\s_n\in\Loc_{\GL_n}^{\arith}(\Qlbar)$, we use $\Loc_{\Gc,\s_n}^{\arith}\sub \Loc_{\Gc}^{\arith}$ to denote its underlying connected component. We use $\Gamma_c(\Sht_{\GL_n,\underline{\e}},\underline{\Qlbar})_{\Loc_{\Gc,\s_n}^{\arith}}$ to denote the corresponding term in the spectral decomposition \eqref{eq:intro:spectraldecompositiongln}.  

The action above extends to an action of $H^*(\Loc_{\GL_n}^{\arith},\cO_{\Loc_{\GL_n}^{\arith}})$ on $\prod_{d\in\ZZ}H^*_c(\Sht_{\GL_n,\underline{\e}}^d,\underline{\Qlbar})$. We use \begin{equation}\label{eq:intro:weilspectraldecompositiongln}(\prod_{d\in\ZZ}H^*_c(\Sht_{\GL_n,\underline{\e}}^d,\underline{\Qlbar}))_{\s_n}\sub \prod_{d\in\ZZ}H^*_c(\Sht_{\GL_n,\underline{\e}}^d,\underline{\Qlbar})\end{equation} to denote the maximal sub-module (scheme theoretically) supported on $\s_n\in \Spec H^0(\Loc_{\GL_n}^{\arith},\cO_{\Loc_{\GL_n}^{\arith}})$. For an irreducible Weil local system $\s_n$, this sub-module turns out to be a perfect complex, and it is what we mean by the $\s_n$-isotypic part of the cohomology of Shtukas. Note that we are taking a direct product instead of a direct sum here to avoid getting an empty sub-module.

We use $\overline{\s}_n$ to denote the base change of $\s_n$ to $\FF_q$, which we call the underlying geometric local system of $\s_n$. We say that $\s_n$ is geometrically irreducible if $\overline{\s}_n$ is irreducible. If $\s_n$ is geometrically irreducible, the underlying connected component $\Loc_{\GL_n,\s_n}^{\arith}$ is isomorphic to $[\Gm/\Gm]$ in which $\Gm$ acts trivially. Here, the first $\Gm$ can be regarded as the moduli of different Weil sheaf structures on the underlying geometric local system $\overline{\s}_n$. The second $\Gm$ is the moduli of automorphisms of $\s_n$. In this case, the isotypic part defined above admits the following conjectural description:
\begin{conj}\label{conj:intro:descriptionofisotypicpart}
    For each $\underline{\e}\in\{\pm 1\}^r_0$, there is a canonical isomorphism \[\Gamma_c(\Sht_{\GL_n,\underline{\e}},\underline{\Qlbar})_{\Loc_{\GL_n,\s_n}^{\arith}}[(n-1)r]\cong \Gamma(C^I,\s_n^{\underline{\e}})\otimes \cO(\Loc_{\GL_n,\s_n}^{\arith})\] where $\s_n^{\underline{\e}}=\boxtimes_{i\in I}\s_n^{\e_i}$. In particular, it gives a canonical isomorphism \begin{equation}\label{eq:intro:shtukaeigensub}(\prod_{d\in\ZZ}H^{*+(n-1)r}_c(\Sht_{\GL_n,\underline{\e}}^d,\underline{\Qlbar}))_{\s_n}\cong H^*(C^I,\s_n^{\underline{\e}}).\end{equation}
\end{conj}

\begin{remark}
    Conjecture \ref{conj:intro:descriptionofisotypicpart} is a consequence of the geometric Langlands conjecture with restricted variation for $\GL_n$ in characteristic $p$ formulated in \cite{arinkin2022stacklocalsystemsrestricted}. In particular, it is a consequence of \cite{GR}.
\end{remark}

\begin{remark}
    While the isomorphism in Conjecture \ref{conj:intro:descriptionofisotypicpart} is claimed to be canonical, there are several different canonical choices indeed. Choice of such an isomorphism can be roughly thought of as a Shtuka analog of the choice of a Hecke eigenform: one can choose it to be Whittaker normalized, $L^2$-normalized... Whenever using such an isomorphism, one should specify which isomorphism is being used.
\end{remark}

For $\GL_n\times\GL_{n-1}$-Shtukas, one can similarly define the action by excursion operators \[H^*(\Loc_{\GL_n\times\GL_{n-1}}^{\arith},\cO_{\Loc_{\GL_n\times\GL_{n-1}}^{\arith}})=H^*(\Loc_{\GL_n}^{\arith},\cO_{\Loc_{\GL_n}^{\arith}})\otimes H^*(\Loc_{\GL_{n-1}}^{\arith},\cO_{\Loc_{\GL_{n-1}}^{\arith}})\] on \[\prod_{(d_n,d_{n-1})\in\ZZ^2}H^*_c(\Sht_{\GL_n\times\GL_{n-1},(\Std_n\boxtimes\Std_{n-1})^{\underline{\e}}}^{(d_n,d_{n-1})},\underline{\Qlbar}).\] One can define the $\s=(\s_n,\s_{n-1})\in \Loc_{\GL_n\times\GL_{n-1}}^{\arith}(\Qlbar)$-isotypic part \begin{equation}\label{eq:intro:nn-1isotypicpart}(\prod_{(d_n,d_{n-1})\in\ZZ^2}H^*_c(\Sht_{\GL_n\times\GL_{n-1},(\Std_n\boxtimes\Std_{n-1})^{\underline{\e}}}^{(d_n,d_{n-1})},\underline{\Qlbar}))_{\s}\sub \prod_{(d_n,d_{n-1})\in\ZZ^2}H^*_c(\Sht_{\GL_n\times\GL_{n-1},(\Std_n\boxtimes\Std_{n-1})^{\underline{\e}}},\underline{\Qlbar}).\end{equation}

\subsubsection{Intersection pairing}
The cohomology of $G$-Shtukas carries an intersection pairing. When $G=\GL_n$, we have the intersection pairing \begin{equation}\label{eq:intro:intersectiongln}
\langle-,-\rangle_{\underline{\e}}^d:
    H_c^{nr}(\Sht_{\GL_n,\underline{\e}}^d,\underline{\Qlbar})\otimes H_c^{nr}(\Sht_{\GL_n,\underline{\e}}^d,\underline{\Qlbar})\xrightarrow{\cup} H_c^{2nr}(\Sht_{\GL_n,\underline{\e}}^d,\underline{\Qlbar})\cong \Qlbar
\end{equation} induced by cup product and taking the degree of $0$-cycles. The following concerns the behavior of the intersection pairing on the $\s_n$-isotypic part \eqref{eq:intro:weilspectraldecompositiongln}:
\begin{conj}\label{conj:intro:intersectionnondeggln}
    For each $d\in\ZZ$, $\underline{\e}\in\{\pm 1\}^r_0$, and irreducible Weil local system $\s_n\in\Loc_{\GL_n}^{\arith}(\Qlbar)$, the restriction of the intersection pairing $\langle-,-\rangle_{\underline{\e}}^d$ defines a non-degenerate bilinear form \begin{equation}\label{eq:intro:diagisotypiccycle}\langle-,-\rangle_{\underline{\e},\s_n}^d:(\prod_{e\in\ZZ} H_c^{nr}(\Sht_{\GL_n,\underline{\e}}^e,\underline{\Qlbar}))_{\s_n}\otimes (\prod_{e\in\ZZ} H_c^{nr}(\Sht_{\GL_n,\underline{\e}}^e,\underline{\Qlbar}))_{\s_n^*}\to \Qlbar.\footnote{Here we are studying the intersection pairing on $\Sht^d_{\GL_n,\underline{\e}}$ which is concentrated in a single degree $d$ on the $\s_n$-isotypic part which spreads into all degrees (accounts for $e\in\ZZ$).}\end{equation}
\end{conj}

Given Conjecture \ref{conj:intro:intersectionnondeggln}, we can use the bilinear form $\langle-,-\rangle_{\underline{\e},\s_n}^d$ to make an identification \[(\prod_{e\in\ZZ} H_c^{nr}(\Sht_{\GL_n,\underline{\e}}^e,\underline{\Qlbar}))_{\s_n^*}^*\cong (\prod_{e\in\ZZ} H_c^{nr}(\Sht_{\GL_n,\underline{\e}}^e,\underline{\Qlbar}))_{\s_n}.\] This gives a non-degenerate bilinear form \begin{equation}\label{eq:intro:dualintersection}
    \langle-,-\rangle_{\underline{\e},\s_n}^{d,*}:(\prod_{e\in\ZZ}H_c^{nr}(\Sht_{\GL_n,\underline{\e}}^e,\underline{\Qlbar}))_{\s_n}^*\otimes (\prod_{e\in\ZZ}H_c^{nr}(\Sht_{\GL_n,\underline{\e}}^e,\underline{\Qlbar}))_{\s_n^*}^*\to \Qlbar
.\end{equation}

For $\GL_n\times\GL_{n-1}$-Shtukas and an irreducible Weil local system $\s=(\s_n,\s_{n-1})\in \Loc_{\GL_n\times\GL_{n-1}}^{\arith}(\Qlbar)$, consider $\s^*=(\s_n^*,\s_{n-1}^*)$. One can similarly define a bilinear form \begin{equation}\label{eq:intro:nn-1isotypicintersection}\begin{split}
    \langle-,-\rangle_{\underline{\e},\s}^{(d_n,d_{n-1}),*}:(\prod_{(e_n,e_{n-1})\in\ZZ^2}H^{2(n-1)r}_c(\Sht_{\GL_n\times\GL_{n-1},(\Std_n\boxtimes\Std_{n-1})^{\underline{\e}}}^{(e_n,e_{n-1})},\underline{\Qlbar}))_{\s}^*\otimes \\ (\prod_{(e_n,e_{n-1})\in\ZZ^2}H^{2(n-1)r}_c(\Sht_{\GL_n\times\GL_{n-1},(\Std_n\boxtimes\Std_{n-1})^{\underline{\e}}}^{(e_n,e_{n-1})},\underline{\Qlbar}))_{\s^*}^*\to\Qlbar
\end{split}.\end{equation} In this setting, we also expect Conjecture \ref{conj:intro:intersectionnondeggln}.

\subsubsection{$L$-function of local system}\label{sec:intro:lfunc}
For a Weil local system $\s$ on $C$, its $L$-function is defined by \begin{equation}
    L(\s,s)=\det (1-q^{-s}\Frob,\Gamma(C,\s))^{-1}=\prod_{i=0}^2\det(1-q^{-s}\Frob,H^i(C,\s))^{(-1)^{i-1}}.
\end{equation} The root number is defined by \begin{equation}\label{eq:intro:rootnumber}
    \e(\s)=\det(\Frob,\Gamma(C,\s)(\frac{1}{2}))^{-1}=\prod_{i=0}^2\det(q^{-1/2}\Frob,H^i(C,\s))^{(-1)^{i-1}}.
\end{equation} 


\subsubsection{Statement of the main conjecture}\label{sec:intro:mainconj}
Before stating our main result, we will first formulate a conjecture, which will be confirmed by our main result under some assumptions that are satisfied in the most interesting cases.

Borel-Moore homology can be naturally regarded as the dual of cohomology with compact support. In our case, we have an isomorphism \[H^{\BM}_{2(n-1)r}(\Sht_{\GL_n\times\GL_{n-1},(\Std_n\boxtimes\Std_{n-1})^{\underline{\e}}})\cong H_{c}^{2(n-1)r}(\Sht_{\GL_n\times\GL_{n-1},(\Std_n\boxtimes\Std_{n-1})^{\underline{\e}}},\underline{\Qlbar})^*.\] Therefore, one can view the special cycle classes \eqref{eq:intro:rscycles} as functionals on the cohomology with compact support and study its restriction on the $\s$-isotypic part \eqref{eq:intro:nn-1isotypicpart}. This gives \begin{equation}\label{eq:intro:isotypicrscycles}
    (\pi_{\Sht,I,!}[\Sht_{\GL_{n-1},\Std_{n-1}^{\underline{\e}}}^d])_{\s}\in (\prod_{(e_n,e_{n-1})\in\ZZ^2}H^{2(n-1)r}_c(\Sht_{\GL_n\times\GL_{n-1},(\Std_n\boxtimes\Std_{n-1})^{\underline{\e}}}^{(e_n,e_{n-1})},\underline{\Qlbar}))_{\s}^*
.\end{equation}

We have the following conjecture:
\begin{conj}\label{conj:intro:main}
    Assume $\s_n,\s_{n-1}$ are irreducible Weil local systems on $C$ with rank $n,n-1$. Take $\s=(\s_n,\s_{n-1})\in\Loc_{\GL_n\times\GL_{n-1}}^{\arith}(\Qlbar)$. Then $(\pi_{\Sht,I,!}[\Sht_{\GL_{n-1},\Std_{n-1}^{\underline{\e}}}^h])_{\s}$ is non-zero for finitely many $h\in\ZZ$. Take \begin{equation}\label{eq:intro:totalrscycle}\begin{split}
        (\pi_{\Sht,I,!}[\Sht_{\GL_{n-1},\Std_{n-1}^{\underline{\e}}}])_{\s}=\sum_{h\in\ZZ}(\pi_{\Sht,I,!}[\Sht_{\GL_{n-1},\Std_{n-1}^{\underline{\e}}}^h])_{\s} \\ \in (\prod_{(e_n,e_{n-1})\in\ZZ^2}H^{2(n-1)r}_c(\Sht_{\GL_n\times\GL_{n-1},(\Std_n\boxtimes\Std_{n-1})^{\underline{\e}}}^{(e_n,e_{n-1})},\underline{\Qlbar}))_{\s}^*.
    \end{split}\end{equation} For any $(d_n,d_{n-1})\in\ZZ^2$, we have
    \begin{equation}\label{eq:intro:main}
    \begin{split}
        \sum_{\underline{\e}\in\{\pm 1\}^r_0} \langle (\pi_{\Sht,I,!}[\Sht_{\GL_{n-1},\Std_{n-1}^{\underline{\e}}}])_{\s},(\pi_{\Sht,I,!}[\Sht_{\GL_{n-1},\Std_{n-1}^{\underline{\e}}}])_{\s^*} \rangle_{\underline{\e},\s}^{(d_n,d_{n-1}),*} \\= q^{\dim\Bun_{\GL_{n-1}}}(\ln q)^{-r-2}\frac{\left(\frac{d}{ds}\right)^{r}\Big|_{s=1/2}\widetilde{L}(\s_n\otimes\s_{n-1}\oplus\s_n^*\otimes\s_{n-1}^*,s)}{\Res_{s=1}\widetilde{L}(\s_n\otimes\s_n^*,s)\Res_{s=1}\widetilde{L}(\s_{n-1}\otimes\s_{n-1}^*,s)}
        \end{split}. 
    \end{equation} Here $\dim\Bun_{\GL_{n-1}}=(n-1)^2(g-1)$ where $g$ is the genus of the curve $C$. The normalized $L$-functions are defined as \begin{equation}\label{eq:intro:normlfunc}
        \widetilde{L}(\s_n\otimes\s_{n-1}\oplus\s_n^*\otimes\s_{n-1}^*,s)=q^{2n(n-1)(g-1)(s-1/2)}L(\s_n\otimes\s_{n-1}\oplus\s_n^*\otimes\s_{n-1}^*,s)
    \end{equation}
    \begin{equation}
        \widetilde{L}(\s_n\otimes\s_n^*,s)=q^{n^2(g-1)s}L(\s_n\otimes\s_n^*,s)
    \end{equation}
        \begin{equation}
        \widetilde{L}(\s_{n-1}\otimes\s_{n-1}^*,s)=q^{(n-1)^2(g-1)s}L(\s_{n-1}\otimes\s_{n-1}^*,s)
    .\end{equation}
\end{conj}

\subsubsection{Statement of the main result}\label{sec:intro:mainres}
Our main result confirms a slightly different formulation of Conjecture \ref{conj:intro:main} when $\s_n,\s_{n-1}$ are geometrically irreducible. 



In this case, instead of working with the subspace \begin{equation}\label{eq:intro:spectralsub2}\begin{split}(\prod_{(e_n,e_{n-1})\in\ZZ^2}H^{2(n-1)r}_c(\Sht_{\GL_n\times\GL_{n-1},(\Std_n\boxtimes\Std_{n-1})^{\underline{\e}}}^{(e_n,e_{n-1})},\underline{\Qlbar}))_{\s}\\ \sub \\ \prod_{(e_n,e_{n-1})\in\ZZ^2}H^{2(n-1)r}_c(\Sht_{\GL_n\times\GL_{n-1},(\Std_n\boxtimes\Std_{n-1})^{\underline{\e}}}^{(e_n,e_{n-1})},\underline{\Qlbar})\end{split}\end{equation} defined via the spectral action, we directly construct an injection \begin{equation}\label{eq:intro:spectralsub3}H^*\xi_{\s,\underline{\e}}:H^r(C^I,(\s_n\otimes\s_{n-1})^{\underline{\e}})\to \prod_{(e_n,e_{n-1})\in\ZZ^2}H^{2(n-1)r}_c(\Sht_{\GL_n\times\GL_{n-1},(\Std_n\boxtimes\Std_{n-1})^{\underline{\e}}}^{(e_n,e_{n-1})},\underline{\Qlbar}).\end{equation} The image of $H^*\xi_{\s,\underline{\e}}$ should coincide with the subspace \eqref{eq:intro:spectralsub2}: While we will not show this in this article, we believe it is an easy consequence of the geometric Langlands conjecture for $\GL_n$, in particular, the result of \cite{GR}. Therefore, the map $H^*\xi_{\s,\underline{\e}}$ should give a choice of the isomorphism \eqref{eq:intro:shtukaeigensub}. The construction of this map will be given in \S\ref{sec:app:isotypicpart}.

Throughout the article after this point, we will always use \eqref{eq:intro:spectralsub3} as our $\s$-isotypic part. The meaning of \eqref{eq:intro:isotypicrscycles} and \eqref{eq:intro:nn-1isotypicintersection} will also be adapted to the isotypic part \eqref{eq:intro:spectralsub3}. We can formulate our main theorem, which will be proved in \S\ref{sec:app:compute}:
\begin{thm}\label{thm:intro:main}
    After replacing the subspace \eqref{eq:intro:spectralsub2} by \eqref{eq:intro:spectralsub3} and assuming
    \begin{itemize}
        \item $\s_n,\s_{n-1}$ are geometrically irreducible,
        \item $p>n$,
    \end{itemize}
    the Conjecture \ref{conj:intro:intersectionnondeggln} and Conjecture \ref{conj:intro:main} hold.
\end{thm}

\begin{remark}
    The first assumption is used for simplicity in constructing a Hecke eigensheaf. The second assumption comes from Assumption \ref{assumption:glcchar}.
\end{remark}

\subsection{Idea of proof of the main result}
In this section, we explain the idea of proof of Theorem  \ref{thm:intro:main}.

From now on, we take $k=\Qlbar$. We work with a split reductive group $G$. To prove Theorem  \ref{thm:intro:main}, we compute explicitly all the terms on the left-hand side of \eqref{eq:intro:main} so that we can explicitly compare both sides. This includes writing down explicitly the isotypic part of Rankin--Selberg cycle classes $(\pi_{\Sht,I,!}[\Sht_{\GL_{n-1},\Std_{n-1}^{\underline{\e}}}^h])_{\s}$ as well as the isotypic part of the intersection pairing $\langle -,- \rangle_{\underline{\e},\s}^{(d_n,d_{n-1}),*}$. These two elements can be uniformly understood by taking categorical trace of the corresponding geometric relative Langlands statements in \cite{BZSV}.

\subsubsection{Categorical trace}\label{sec:intro:cattr}

We first briefly explain the formalism of categorical trace. We refer to \S\ref{sec:cattrace} for a more precise explanation. For a dualizable presentable $(\infty,1)$-category $\cC$ with an endomorphism $F\in\End(\cC)$, its categorical trace is a space $\tr(F,\cC)$. Therefore, to understand a space $S$, one can try to find a pair $(\cC,F)$ such that $S\cong \tr(F,\cC)$. This allows one to study the richer category $\cC$ instead of the space $S$.

In our case, we would like to understand a morphism between spaces $z:S_1\to S_2$. This can be achieved by the formalism of the functoriality of the categorical trace. Suppose we can write $S_1\cong\tr(F_1,\cC_1)$ and $S_2\cong\tr(F_2,\cC_2)$. For each continuous (i.e. colimit preserving) functor $L:\cC_1\to\cC_2$  admitting a continuous right adjoint, given a natural transformation $\eta:L\circ F_1\to F_2\circ L$, there is an induced morphism between spaces $\tr(\eta):\tr(F_1,\cC_1)\to\tr(F_2,\cC_2)$. Therefore, one could look for a natural transformation $\eta$ such that $\tr(\eta)=z$, which changes the study of a map between spaces to the study of a natural transformation.

The discussion above also makes sense when working with categories linear over a symmetric monoidal category $\cA$. In that case, one replaces the word ``spaces" by ``objects in $\cA$" and ``functors" by ``$\cA$-linear functors". We use $\tr_{\cA}(\cC)\in\cA$ to denote the ($\cA$-linear) categorical trace. See Example \ref{eg:lincat} for a more precise treatment.

\subsubsection{Fundamental diagram}
In this section, we pretend that $G$ is semisimple, which is unfortunately not the case since we care about the case $G=\GL_n$. However, this would simplify the situation and is enough for explaining the idea.

In our setting, we would like to understand the sequence of morphisms \begin{equation}\label{diag:intro:funddiagsht}
    V_{\s}^I\langle-d_I\rangle\xrightarrow{\xi_{\s,I}} l_{I,!}(\IC_{V^I}|_{\Sht_{G,I}}\langle -d_I\rangle)\xrightarrow{[Z_{V^I}^X]}\uk_{C^I}
.\end{equation} Here, \begin{itemize}
    \item $I$ is a finite set, $V^I\in\Rep(\Gc^I)$, $\s\in\Loc_{\Gc}^{\arith}(k)$ is a Langlands parameter, $V_{\s}^I$ is the local system on $C^I$ involved in the Tannakian definition of $\s$, $\langle n\rangle=\Pi(\frac{n}{2})[n]$ is the shearing. Here $\Pi$ means changing the parity, which can be safely ignored at this point. $d_I\in\ZZ$;
    \item $\Sht_{G,I}$ is the moduli of $G$-Shtukas with $I$-legs without bounding the poles. $\IC_{V^I}|_{\Sht_{G,I}}\in\Shv(\Sht_{G,I})$ is the sheaf attached by the geometric Satake equivalence. $l_I:\Sht_{G,I}\to C^I$ is the map remembering only the legs. When $G=\GL_n$, $V^I=\Std_n^{\underline{\e}}$, $\s=\s_n$, $d_I=r$, we have $\Gamma_c(l_{I,!}(\IC_{V^I}|_{\Sht_{G,I}}\langle -d_I\rangle))=\Gamma_c(\Sht_{\GL_n,\underline{\e}},\underline{\Qlbar})\langle (n-2)r \rangle$.
    \item $\uk_{C^I}$ is the constant sheaf on $C^I$;
    \item $\xi_{\s,I}:V_{\s}^I\to  l_{I,!}(\IC_{V^I}|_{\Sht_{G,I}})$ is the $\s$-isotypic part map. In $\GL_n$-case as above, ignoring the difference between direct sum and direct product, we have $H^*\xi_{\s,I}=H^*\xi_{\s_n,\underline{\e}}:H^*(C^I,\s_n^{\underline{\e}})\to H^{*+(n-1)r}_c(\Sht_{\GL_n,\underline{\e}},\underline{\Qlbar})((n-1)r/2)$ which is the $\GL_n$-version of the map \eqref{eq:intro:spectralsub3}.
    \item $[Z_{V^I}^X]$ is a special cycle class attached to an affine smooth $G$-variety $X$. In this article, we are particularly interested in two cases: \begin{itemize}
        \item Rankin--Selberg case: $G=\GL_n\times\GL_{n-1}$, $X=\GL_{n-1}\backslash \GL_{n}\times\GL_{n-1}$, $V^I=(\Std_n\boxtimes\Std_{n-1})^{\underline{\e}}$, and \[[Z_{V^I}^X]=\pi_{\Sht,I,!}[\Sht_{\GL_{n-1},\Std_{n-1}^{\underline{\e}}}].\]
        \item Group case: $G=\GL_n\times\GL_n$, $X=\GL_n\backslash \GL_n\times\GL_n$, $V^I=(\Std_n\boxtimes\Std_{n})^{\underline{\e}}$, and \[[Z_{V^I}^X]=\D_{\Sht,I,!}[\Sht_{\GL_{n},\Std_{n}^{\underline{\e}}}]=\langle-,-\rangle_{\underline{\e}}.\] Here, $\D_{\Sht,I}:\Sht_{\GL_{n},\Std_{n}^{\underline{\e}}}\to\Sht_{\GL_{n}\times\GL_n,(\Std_n\boxtimes\Std_{n})^{\underline{\e}}}$ is the diagonal map.
    \end{itemize}

\end{itemize}

Following the principle in \S\ref{sec:intro:cattr}, we switch to understanding a sequence of functors between categories with endomorphisms. We consider the following \emph{fundamental diagram}:
\begin{equation}\label{diag:intro:funddiagcat}\begin{tikzcd}
    \QLisse(C^I) \ar[r, "(-\otimes\LL_{\s})\otimes\id"] \ar[d, "-\otimes V_{\s}^I\langle-d_I\rangle"'] & \Shv_{\Nilp}(\Bun_G)\otimes\QLisse(C^I) \ar[r, "\int_{X,\Nilp,I}"] \ar[d, "(\Frob\times\id)_!\circ T_{V^I\langle-d_I\rangle}" description] & \QLisse(C^I) \ar[d, "\id"] \\
    \QLisse(C^I) \ar[r, "(-\otimes\LL_{\s})\otimes\id"'] \ar[ur, Rightarrow, "\eta_{\s}^{(1)}"] & \Shv_{\Nilp}(\Bun_G)\otimes\QLisse(C^I) \ar[r, "\int_{X,\Nilp,I}"'] \ar[ur, Rightarrow, "\eta_{\frc_{V^I}^X,\Nilp}^{(1)}"'] & \QLisse(C^I)
\end{tikzcd}.\end{equation}

This diagram contains a wealth of information. Let us explain it step by step.
We first explain the categories in \eqref{diag:intro:funddiagcat}: \begin{itemize}
    \item $\QLisse(C^I)$ is the category of sheaves with liss\'e cohomologies on $C^I$ defined in \cite[Definition\,1.2.6]{arinkin2022stacklocalsystemsrestricted};
    \item $\Shv_{\Nilp}(\Bun_G)$ is the category of sheaves on $\Bun_G$ with nilpotent singular support. This is the main player in the geometric Langlands with restricted variation \cite{arinkin2022stacklocalsystemsrestricted}.
\end{itemize}

We then explain the functors between categories, which are horizontal maps in \eqref{diag:intro:funddiagcat}: \begin{itemize}
    \item $\LL_{\s}\in\Shv_{\Nilp}(\Bun_G)$ is a choice of Hecke eigensheaf with eigenvalue $\s\in\Loc_{\Gc}(k)$. When $G=\GL_n$ and $\s=\s_n$ is geometrically irreducible, such an eigensheaf is constructed in \cite{frenkel2002geometric}.
    \item $\int_{X,\Nilp,I}=\Gamma_c(-\otimes\cP_X)\otimes\id$ where $\cP_X\in\Shv(\Bun_G)$ is the period sheaf attached to the affine smooth $G$-variety $X$ as before. See \cite[\S10.3]{BZSV}. We are particularly interested in two cases: \begin{itemize}
        \item $G=\GL_n\times\GL_{n-1}$, $X=\GL_{n-1}\backslash \GL_{n}\times\GL_{n-1}$, $\cP_X=\pi_!\uk_{\Bun_{\GL_{n-1}}}$ where \[\pi:\Bun_{\GL_{n-1}}\to\Bun_{\GL_n\times\GL_{n-1}}.\]
        \item $G=\GL_n\times\GL_{n}$, $X=\GL_{n}\backslash \GL_{n}\times\GL_{n}$, $\cP_X=\D_!\uk_{\Bun_{\GL_{n}}}$ where \[\D:\Bun_{\GL_{n}}\to\Bun_{\GL_n\times\GL_{n}}.\]
    \end{itemize}
\end{itemize}

Now we come to the endomorphisms of categories which are vertical maps in \eqref{diag:intro:funddiagcat}: \begin{itemize}
    \item $T_{V^I}:\Shv_{\Nilp}(\Bun_G)\otimes\QLisse(C^I)\to\Shv_{\Nilp}(\Bun_G)\otimes\QLisse(C^I)$ is the Hecke operator attached to $V^I\in\Rep(\Gc^I)$.
    \item $\Frob:\Bun_G\to\Bun_G$ is the (relative) Frobenius defined over $\FF_q$.
\end{itemize}

Finally, we explain the natural transformations in \eqref{diag:intro:funddiagcat}: \begin{itemize}
    \item The natural transformation $\eta_{\s}^{(1)}$ comes from the Frobenius equivariant structure and Hecke eigen-property of $\LL_{\s}$.
    \item The natural transformation $\eta_{\frc_{V^I}^X,\Nilp}^{(1)}$ comes from a cohomological correspondence $\frc_{V^I}^X$ on the period sheaf $\cP_X$. The procedure of attaching such a natural transformation to a cohomological correspondence will be explained in \S\ref{sec:cctrnt}.
\end{itemize}

Pretending that the horizontal maps preserve compact objects, by the functoriality of categorical trace applied to \eqref{diag:intro:funddiagcat}, we get the upper row in the diagram:
\begin{equation}\label{diag:intro:funddiagtr}
\adjustbox{scale=.7,center}{%
    \begin{tikzcd}
        \tr_{\QLisse(C^I)}(-\otimes V_{\s}^I\langle-d_I\rangle,\QLisse(C^I)) \ar[r, "\tr_{\QLisse(C^I)}(\eta_{\s}^{(1)})"] \ar[d, "\sim"] & \tr_{\QLisse(C^I)}((\Frob\times\id)_!\circ T_{V^I\langle-d_I\rangle},\Shv_{\Nilp}(\Bun_G)\otimes\QLisse(C^I)) \ar[r, "\tr_{\QLisse(C^I)}(\eta_{\frc_{V^I}^X,\Nilp}^{(1)})"] \ar[d, "\LT^{\Serre}", "\sim"']& \tr_{\QLisse(C^I)}(\id, \QLisse(C^I)) \ar[d, "\sim"] \\
        V_{\s}^I\langle-d_I\rangle\ar[r, "\xi_{\s,I}"] & l_{I,!}(\IC_{V^I}|_{\Sht_{G,I}}\langle -d_I\rangle)\ar[r, "\lbrack Z_{V^I}^X\rbrack"] & \uk_{C^I}
    \end{tikzcd}
.}\end{equation}
Note that the lower row of the diagram \eqref{diag:intro:funddiagtr} is \eqref{diag:intro:funddiagsht}. 

We first explain the vertical isomorphisms in \eqref{diag:intro:funddiagtr}:
\begin{itemize}
    \item The left and right vertical isomorphisms follow from the most obvious computation of categorical trace.
    \item The middle map $\LT^{\Serre}$ is defined in \cite[\S5.4]{arinkin2022automorphicfunctionstracefrobenius}. The fact that it is an isomorphism is the main subject of \cite{arinkin2022automorphicfunctionstracefrobenius}. See \S\ref{sec:geo=cat:ltserre} for a discussion.
\end{itemize}

We now explain the commutativity of the diagram \eqref{diag:intro:funddiagtr}:
\begin{itemize}
    \item The commutativity of the left square defines the $\s$-isotypic part map $\xi_{\s,I}:V_{\s}^I\to l_{I,!}(\IC_{V^I}|_{\Sht_{G,I}})$.
    \item The commutativity of the right square will be illustrated in \S\ref{sec:intro:maintool}. In this article, we prove it under Assumption \ref{assumption:goodfil}.
\end{itemize}

\begin{remark}
    In the discussion above, we pretend that the horizontal functors in \eqref{diag:intro:funddiagcat} preserve compact objects. While this is usually satisfied when $G$ is semisimple and $\s$ is irreducible, it mostly fails for non-semisimple split reductive groups, which unfortunately includes the case $G=\GL_n$. In \S\ref{sec:isotypicpartscc2}, we will discuss a replacement of \eqref{diag:intro:funddiagcat} for general split reductive groups.
\end{remark}

\subsubsection{Fake special cycle classes}

Now we only need to understand the outer square of \eqref{diag:intro:funddiagtr}. For this, we need to understand the natural transformation $\eta_{\frc_{V^I}^X,\Nilp}^{(1)}\circ \eta_{\s}^{(1)}$ for the outer square of \eqref{diag:intro:funddiagcat}, and we need to understand how to compute $\tr_{\QLisse(C^I)}(\eta_{\frc_{V^I}^X,\Nilp}^{(1)}\circ \eta_{\s}^{(1)})$. This is the study of fake special cycle classes, which is the main subject of \cite{liu2025higherperiodintegralsderivatives}.

From the cohomological correspondence $\frc_{V^I}^X$, one obtains a map \begin{equation}\label{eq:intro:heckeaction}
    a_{\frc_{V^I}^X,\s}: V_{\s}^I\langle-d_I\rangle\otimes\int_{X,\Nilp}\LL_{\s}\to \uk_{C^I}\otimes\int_{X,\Nilp}\LL_{\s}
.\end{equation} Under the assumption that horizontal maps in \eqref{diag:intro:funddiagcat} preserve compact objects, the complex $\int_{X,\Nilp}\LL_{\s}$ is a perfect complex. Therefore, one arrives at the \emph{fake special cycle class} \begin{equation}\label{eq:intro:fakespecialcycleclasses}\begin{split}
    z_{\frc_{V^I}^X,\s}:V_{\s}^I\langle-d_I\rangle &\xrightarrow{\id\otimes\coev_{\int_{X,\Nilp}\LL_{\s}}} V_{\s}^I\langle-d_I\rangle \otimes\int_{X,\Nilp}\LL_{\s}\otimes (\int_{X,\Nilp}\LL_{\s})^*
    \\&\xrightarrow{\id\otimes\Frob\otimes\id}V_{\s}^I\langle-d_I\rangle \otimes\int_{X,\Nilp}\LL_{\s}\otimes (\int_{X,\Nilp}\LL_{\s})^*
    \\&\xrightarrow{a_{\frc_{V^I}^X,\s}\otimes\id}\uk_{C^I} \otimes\int_{X,\Nilp}\LL_{\s}\otimes (\int_{X,\Nilp}\LL_{\s})^*
    \\&\xrightarrow{\id\otimes\ev_{\int_{X,\Nilp}\LL_{\s}}} \uk_{C^I}
\end{split}.\end{equation} See \eqref{eq:fakespecialcycleclasses} for an equivalent definition of $z_{\frc_{V^I}^X,\s}$. These elements have been studied in \cite[\S1.1.5]{liu2025higherperiodintegralsderivatives}. It turns out that the composition of the top horizontal map in \eqref{diag:intro:funddiagtr} coincides with $z_{\frc_{V^I}^X,\s}$.

We are particularly interested in the following cases:
\begin{itemize}
    \item When $G=\GL_n\times\GL_{n-1}$, $X=\GL_{n-1}\backslash\GL_n\times\GL_{n-1}$. The fake special cycle classes \eqref{eq:intro:fakespecialcycleclasses} are understood via Theorem \ref{thm:app:kolydermain}.
    \item When $G=\GL_n\times\GL_{n}$, $X=\GL_n\backslash\GL_n\times\GL_n$. The fake special cycle classes are described in Theorem \ref{thm:geoLMglntrsht}.
\end{itemize}

During the proof of results above, we heavily use the automorphic commutator relation developed in \cite[\S4]{liu2025higherperiodintegralsderivatives}. See \S\ref{sec:commutatorrelation} for a brief recollection.

\subsection{Main tool: Categorical trace interpretation of special cycle classes}\label{sec:intro:maintool}
In this section, we explain in more detail the right commutative square of \eqref{diag:intro:funddiagtr}. It turns out we can make sense of it and prove it in two cases: \begin{itemize}
    \item Homogeneous minuscule case: $G$ is split reductive, $X=H\backslash G$ for some spherical split reductive subgroup $H\sub G$, $V^I\in\Rep(\Gc^I)$ is irreducible minuscule, $[Z_{V^I}^X]=\pi_{\Sht,I,!}[\Sht_{H,V_H^I}]$ for some irreducible minuscule representation $V_H^I\in\Rep(\Hc^I)$. This includes all the cases we need in the proof of Theorem  \ref{thm:intro:main} and will be introduced in \S\ref{sec:intro:mincycleclass}.
    \item Diagonal case: $H$ is split reductive, $G=H\times H$, $X=H\backslash H\times H$, $V^I=V_H^I\boxtimes V_H^I\in\Rep(\Gc^I)$ for some irreducible representation $V_H^I\in\Rep(\Hc^I)$. This case will be introduced in \S\ref{sec:intro:diagcycleclass}. This case in general will not be used in the proof of Theorem  \ref{thm:intro:main}, but we will use it to prove the non-degeneracy of intersection pairing for $\GL_n$ with arbitrary coweight in Corollary \ref{cor:app:nondegenracy}, which has its own interest.
\end{itemize}

\subsubsection{Minuscule homogeneous special cycle classes} \label{sec:intro:mincycleclass}
Consider a split connected reductive group $G$ and a spherical split connected reductive subgroup $H$. Take $X=H\backslash G$. We choose maximal tori $T\sub G$ and $T_H\sub H$ such that $T_H\sub T$. We use $X_*(T)$, $X_*(T_H)$ to denote the coweight lattices.

For the finite set $I=\{1,2\cdots,r\}$ and a sequence of minuscule coweights $\lambda_{H,I}=(\lambda_{H,1},\cdots,\lambda_{H,r})\in X_*(T_H)^I$, we obtain a sequence of coweights $\lambda_{I}=(\lambda_1,\cdots,\lambda_r)\in X_*(T)^I$ such that $\lambda_i$ in the image of $\lambda_{H,i}$ under the natural map $X_*(T_H)\sub X_*(T)$. Let $V_{\lambda_I}\in\Rep(\Gc^I)$ be the irreducible representation with highest weight $\lambda_I$. We use $\Sht_{G,\lambda_I}\sub \Sht_{G,I}$ to denote the closed Schubert cell of type $\lambda_I$.

We define the \emph{minuscule homogeneous special cycle class} \[[Z_{V^I}^X]=\pi_{\Sht,I,!}[\Sht_{H,\lambda_{H,I}}/C^I]\in H_{-d_{\lambda_I}+2d_{\lambda_{H,I}}}^{BM}(\Sht_{G,I}/C^I,\IC_{V^I}|_{\Sht_{G,I}})\] in which \[ H_{-d_{\lambda_I}+2 d_{\lambda_{H,I}}}^{BM}(\Sht_{G,{\lambda_I}}/C^I,\IC_{V^I}|_{\Sht_{G,I}})=\Hom^0(l_{I,!}(\IC_{V^I}|_{\Sht_{G,I}})\langle  -d_{\lambda_I}+2d_{\lambda_{H,I}}\rangle,\uk_{C^I}) .\] Here,\begin{itemize}
    \item $\pi_{\Sht,I}:\Sht_{H,\lambda_{H,I}}\to\Sht_{G,\lambda_I}$ is finite and schematic by \cite{yun2022special}.\footnote{In some part of this article, we write $\Sht_{G,\lambda_I}=\Sht_{G,V_{\lambda_I}}$.}
    \item $[\Sht_{H,\lambda_{H,I}}/C^I]\in H_{2 d_{\lambda_{H,I}}}^{BM}(\Sht_{H,\lambda_{H,I}}/C^I)\cong H_{2 (d_{\lambda_{H,I}}+r)}^{BM}(\Sht_{H,\lambda_{H,I}})$ is the (relative) fundamental class.
    \item $l_{I}:\Sht_{G,\lambda_{I}}\to C^I$ is the map remembering only the legs.
    \item $d_{\lambda_{H,I}}=\sum_{i=1}^r\langle2\rho_H,\lambda_{H,i}\rangle=\dim(\Sht_{H,\lambda_{H,I}})-\dim(C^I)$, $d_{\lambda_I}=\sum_i\langle 2\rho_G,\lambda_i\rangle$.
\end{itemize}

We now explain the right square in \eqref{diag:intro:funddiagcat}. We first explain the Hecke operator $T_{V^I}$. Consider the correspondence \[\begin{tikzcd}
    \Bun_G\times C^I & \Hk_{G,I} \ar[l, "\lh_I"'] \ar[r, "\rh_I"]& \Bun_G\times C^I 
\end{tikzcd}.\]
For $V^I\in\Rep(\Gc^I)^{\heartsuit}$, via the geometric Satake equivalence, one obtains a (super) sheaf $\IC_{V^I}\in\Shv(\Hk_{G,I})$ normalized such that it is perverse on each fiber of $\Hk_{G,I}\to C^I$ with parity same as $\sum_{i\in I}\langle 2\rho_G,\lambda_i\rangle$. This gives the Hecke operator \[T_{V^I}=\rh_{I,*}(\lh_I^!(-)\otimes^{!} \IC_{V^I})\cong\rh_{I,!}(\lh_I^*(-)\otimes\IC_{V^I}):\Shv(\Bun_G\times C^I)\to\Shv(\Bun_G\times C^I)\] which preserves the full-subcategory $\Shv_{\Nilp}(\Bun_G)\otimes\QLisse(C^I)\sub\Shv(\Bun_G\times C^I)$.

We now turn to the definition of the natural transformation $\eta_{\frc_{V^I}^X,\Nilp}^{(1)}=\eta^{(1)}_{\frc_{\lambda_{H,I}},\Nilp}$ via cohomological correspondence. We leave it to \S\ref{sec:cccor} for notations and operations involving cohomological correspondences.

Consider the diagram
\begin{equation}\label{diag:mincorfund}\begin{tikzcd}
    \Bun_H\times C^I \ar[d, "\pi\times\id"] & \Hk_{H,\lambda_{H,I}} \ar[l, "\lh_{H,I}"'] \ar[r, "\rh_{H,I}"] \ar[d,"\pi_{\Hk,I}"] & \Bun_H\times C^I \ar[d, "\pi\times\id"] \\
    \Bun_G\times C^I & \Hk_{G,\lambda_{I}} \ar[l, "\lh_{I}"'] \ar[r, "\rh_I"] & \Bun_G\times C^I
\end{tikzcd}.\end{equation} We have the relative fundamental class of $\lh_{H,I}:\Hk_{H,\lambda_{H,I}}\to \Bun_H\times C^I$ which can be regarded as a cohomological correspondence \[\begin{split} [\Hk_{H,\lambda_{H,I}}/\Bun_H\times C^I]&\in \Cor_{\Hk_{H,\lambda_{H,I}},\uk\langle 2d_{\lambda_{H,I}}\rangle}(\uk_{\Bun_H\times C^I},\uk_{\Bun_H\times C^I}) \\ &=H_{2d_{\lambda_{H,I}}}^{BM}(\Hk_{H,\lambda_{H,I}}/\Bun_H\times C^I) \\ &=\Hom^0(\rh_{H,I}^*\uk_{\Bun_H\times C^I}\langle 2d_{\lambda_{H,I}} \rangle,\lh_{H,I}^!\uk_{\Bun_H\times C^I} )\end{split}.\]
We apply the push-forward of cohomological correspondence in \S\ref{sec:ccpush}. Since that the right square of \eqref{diag:mincorfund} is pushable, given $(\pi\times\id)_!\uk_{\Bun_H\times C^I}\cong\cP_X\boxtimes\uk_{C^I}$, one can define the \emph{minuscule homogeneous cohomological correspondence} \[\begin{split}\frc_{\lambda_{H,I}}:=\pi_{\Hk,I,!}[\Hk_{H,\lambda_{H,I}}/\Bun_H\times C^I]&\in \Cor_{\Hk_{G,\lambda_I},\uk\langle 2d_{\lambda_{H,I}} \rangle}(\cP_X\boxtimes\uk_{C^I},\cP_X\boxtimes\uk_{C^I})\\
&=\Hom^0(\rh_{I}^*(\cP_X\boxtimes\uk_{C^I})\langle 2d_{\lambda_{H,I}} \rangle ),\lh_{I}^!(\cP_X\boxtimes\uk_{C^I}) \\
&\cong\Hom^0(T_{c^*V_{\lambda_I}}\langle -d_{\lambda_I}+2d_{\lambda_{H,I}} \rangle( \cP_X\boxtimes\uk_{C^I}), \cP_X\boxtimes\uk_{C^I})
\end{split}\] where $c^*$ is the pull-back along the Cartan involution $c:\Gc^I\to\Gc^I$.

Composing with the natural isomorphism $\Frob^*\cP_X\isom \cP_X$, one gets a cohomological correspondence \begin{equation}\label{eq:intro:minhomocc}\begin{split}\frc_{\lambda_{H,I}}^{(1)}& \in \Hom^0(\rh_{I}^*(\Frob\times\id)^*(\cP_X\boxtimes\uk_{C^I})\langle 2d_{\lambda_{H,I}} \rangle,\lh_{I}^!(\cP_X\boxtimes\uk_{C^I}) )\\
&\cong\Hom^0(T_{c^*V_{\lambda_I}}\langle -d_{\lambda_I}+2d_{\lambda_{H,I}} \rangle\circ(\Frob\times\id)^* (\cP_X\boxtimes\uk_{C^I}), \cP_X\boxtimes\uk_{C^I})
\end{split}.\end{equation}

By projection formula, we have natural isomorphisms \[l_{I,!}(T_{V^I}(-)\otimes-)\cong l_{I,!}(-\otimes T_{c^*V^I}(-))\] and \[l_{I,!}(\Frob^*(-)\otimes-)\cong l_{I,!}(-\otimes\Frob_!(-))\] where $l_I:\Bun_G\times C^I\to C^I$. Since $\int_{X,\Nilp,I}=l_{I,!}(-\otimes(\cP_X\boxtimes\uk_{C^I}))$, the cohomological correspondence $\frc_{\lambda_{H,I}}^{(1)}$ gives us a natural transformation \[\eta^{(1)}_{\frc_{\lambda_{H,I}},\Nilp}:\int_{X,\Nilp,I}\circ(\Frob\times\id)_!\circ T_{V_{\lambda_I}}\langle -d_{\lambda_I}+2d_{\lambda_{H,I}}\rangle\to \int_{X,\Nilp,I}.\]

The main result in this part is the following:
\begin{thm}\label{thm:minmain}
    Assuming Assumption \ref{assumption:goodfil}, the right square of \eqref{diag:intro:funddiagtr} is commutative for homogeneous minuscule special cycle classes. That is, we have a commutative square
    \begin{equation}
    \adjustbox{scale=.95,center}{%
        \begin{tikzcd}
            \tr_{\QLisse(C^I)}((\Frob\times\id)_!\circ T_{V_{\lambda_I}}\langle -d_{\lambda_I}+2d_{\lambda_{H,I}}\rangle,\Shv_{\Nilp}(\Bun_G)\otimes\QLisse(C^I)) \ar[d, "\LT^{\Serre}", "\sim"'] \ar[r, "\tr_{\QLisse(C^I)}(\eta^{(1)}_{\frc_{\lambda_{H,I}},\Nilp})"] & \tr_{\QLisse(C^I)}(\id,\QLisse(C^I))   \ar[d, "\sim"] \\
            l_{I,!}(\IC_{V^I}|_{\Sht_{G,I}})\langle -d_{\lambda_I}+ 2d_{\lambda_{H,I}}\rangle \ar[r, "\pi_{\Sht,I,!}\lbrack \Sht_{H,\lambda_{H,I}}/C^I\rbrack"]  & \uk_{C^I}
        \end{tikzcd}
        }
    \end{equation}

\end{thm}

\subsubsection{Diagonal cycle classes} \label{sec:intro:diagcycleclass}
When $G=H\times H$ and $H\sub G$ is the diagonal embedding. Take $X=H\backslash H\times H$. In this section, we consider special cycle classes for possibly non-minuscule modification types. Note that the construction of the natural transformation $\eta_{\frc_{\lambda_{H,I}},\Nilp}^{(1)}$ from the cohomological correspondence $\frc_{\lambda_{H,I}}$ makes sense without assuming that $\lambda_{H,I}$ is minuscule. Only when defining $\frc_{\lambda_{H,I}}$ and $[Z_{V^I}^X]$, the steps involving fundamental classes need a clarification. 

The diagram \eqref{diag:mincorfund} becomes
\begin{equation}\label{diag:diagcorfund}\begin{tikzcd}
    \Bun_H\times C^I \ar[d, "\D\times\id"] & \Hk_{H,\lambda_{H,I}} \ar[l, "\lh_{H,I}"'] \ar[r, "\rh_{H,I}"] \ar[d,"\D_{\Hk,I}"] & \Bun_H\times C^I \ar[d, "\D\times\id"] \\
    \Bun_G\times C^I & \Hk_{G,\lambda_{I}} \ar[l, "\lh_{I}"'] \ar[r, "\rh_I"] & \Bun_G\times C^I
\end{tikzcd}\end{equation} where $\D:\Bun_H\to\Bun_G$ is the diagonal map.

In this case, we consider relative fundamental class \[[\Sht_{H,\lambda_{H,I}}/C^I]\in H_{0}^{BM}(\Sht_{H,\lambda_{H,I}}/C^I,\IC_{V_{\lambda_{H,I}}}^{\otimes 2}|_{\Sht_{H,\lambda_{H,I}}})=\Hom^0(\IC_{V_{\lambda_{H,I}}}^{\otimes 2}|_{\Sht_{H,\lambda_{H,I}}},\om_{\Sht_{H,\lambda_{H,I}}/C^I})\] given by the canonical map \[\IC_{V_{\lambda_{H,I}}}^{\otimes 2}|_{\Sht_{H,\lambda_{H,I}}}\cong \IC_{V_{\lambda_{H,I}}}|_{\Sht_{H,\lambda_{H,I}}}\otimes \DD_{C^I}(\IC_{V_{\lambda_{H,I}}}|_{\Sht_{H,\lambda_{H,I}}})\to \om_{\Sht_{H,\lambda_{H,I}}/C^I}.\] Here we are using $\DD_{C^I}:=\underline{\Hom}(-,\om_{\Sht_{H,\lambda_{H,I}}/C^I}):\Shv(\Sht_{H,\lambda_{H,I}})\to \Shv(\Sht_{H,\lambda_{H,I}})$ to denote the (relative) Verdier duality over $C^I$ and $\om_{\Sht_{H,\lambda_{H,I}}/C^I}$ to denote the (relative) dualizing sheaf. A similar construction gives the relative fundamental class \[[\Hk_{H,\lambda_{H,I}}/\Bun_H\times C^I]\in H_{0}^{BM}(\Hk_{H,\lambda_{H,I}}/\Bun_H\times C^I,\IC_{V_{\lambda_{H,I}}}^{\otimes 2}).\] We call $\frc_{\lambda_{H,I}}=\D_{\Hk,I,!}[\Hk_{H,\lambda_{H,I}}/\Bun_H\times C^I]$ the \emph{diagonal cohomological correspondence}.

In this case, the element \begin{equation}\label{eq:intro:diagcycle}\langle-,-\rangle_{\lambda_{H,I}}:=\D_{\Sht,I,!}[\Sht_{H,\lambda_{H,I}}/C^I]:(l_{H,I,!}(\IC_{V_{\lambda_{H,I}}}|_{\Sht_{H,\lambda_{H,I}}}))^{\otimes 2}\to \uk_{C^I}\end{equation} is called the \emph{diagonal cycle class} and has a significant meaning: It is the intersection pairing on the compact support cohomology of Shtukas. The functor \[\ev:=\int_{X,\Nilp}:\Shv_{\Nilp}(\Bun_G)^{\otimes 2}\to\Vect\] also has a significant meaning: It is the counit for the miraculous duality of $\Shv_{\Nilp}(\Bun_G)$.

We have the following parallel result of Theorem  \ref{thm:minmain}:

\begin{thm}\label{thm:diagmain}
    The right square of \eqref{diag:intro:funddiagtr} is commutative for diagonal special cycle classes. That is, we have a commutative square
    \begin{equation}
        \begin{tikzcd}
            \tr_{\QLisse(C^I)}((\Frob\times\id)_!\circ T_{V^I},\Shv_{\Nilp}(\Bun_G)\otimes\QLisse(C^I)) \ar[d, "\LT^{\Serre}", "\sim"'] \ar[r, "\tr_{\QLisse(C^I)}(\eta^{(1)}_{\frc_{\lambda_{H,I}},\Nilp})"] & \tr_{\QLisse(C^I)}(\id,\QLisse(C^I))   \ar[d, "\sim"] \\
            (l_{H,I,!}(\IC_{V_{\lambda_{H,I}}}|_{\Sht_{H,\lambda_{H,I}}}))^{\otimes 2} \ar[r, "{\langle-,-\rangle_{\lambda_{H,I}}}"]  & \uk_{C^I}
        \end{tikzcd}
    \end{equation}
\end{thm}

\subsubsection{Strategy of proof}
The proof of Theorem \ref{thm:minmain} and Theorem \ref{thm:diagmain} will be given in \S\ref{sec:ss=catproof}. One may think Theorem  \ref{thm:minmain} and Theorem  \ref{thm:diagmain} are purely formal consequences of the six-functor formalism given the deep result of \cite{arinkin2022automorphicfunctionstracefrobenius}. While we do not regard them as deep results themselves, there are some seemingly necessary and non-trivial ingredients involved in the proof.

We only explain the proof of Theorem  \ref{thm:minmain} since the proof of Theorem  \ref{thm:diagmain} is similar after the development of the theory of cohomological correspondences with kernels, which is the main subject of \S\ref{sec:cccor}.

We prove Theorem  \ref{thm:minmain} by establishing a two-step equality \[\tr_{\QLisse(C^I)}(\eta^{(1)}_{\frc_{\lambda_{H,I}},\Nilp})= \tr_{\Sht,C^I}(\frc_{\lambda_{H,I}}) =\pi_{\Sht,I,!}[\Sht_{H,\lambda_{H,I}}/C^I].\] Here, the element $\tr_{\Sht,C^I}(\frc_{\lambda_{H,I}})\in H_{-d_{\lambda_I}+2 d_{\lambda_{H,I}}}^{BM}(\Sht_{G,{\lambda_I}}/C^I,\IC_{V^I}|_{\Sht_{G,I}})$ is the geometric trace of the cohomological correspondence $\frc^{(1)}_{\lambda_{H,I}}$. The notion of geometric trace of cohomological correspondence \footnote{In literature, it is also called sheaf-cycle correspondence.} has been widely used in literature. An incomplete list of its usage includes \cite{FYZ3}\cite{YZ1}\cite{varshavsky2007lefschetz}\cite{gaitsgory2024localtermscategoricaltrace}. It can be regarded as a generalization of the sheaf-function correspondence, in which case the correspondence is taken to be the Frobenius morphism.

The first equality follows from a direct diagram chasing and is completely formal given \cite{arinkin2022automorphicfunctionstracefrobenius}. Its proof will be provided in \S\ref{sec:cat=geo}.

The second equality is more subtle: It involves compatibility between the push-forward of cohomological correspondence and the geometric trace construction. Its proof will be given in \S\ref{sec:geo=cycle}. Note that such compatibility is proved in \cite{varshavsky2007lefschetz}\cite{Lu_2022} for push-forward along proper maps. The problem here is that while the map $\pi_{\Sht,I}:\Sht_{H,\lambda_{H,I}}\to\Sht_{G,\lambda_I}$ is proper, the map $\pi:\Bun_H\to\Bun_G$ (or more essentially the map $\pi_{\Hk,I}:\Hk_{H,\lambda_{H,I}}\to\Hk_{G,\lambda_I}$) is not proper. We remedy this by showing that the map $\pi:\Bun_H\to\Bun_G$ (hence the map $\pi_{\Hk,I}:\Hk_{H,\lambda_{H,I}}\to\Hk_{G,\lambda_I}$) admits a compactification such that the Hecke-Frobenius correspondence is \emph{contracting along the boundary}. The idea of studying contracting correspondence in the theory of cohomological correspondence is not new: It appears in \cite{varshavsky2007lefschetz} and is also used in \cite{gaitsgory2024localtermscategoricaltrace}\cite{feng2024modularityhigherthetaseries}. In these articles, they show that \emph{pull-back} of cohomological correspondence along contracting substacks is compatible with the trace construction. In contrast, we show that \emph{push-forward} of cohomological correspondence along a map admitting contracting boundary is compatible with trace construction, which seems to be new as far as the author knows. This is Theorem \ref{thm:trcontractpush}.

The ``compactification" of $\pi:\Bun_H\to\Bun_G$ is a map $\overline{\pi}:\overline{\Bun}_G^X\to\Bun_G$ coming from the affine degeneration of the spherical variety $X=H\backslash G$. See \S\ref{sec:relativecompactification} for a discussion. This is where the sphericity of $H\sub G$ and Assumption \ref{assumption:goodfil} are used. In the diagonal case $G=H\times H$, this compactification is the famous Drinfeld's compactification $\overline{\D}:\overline{\Bun}_G\to\Bun_G\times\Bun_G$ coming from Vinberg's semigroup. A proof of that such construction gives a compactification (i.e., $\overline{\D}$ is proper when $G$ is semisimple)  is contained in \cite[Appendix\,A]{finkelberg2020drinfeld}. For a general homogeneous spherical variety $X=H\backslash G$, the existence of such a compactification seems well-known to the experts, but we cannot find appropriate literature. We prove this by simplifying and generalizing the argument in \emph{loc.cit}, which also yields a shorter proof in the diagonal case. This is Proposition \ref{prop:relcompact}.

\subsection{Organization of the article}
The organization of this article is as follows: \begin{itemize}
\item In \S\ref{sec:prelim}, we introduce the preliminaries for this article. In particular, we develop the theory of cohomological correspondence with a kernel that appears ubiquitously in this work.
\item In \S\ref{sec:geo=cycle}, we prove the identity between the special cycle classes and the geometric trace of the cohomological correspondences.
\item In \S\ref{sec:cat=geo}, we prove the identity between the geometric trace of cohomological correspondences and the classes obtained from the functoriality of categorical trace.
\item In \S\ref{sec:isotypicpartsc}, we define the $\s$-isotypic part in the cohomology of Shtukas and develop general tools to study the isotypic part of special cycle classes.
\item In \S\ref{sec:rs}, we apply the machinery developed in the previous sections to the Rankin--Selberg case and prove the main result, Theorem \ref{thm:intro:main}.
\end{itemize}

\subsection{Notations and conventions}\label{sec:intro:not}
We now introduce the commonly used notations and conventions in this article.
\subsubsection{Category theory}\label{sec:intro:not:cat}
In this article, by a vector space, we mean a super vector space. By a category, we usually mean an $(\infty,1)$-category. We use $\Space$ to denote the $\infty$-category of spaces and $\Vect$ to indicate the $\infty$-category of (super) vector spaces. For a category $\cC$ with a $t$-structure, we use $\cC^{\heartsuit}$ to denote its heart, which is an abelian category.

For $V\in\Vect$, we define $V\langle 1\rangle=\Pi V[1](1/2)$ where \begin{itemize} \item $\Pi:\Vect\to\Vect $ is the functor changing the parity; \item $[1]:\Vect\to\Vect$ is the shift functor; \item $(1/2)$ is the Tate twist whenever it makes sense (e.g, when working with Frobenius equivariant vector spaces).  \end{itemize}

For a category $\cC$ and two objects $x,y\in\cC$, we use $\Hom(x,y)\in \Space$ to denote the mapping space. We define $\Hom^0(x,y)=\pi_0(\Hom(x,y))\in\Vect$. 

We say a functor between two categories $L:\cC\to\cD$ is continuous if $L$ preserves colimits.

\subsubsection{Sheaf theory}\label{sec:intro:not:sheaf}
In this article, we work with algebraic stacks, which are Artin stacks locally of finite type over a field $F$. For a algebraic stack $X$ over $F$, we use $\Shv(X)=\Shv(X_{\overline{F}})$ to denote the category of (ind-constructible) \'etale sheaves with $k=\Qlbar$-coefficient as introduced in \cite[Appendix\,F]{arinkin2022stacklocalsystemsrestricted}. We use $\Shv(X)_c\sub\Shv(X)$ to denote the full subcategory of constructible sheaves.

\subsubsection{Geometric notations}\label{sec:intro:not:geometricsetup}
In this article, by a $G$-variety $X$, we mean a variety $X$ with a \emph{right} $G$-action.

We fix a smooth projective curve $C$ over $\FF_q$. We use $\Bun_G=\Map(C,[*/G])$ to denote the moduli space of $G$-bundles over $C$. For a smooth affine $G$-variety $X$, we define $\Bun_G^X:=\Map(C,[X/G])$. It is equipped with a map $\pi:\Bun_G^X\to\Bun_G$ which is of finite type. We define the \emph{period sheaf} to be $\cP_X:=\pi_!\uk_{\Bun_G^X}$.

We use $I=\{1,2,\cdots,r\}$ to denote a finite set that serves as the index set of legs. We use $\Hk_{G,I}$ to denote the (iterated) Hecke stack of $G$ with $I$-legs, that is, the moduli space of tuples \[( (c_i)_{i\in I}, \cE_0\xdashrightarrow{c_1}\cE_1 \xdashrightarrow{ c_2}\cdots \xdashrightarrow{c_r}\cE_r)\] where \begin{itemize}
    \item $c_i\in C$ for $i\in I$;
    \item $\cE_{i-1}\xdashrightarrow{c_{i}}\cE_i $ is a map between $G$-bundles on $C-\{c_i\}$ for $i\in I$.
\end{itemize} It gives a correspondence \[\Bun_G\times C^I \xleftarrow{\lh_I} \Hk_{G,I} \xrightarrow{\rh_I} \Bun_G\times C^I\] where $\lh_I$ sends data above to $(\cE_0,(c_i)_{i\in I})$ and $\rh_I$ sends data above to $(\cE_r,(c_i)_{i\in I})$. We use $l_{I}:\Hk_{G,I}\to C^I$ to denote the map remembering the legs (sending data above to $(c_i)_{i\in I}$).

The geometric Satake equivalence attaches each $V^I\in\Rep(\Gc^I)$ a sheaf $\IC_{V^I}\in\Shv(\Hk_G^I)$, which is normalized such that when $V^I=V_{\lambda_I}\in\Rep(\Gc^I)^{\heartsuit}$ which is irreducible with highest weight $\lambda_I=(\lambda_1,\cdots,\lambda_r)\in X_*(T)^I$, the sheaf $\IC_{V^I}$ is perverse, pure of weight $0$, and has the same parity as $\sum_{i\in I}\langle 2\rho,\lambda_i\rangle$. We use $\Hk_{G,\lambda_I}$ (or $\Hk_{G,V_{\lambda_I}}$) to denote the support of $\IC_{V^I}$ (known as the closed Schubert cell). We use $\Hk_{G,\lambda_I}^{\circ}\sub \Hk_{G,\lambda_I}$ to denote the open Schubert cell.

The Hecke operator attached to $V^I\in\Rep(\Gc^I)$ is defined to be the functor \begin{equation}\label{eq:intro:not:hecke} T_{V^I}=\rh_{I,*}(\lh_I^!(-)\otimes^{!} \IC_{V^I})\cong\rh_{I,!}(\lh_I^*(-)\otimes\IC_{V^I}):\Shv(\Bun_G\times C^I)\to\Shv(\Bun_G\times C^I).\end{equation}

We define the moduli space of $G$-Shtukas with $I$-legs to be the stack defined by the fiber product \[\begin{tikzcd}\Sht_{G,I} \ar[r] \ar[d] & \Hk_{G,I} \ar[d, "{(\lh_I,\Frob\circ \rh_I,l_I)}"] \\ \Bun_G\times C^I  \ar[r, "\D_{\Bun_G}\times\id"] & \Bun_G\times\Bun_G\times C^I
\end{tikzcd}.\] We still use $l_{I}:\Sht_{G,I}\to C^I$ to denote the map remembering the legs. We also have the obvious notion $\Sht_{G,\lambda_I}\sub \Sht_{G,I}$ (or $\Sht_{G,V_{\lambda_I}}\sub\Sht_{G,I}$).

We have the moduli space $\Hk_{G,I}^X$ by adding a rational $X$-section to the data of $\Hk_{G,I}$ that is regular for each $\cE_i$. Similarly, we have moduli spaces $\Sht_{G,I}^X$, $\Hk_{G,\lambda_I}^X$, and $\Sht_{G,\lambda_I}^X$.

\subsection{Acknowledgment}
The author would like to express his sincere gratitude to his supervisor, Zhiwei Yun, for his invaluable guidance, insightful feedback, and constant support throughout the writing process of this article. This work would never have come out without his help. This work originates from the collaborative work \cite{liu2025higherperiodintegralsderivatives} of the author and Shurui Liu. The author would like to thank Shurui Liu for many inspiring discussions, especially drawing the author's attention to \cite[Corollary\,3.3.7]{arinkin2022dualityautomorphicsheavesnilpotent}, which is one of the starting points of this work. Also, the author would like to thank Tsao-Hsien Chen, Yiannis Sakellaridis, Jinfeng Song, Yakov Varshavsky, and Lingfei Yi for helpful discussions.

\section{Preliminaries}\label{sec:prelim}
In this section, we introduce the preliminaries for the article.
\begin{itemize}
    \item In \S\ref{sec:cattrace}, we recall the formalism of categorical trace.
    \item In \S\ref{sec:cccor}, we develop the theory of cohomological correspondences with kernel.
    \item In \S\ref{sec:relativecompactification}, we explain the compactification of the map $\pi:\Bun_G^X\to\Bun_G$.
\end{itemize}

\subsection{Categorical trace}\label{sec:cattrace}
In this section, we recall the formalism of categorical trace. We refer to \cite[\S7.3]{zhu2025tamecategoricallocallanglands} for a more detailed treatment.

In the following, by a category, we mean either 1-category, 2-category, $(\infty,1)$-category, or $(\infty,2)$-category.

\begin{defn}[Categorical trace]\label{def:cattr}
    Suppose $(\cC,\otimes,1_{\cC})$ is a symmetric monoidal category. For each dualizable object $x\in\cC$ and an endomorphism $F\in\End_{\cC}(x)$, we define the \emph{categorical trace} of $F$ to be the element \[\tr(F,x)\in\End_{\cC}(1_{\cC})\] defined as the composition \[\begin{tikzcd}
        1_{\cC} \ar[r, "\rmu_x"] & x\otimes x^{\vee} \ar[r,"F\otimes\id"] & x \otimes x^{\vee} \ar[r, "\ev_x"] & 1_{\cC} 
    \end{tikzcd}.\] Here, $x^{\vee}$ is the dual of $x$, $\rmu_x:1_{\cC}\to x\otimes x^{\vee}$ and $\ev_x:x\otimes x^{\vee}\to 1_{\cC}$ are the unit and counit map for the duality between $x$ and $x^{\vee}$.
\end{defn}

When $\cC$ is a 2-category or $(\infty,2)$-category, the categorical trace admits the following functorial property.

\begin{defn}[Functoriality of trace]\label{def:funccattr}
    Suppose $(\cC,\otimes,1_{\cC})$ is a symmetric monoidal 2-category (or $(\infty,2)$-category). For any objects $x,y\in\cC$ and an adjoint pair of morphisms \[\begin{tikzcd}
        x \ar[r, shift left=.75ex, "L"] & y \ar[l, shift left=.75ex, "R"]
    \end{tikzcd},\] we use $R^{\vee}:x^{\vee}\to y^{\vee}$ to denote the map dual to $R$ given by \[R^{\vee}:x^{\vee}\xrightarrow{\id\otimes \rmu_{y}}x^{\vee}\otimes y\otimes y^{\vee}\xrightarrow{\id\otimes R\otimes\id}x^{\vee}\otimes x\otimes y^{\vee}\xrightarrow{\ev_x\otimes\id} y^{\vee}.\] Suppose we are given $F\in\End_{\cC}(x)$, $G\in\End_{\cC}(y)$, and a 2-morphism $\eta:L\circ F\to G\circ L$, we define the 2-morphism \[\tr(\eta):\tr(F,x)\to \tr(G,y)\] to be the 2-morphism from the upper route to the lower route in the diagram \begin{equation}\label{diag:trfuncfunddiag}\begin{tikzcd}
        1_{\cC} \ar[r, "\rmu_x"] \ar[d, "\id"] & x\otimes x^{\vee} \ar[r,"F\otimes\id"] \ar[d, "L\otimes R^{\vee}"] \ar[dl, Rightarrow, "\gamma"] & x \otimes x^{\vee} \ar[r, "\ev_x"] \ar[d, "L\otimes R^{\vee}"] \ar[dl, Rightarrow, "\eta\otimes\id"] & 1_{\cC} \ar[d, "\id"] \ar[dl, Rightarrow, "\delta"] \\
        1_{\cC} \ar[r, "\rmu_y"] & y\otimes y^{\vee} \ar[r,"F\otimes\id"] & y \otimes y^{\vee} \ar[r, "\ev_y"] & 1_{\cC} 
    \end{tikzcd}.\end{equation} Here the 2-morphism $\gamma$ is defined as \[\gamma: (L\otimes R^{\vee})\circ \rmu_x = (L\otimes\ev_x\otimes\id)\circ (\rmu_x\otimes R\otimes\id)\circ \rmu_y\cong ((L\circ R)\otimes\id)\circ \rmu_y\to\rmu_y.\] The 2-morphism $\delta$ is defined as \[\delta:\ev_x\to \ev_x\circ ((R\circ L)\otimes \id)\cong \ev_x\circ (R\otimes \ev_y\otimes \id)\circ (\rmu_y\otimes L\otimes\id))=\ev_y\circ (L\otimes R^{\vee}).\]
\end{defn}

\begin{example}\label{eg:prl}
    Consider the symmetric monoidal $(\infty,2)$-category $(\Pr^{L},\otimes,\mathrm{Space})$ consisting of presentable $(\infty,1)$-categories with continuous (i.e. colimit preserving) functors. Here $\mathrm{Space}\in\Pr^L$ is the category of spaces. The construction above assigns to each dualizable presentable $(\infty,1)$-category $\cC\in \Pr^{L}$ together with a continuous functor $F:\cC\to\cC$ a space $\tr(F,\cC)\in \End_{\Pr^L}(\mathrm{Space})\cong\mathrm{Space}$. Moreover, given another $\cD\in\Pr^L$ equipped with a continuous functor $G:\cD\to\cD$, suppose one has a continuous functor $L:\cC\to\cD$ preserving compact objects (i.e. admits right adjoint $R:\cD\to\cC$ in $\Pr^L$), for each natural transformation $\eta:L\circ F\to G\circ L$, one gets a map between spaces $\tr(\eta):\tr(F,\cC)\to\tr(G,\cD)$.
\end{example}

\begin{example}\label{eg:lincat}
    For each algebra object $(\cA,\otimes)\in\Pr^L$ (i.e. a presentable symmetric monoidal $(\infty,1)$-category such that $\otimes$ is continuous), one can consider the $(\infty,2)$-category of linear categories over $\cA$ defined as $\LinCat_{\cA}:=\mathrm{LMod}_{\cA}(\Pr^L)$. This defines a symmetric monoidal $(\infty,2)$-category $(\LinCat_{\cA},\otimes_{\cA},\cA)$. The discussion in Example \ref{eg:prl} carries over by requiring linearity over $\cA$ everywhere, and we obtain $\tr_{\cA}(F,\cC)\in \End_{\LinCat_{\cA}}(\cA)\cong \cA$ and map $\tr_{\cA}(\eta):\tr_{\cA}(F,\cC)\to\tr_{\cA}(G,\cD)$.
\end{example}

\begin{example}\label{eg:lincatrigid}
    Continuing with Example \ref{eg:lincat}, suppose $\cA$ is rigid. Then $\cC\in\LinCat_{\cA}$ is dualizable if and only if $\cC$ is dualizable as an element in $\Pr^L$. The most important cases for us are when $\cA=\Vect_k$ for a field $k$, and $\cA=\QLisse(C^I)$ for a curve $C$ over $k$ and a finite set $I$. Here, the category $\QLisse(C^I)$ is defined in \cite[Definition\,1.2.6]{arinkin2022stacklocalsystemsrestricted}.
\end{example}

\subsection{Cohomological correspondences with kernel}\label{sec:cccor}
In this section, we develop the theory of cohomological correspondences with kernel. In this section, by an algebraic stack, we mean an Artin stack locally of finite type over a field $F$.

\subsubsection{Definition}

Consider a correspondence of algebraic stacks \[A_1 \xleftarrow{c_1}  C\xrightarrow{c_0}  A_0.\] Given $\cK\in\Shv(C)$.

\begin{defn}
    For any $\cF_0\in\Shv(A_0)$ and $\cF_1\in\Shv(A_1)$, a \emph{cohomological correspondence} between $\cF_0$, $\cF_1$ with kernel $\cK$ is an element \[\frc\in\Cor_{C,\cK}(\cF_0,\cF_1):=\Hom^0(c_0^*\cF_0\otimes\cK,c_1^!\cF_1).\]
\end{defn}

\begin{remark}
    When $\cK=\uk_{C}\langle -d\rangle$ for $d\in \ZZ$, this is the usual notion of cohomological correspondence of degree $d$.
\end{remark}

\subsubsection{Push-forward}\label{sec:ccpush}
Consider a map of correspondences
\begin{equation}\label{diag:functdiag}
\begin{tikzcd}
    A_1\ar[d, "f_{A_1}" '] & C \ar[l,"c_1" ']\ar[r,"c_0"]\ar[d,"f"] & A_0 \ar[d, "f_{A_0}"] \\
    B_1& D \ar[l, "d_1" ']\ar[r, "d_0"]& B_0
\end{tikzcd}.
\end{equation}

Suppose the right square of \eqref{diag:functdiag} is pushable (i.e., the map $C\to A_0\times_{B_0} D$ is proper). Given $\cF_i\in\Shv(A_i)$ for $i=0,1$, $\cK\in\Shv(C)$ and $\cL\in\Shv(D)$ together with a map $\alpha:f^*\cL\to\cK$, one can define the \emph{push-forward} of cohomological correspondence which is a map \[f_!:\Cor_{C,\cK}(\cF_0,\cF_1)\to\Cor_{D,\cL}(f_{A_0,!}\cF_0,f_{A_1,!}\cF_1)\] defined such that for any $\frc\in \Cor_{C,\cK}(\cF_0,\cF_1)=\Hom^0(c_0^*\cF_0\otimes\cK,c_1^!\cF_1)$, we have \[f_!\frc: d_0^*f_{A_0,!}\cF_0\otimes\cL \to f_!c_0^*\cF_0 \otimes \cL\isom f_!(c_0^*\cF_0\otimes f^*\cL)\to f_!(c_0^*\cF_0\otimes\cK)\to f_!c_1^!\cF_1\to d_1^!f_{A_1,!}\cF_1.\] Here, the first map uses the base change map for pushable square (see \cite[\S3.2]{FYZ3}), the second map uses projection formula for $f$, the third map uses $\alpha:f^*\cL\to\cK$, the fourth map uses the map $\frc$, the final map is the Beck-Chevalley base change map for the left square of \eqref{diag:functdiag}.

\begin{example}[Push-forward Borel-Moore homology class]\label{eg:pushbm}
    Consider the map of correspondences
\begin{equation}\label{diag:bmfund}
\begin{tikzcd}
    S\ar[d, "\id" '] & C \ar[l,"c" ']\ar[r,"c"]\ar[d,"f"] & S \ar[d, "\id"] \\
    S& D \ar[l, "d" ']\ar[r, "d"]& S
\end{tikzcd}. 
\end{equation}
The condition that the right square is pushable is equivalent to requiring $f$ to be proper. Note that $\Cor_{C,\cK}(\uk_S,\uk_S)=H_0^{BM}(C/S,\cK)=\Hom^0(\cK,\om_{C/S})$ and a similar statement holds for $D$. The push-forward of cohomological correspondence gives for each $\alpha:f^*\cL\to\cK$ a map \[f_!:H_0^{BM}(C/S,\cK)\to H_0^{BM}(D/S,\cL).\] In particular, when $\cK=\uk_{C}$, $\cL=\uk_D$, and $\alpha$ is the natural map $f^*\uk_{D}\to\uk_C$, this gives the usual proper push-forward of Borel-Moore homology class.

\end{example}

\subsubsection{Pull-back} \label{sec:ccpull}
In this section, we switch to derived algebraic geometry and consider derived algebraic stacks. For applications in this article, one can safely work within classical algebraic geometry and pretend all the quasi-smooth maps involved are smooth. Consider the same diagram \eqref{diag:functdiag} but assume the left square is pullable of defect $d$ (i.e., the map $C\to A_1\times_{B_1} D$ is quasi-smooth of relative dimension $d$). Given $\cF_i\in\Shv(B_i)$ for $i=0,1$, $\cK\in\Shv(D)$ and $\cL\in \Shv(C)$ together with a map $\alpha:\cL\langle -2d\rangle\to f^*\cK$, one can define the \emph{pull-back} of cohomological correspondence which is a map \[f^*:\Cor_{D,\cK}(\cF_0,\cF_1)\to\Cor_{C,\cL}(f^*_{A_0}\cF_0,f^*_{A_1}\cF_1)\] defined such that for any $\frc\in \Cor_{D,\cK}(\cF_0,\cF_1)=\Hom^0(c_0^*\cF_0\otimes\cK,c_1^!\cF_1)$, we have \[f^*\frc:(c_0^*f_{A_0}^*\cF_0)\otimes\cL\to f^*d_0^*\cF_0\otimes f^*\cK\langle 2d\rangle\isom f^*(d_0^*\cF_0\otimes\cK)\langle 2d\rangle\to f^*d_1^!\cF_1\langle 2d\rangle\to c_1^!f_{A_1}^*\cF_1.\] Here, the first maps uses $\alpha$, the third map uses $\frc$, the last map uses the base change map $f^*d_1^!\langle2d\rangle\to c_1^!f_{A_1}^*$ for pull-back square of defect $d$ (see \cite[\S3.5]{FYZ3}).

\begin{example}[Pull-back of Borel-Moore homology class]\label{eg:pullbm}
    Consider diagram \eqref{diag:bmfund}. The condition that the left square is pullable of defect $d$ is equivalent to requiring $f$ to be quasi-smooth of relative dimension $d$. The pull-back of cohomological correspondence gives for each $\alpha:\cL\langle-2d\rangle\to f^*\cK$ a map \[f^*:H_0^{BM}(D/S,\cK)\to H_{0}^{BM}(C/S,\cL).\] When $\cK=\uk_{D}$, $\cL=\uk_C\langle 2d\rangle$, and $\alpha$ is the natural map $\uk_C\to f^*\uk_{D}$, this gives the usual pull-back of Borel-Moore homology class along a quasi-smooth map.
\end{example}

\subsubsection{Specialization}\label{sec:ccspecial}
In this section, we use $\eta$ to denote the generic point of $\AA^1_{F}$ and $s=0\in\AA^1_F(F)$.
For any algebraic stack $A$ over $\AA^1_F$, one has the nearby cycle functor $\Psi:\Shv(A_{\eta})\to\Shv(A_s)$.

Consider a cohomological correspondence over $\AA^1_F$:  \[A_1\xleftarrow{c_1} C \xrightarrow{c_0} A_0.\] Given $\cF_i\in \Shv(A_{i,\eta})$ for $i=0,1$, $\cK\in\Shv(C_{\eta})$, $\cL\in\Shv(C_s)$ together with a morphism $\alpha:\cL\to\Psi\cK$, one can define the \emph{specialization} of cohomological correspondence to be the map \[\Psi:\Cor_{C_{\eta},\cK}(\cF_0,\cF_1)\to \Cor_{C_s,\cL}(\Psi\cF_0,\Psi\cF_1)\] defined such that for any $\frc\in\Cor_{C_{\eta},\cK}(\cF_0,\cF_1)$ one has \[\Psi\frc:(c_{0,s}^*\Psi\cF_0)\otimes \cL\to\Psi (c_{0,\eta}^*\cF_0)\otimes\Psi\cK\to \Psi(c_{0,\eta}^*\cF\otimes\cK)\xrightarrow{\Psi(\frc)}\Psi c_{1,\eta}^!\cF_1\to c_{1,s}^!\Psi\cF_1.\]

\begin{example}[Specialization of Borel-Moore homology class]\label{eg:spbm}
    Suppose the correspondence has the form \[S\xleftarrow{c} C \xrightarrow{c} S.\] The construction above defines for each $\alpha:\cL\to\Psi\cK$ a map \[\Psi:H_0^{BM}(C_{\eta}/S_{\eta},\cK)\to H_0^{BM}(C_s/S_s,\cL),\] which coincides with the usual specialization of Borel-Moore homology class when $\cK=\uk_{C_{\eta}}$, $\cL=\uk_{C_s}$, and $\alpha:\Psi\uk_{C_{\eta}}\to\uk_{C_s}$ is the natural map.
\end{example}

\subsubsection{Geometric trace construction}\label{sec:cctrace}
In the following, we work with algebraic stacks over a base algebraic stack $S$. For any such stack $A$, we define the category $\Shv(A)_{\ULA}\sub\Shv(A)$ to be the full subcategory of sheaves that are ULA over $S$. We use $\DD_S:\Shv(A)_{\ULA}\to \Shv(A)_{\ULA}^{\op}$ to denote the relative Verdier duality functor.

Consider a self-correspondence between algebraic stacks $(A\xleftarrow{c_1}C\xrightarrow{c_0}A)$ with a kernel $\cK\in\Shv(C)$. Given $\cF\in\Shv(A)_{\ULA}$ and a cohomological correspondence $\frc\in\Cor_{C,\cK}(\cF,\cF)=\Hom^0(c_0^*\cF\otimes\cK,c_1^!\cF)$. Consider Cartesian diagram \[ \begin{tikzcd}
    \Fix(C) \ar[d, "p"] \ar[r, "i"] & C \ar[d, "{(c_1,\,c_0)}"] \\
    A \ar[r, "\D"] & A\times_S A
\end{tikzcd}.\]

\begin{defn}
    We define the \emph{geometric trace} of cohomological correspondence $\frc$ to be the element \[\tr_S(\frc)\in H_0^{BM}(\Fix(C)/S,\cK|_{\Fix(C)})=\Hom^0(\cK|_{\Fix(C)},\om_{\Fix(C)/S})\] given by the composition \[\begin{split}
        \cK|_{\Fix(C)}&=i^*\cK \\ &\xrightarrow{i^*\frc} i^*\underline{\Hom}(c_0^*\cF,c_1^!\cF) \\ &\isom i^*(c_1,c_0)^!\underline{\Hom}(\pr_2^*\cF,\pr_1^!\cF)\\ &\to p^!\D^*(\DD_S(\cF)\boxtimes_S\cF) \\ &\isom p^!(\DD_S(\cF)\otimes\cF) \\ &\to p^!\om_{A/S} \\ &\isom \om_{\Fix(C)/S}
    \end{split}.\] The third row uses isomorphism $\underline{\Hom}(c_0^*\cF,c_1^!\cF)\cong (c_1,c_0)^!\underline{\Hom}(\pr_2^*\cF,\pr_1^!\cF)$, the fourth row uses base change map $i^*(c_1\times c_0)^!\to p^!\D^*$ and the natural isomorphism $\underline{\Hom}(\pr_2^*\cF,\pr_1^!\cF)\cong \DD_S(\cF)\boxtimes_S\cF$, the sixth line uses $\DD_S(\cF)\otimes\cF\to\om_{A/S}$.
\end{defn}

\begin{remark}
    When $\cK=\uk_{C}$, the definition above coincides with the definition in \cite[\S4.1]{FYZ3}.
\end{remark}

\begin{remark}
    We call the above construction a geometric trace to distinguish it from the categorical trace considered in other parts of the article. Note that the diagram \eqref{diag:geotracefund} can not be obtained from the diagram \eqref{diag:trfuncfunddiag} by taking $\cC=\LinCat_{\Shv(S)}$, $x=\Shv(A)$, and $y=\Shv(S)$ due to the failure of the categorical Künneth formula. There are two ways to understand the geometric trace via categorical trace: It can be either understood as a categorical trace in Lu--Zheng's category, as in \S\ref{sec:lzcat}, or as a functoriality of categorical trace in the geometric Langlands setting by putting a nilpotent singular support condition which remedies the categorical Künneth formula.
\end{remark}

\subsubsection{Geometric trace as a natural transformation}\label{sec:cctrnt}
Now we give another description of the geometric trace \[\tr_S(\frc)\in \Hom^0(\cK|_{\Fix(C)},\om_{\Fix(C)/S})\cong \Hom^0(f_{\Fix(C),!}\cK|_{\Fix(C)},\uk_{S}).\] Here we use $f_{\Fix(C)}:\Fix(C)\to S$ to denote the structural morphism of $\Fix(C)$. We use $f:A\to S$ (resp, $f_2:A\times_S A\to S$) to denote the structural morphism of $A$ (resp. $A\times_S A$).

Consider correspondence $(A\times_S A\xleftarrow{c_1\times\id} C\times_S A\xrightarrow{c_0\times\id}A\times_S A)$ and the map $\pr_1:C\times_S A\to C$ which is projection onto the first coordinate. We have the lax commutative diagram
\begin{equation}\label{diag:geotracefund}
    \begin{tikzcd}
        \Shv(S) \ar[r, "\D_!f^*"] \ar[d, "\id"] & \Shv(A\times_S A) \ar[dl, Rightarrow, "\gamma_{\cF}"'] \ar[r, "(c_0\times\id)_!(\pr_1^*\cK\otimes (c_1\times\id)^*(-))"] \ar[d, "f_{2,!}(-\otimes\cF\boxtimes_S\DD_S(\cF))" description] & \Shv(A\times_S A) \ar[dl, Rightarrow, "\eta_{\frc\boxtimes\id}"'] \ar[d, "f_{2,!}(-\otimes\cF\boxtimes_S\DD_S(\cF))" description] \ar[r, "f_!\D^*"] & \Shv(S) \ar[dl, Rightarrow, "\delta_{\cF}"] \ar[d, "\id"] \\
        \Shv(S) \ar[r, "\id"] & \Shv(S) \ar[r, "\id"] & \Shv(S) \ar[r, "\id"] & \Shv(S)
    \end{tikzcd}
\end{equation}
We now explain the three natural transformations in \eqref{diag:geotracefund}:
\begin{itemize}
\item The natural transformation $\gamma_{\cF}$ is defined by \[\gamma_{\cF}:f_{2,!}(\D_!f^*(-)\otimes\cF\boxtimes_S\DD_S(\cF))\isom f_{2,!}\D_!(f^*(-)\otimes\cF\otimes\DD_S(\cF))\to f_!(f^*(-)\otimes\om_{A/S})\isom -\otimes f_!f^!\uk_{S}\to\id .\] The first map is given by the projection formula for $\D$, the second map is given by $f_{2,!}\D_!=f_!$ and $\cF\otimes\DD_S(\cF)\to\om_{A/S}$, the third map is given by the projection formula for $f$, the final map is given by the adjunction map for adjoint pair $(f_!,f^!)$.

\item Consider the cohomological correspondence $\frc\boxtimes\id\in\Cor_{C\times_S A,\cK\boxtimes\uk_{A}}(\cF\boxtimes\uk_{A},\cF\boxtimes\uk_{A})$. The natural transformation $\eta_{\frc\boxtimes\id}$ is given by \begin{equation}\label{eq:cctont} \begin{split}
    \eta_{\frc\boxtimes\id}:f_{2,!}((c_0\times\id)_!(\pr_1^*\cK\otimes (c_1\times\id)^*(-))\otimes\cF\boxtimes_S\DD_S(\cF))&\isom f_{2,!}( -\otimes c_{1,!}(c_0^*\cF\otimes\cK)\boxtimes_S\DD_S(\cF)) \\ & \to f_{2,!}(-\otimes \cF\boxtimes_S\DD_S(\cF))
\end{split}.\end{equation} The first isomorphism comes from several applications of base change isomorphisms and projection formula, the second map comes from the cohomological correspondence $\frc:c_{1,!}(c_0^*\cF\otimes\cK)\to\cF$. Note that the middle square in \eqref{diag:geotracefund} is induced from the following more primitive square:
\begin{equation}\label{eq:cctrnt}
    \begin{tikzcd}
        \Shv(A) \ar[d, "f_!(-\otimes\cF)"'] \ar[r, "c_{0,!}(\cK\otimes c_{1}^*(-))"] & \Shv(A) \ar[d, "f_!(-\otimes\cF)"] \ar[dl, Rightarrow, "\eta_{\frc}"] \\
        \Shv(S) \ar[r, "\id"] & \Shv(S)
    \end{tikzcd}
.\end{equation}

\item The natural transformation $\delta_{\cF}$ is given by \[\delta_{\cF}:f_!\D^*\isom f_{2,!}(-\otimes\D_!\uk_{A})\to f_{2,!}(-\otimes \D_!(\cF\otimes^!\DD_S(\cF)))\to f_{2,!}(-\otimes\cF\boxtimes_S\DD_S(\cF)).\] The first isomorphism uses the projection formula for $\Delta$, the second map uses the map $\uk_A\to\cF\otimes^!\DD_S(\cF)$, the third map uses the adjunction map for adjoint pair $(\D_!,\D^!)$.

\end{itemize}

Note that after evaluating two routes from the left-top corner of \eqref{diag:geotracefund} to the right-bottom corner of \eqref{diag:geotracefund} on the element $\uk_{S}\in\Shv(S)$, we get a map \begin{equation} \label{eq:geotraceseconddef}\tr'_S(\frc):f_{\Fix(C),!}\cK|_{\Fix(C)}\cong f_!\D^*(c_0\times\id)_!(\pr_1^*\cK\otimes (c_1\times\id)^*\D_!f^*\uk_{S})\to \uk_{S}.\end{equation}
\begin{lemma}\label{lem:geotrace=diagtrace}
    We have $\tr'_S(\frc)=\tr_S(\frc)\in \Hom^0(f_{\Fix(C),!}\cK|_{\Fix(C)},\uk_{S})$.
\end{lemma}

\begin{proof}
    This is a routine diagram chase. We leave it to the reader.
\end{proof}

\subsubsection{Geometric Shtuka construction}\label{sec:geoshtcons}
In this section, we assume the base field $F=\FF_q$.

\begin{defn}\label{def:ccsht}
    For each correspondence $c:(A\xleftarrow{c_1}C\xrightarrow{c_0}A)$, we define its \emph{Frobenius twist} to be the correspondence $c^{(1)}:(A\xleftarrow{c_1}C\xrightarrow{\Frob\circ c_0}A)$. For each cohomological correspondence $\frc\in\Cor_{C,\cK}(\cF,\cF)$, we define \[\frc^{(1)}\in \Cor_{C,\cK}(\cF,\cF)\]\footnote{Note that the two vector spaces $\Cor_{C,\cK}(\cF,\cF)$ appearing here are different: The first is defined using $c$ and the second is defined using $c^{(1)}$.} to be the element \[\frc^{(1)}:(c_0^*\circ \Frob^*\cF)\otimes\cK \isom (c_0^*\cF)\otimes\cK\xrightarrow{\frc}c_1^!\cF.\]
\end{defn}

Consider Cartesian diagram \[\begin{tikzcd} \Sht(C) \ar[r, "i"]\ar[d, "p"] & C\ar[d, "{(c_1,\,\Frob\circ c_0)}"] \\ A\ar[r, "\D"] & A\times_S A\end{tikzcd}.\] Note that $\Sht(C)$ is $\Fix(C)$ for the correspondence $c^{(1)}$. Therefore, the construction in the precious section gives a map \[\tr_{\Sht,S}:=\tr_{S}\circ (-)^{(1)}:\Cor_{C,\cK}(\cF,\cF)\to H_0^{BM}(\Sht(C)/S,\cK|_{\Sht(C)}).\]

\subsubsection{Functoriality of geometric trace I}\label{sec:ccfunc}
Now we formulate compatibilities of the geometric trace construction in \S\ref{sec:cctrace} with proper push-forward, smooth pull-back, and specialization. We will leave their proof to \S\ref{sec:lzcat}. 

\begin{thm}[Compatibility with proper push-forward]\label{thm:trproperpush}
    Consider a morphism of correspondences between algebraic stacks over $S$ \begin{equation}\label{diag:trpushdiag}\begin{tikzcd}
        A \ar[d, "f_A"] & C \ar[l, "c_1"'] \ar[r, "c_0"] \ar[d, "f"] & A \ar[d, "f_A"]  \\
        B & D \ar[l, "d_1"'] \ar[r, "d_0"] & B 
    \end{tikzcd}\end{equation} in which $f_A, f$ are proper\footnote{By a proper morphism between algebraic stacks, we mean a map representable in proper algebraic spaces.} (hence the right square is pushable). Then the induced map $\Fix(f):\Fix(C)\to \Fix(D)$ is proper.
    Given $\cF\in \Shv(A)_{\ULA}$ and $\cK\in\Shv(C)$, together with a map $\alpha:f^*\cL\to\cK$, the diagram \[\begin{tikzcd}
        \Cor_{C,\cK}(\cF,\cF) \ar[r, "f_!"] \ar[d, "\tr_S"] & \Cor_{D,\cL}(f_{A,!}\cF,f_{A,!}\cF) \ar[d, "\tr_{S}"] \\
        H_0^{BM}(\Fix(C)/S,\cK|_{\Fix(C)}) \ar[r, "\Fix(f)_!"] & H_0^{BM}(\Fix(D)/S,\cL|_{\Fix(D)})
    \end{tikzcd}\] is commutative. Here the map $\Fix(f)_!:H_0^{BM}(\Fix(C)/S,\cK|_{\Fix(C)})\to H_0^{BM}(\Fix(D)/S,\cL|_{\Fix(D)})$ is induced from the map \[\Fix(\alpha):\Fix(f)^*(\cL|_{\Fix(D)})\isom (f^*\cL)|_{\Fix(C)}\xrightarrow{\alpha|_{\Fix(C)}} \cK|_{\Fix(C)}\] via Example \ref{eg:pushbm}.
\end{thm}

\begin{thm}[Compatibility with smooth pull-back]\label{thm:trpull}
    Consider the diagram \eqref{diag:trpushdiag}. Assume $f_A$ is smooth and the left square is pullable of defect $d$. Then the induced map $\Fix(f):\Fix(C)\to\Fix(D)$ is quasi-smooth of relative dimension $d$. Given $\cF\in\Shv(B)_{\ULA}$, $\cK\in\Shv(D)$, $\cL\in\Shv(C)$, and a map $\alpha:\cL\langle -2d\rangle\to f^*\cK$, the diagram \[\begin{tikzcd}\Cor_{D,\cK}(\cF,\cF) \ar[r, "f^*"] \ar[d, "\tr_S"] & \Cor_{C,\cL}(f_A^*\cF,f_A^*\cF) \ar[d, "\tr_S"] \\
    H_0^{BM}(\Fix(D)/S,\cK|_{\Fix(D)}) \ar[r, "\Fix(f)^*"] & H_0^{BM}(\Fix(C)/S,\cL|_{\Fix(C)})\end{tikzcd}\] is commutative. Here the map $\Fix(f)^*:H_0^{BM}(\Fix(D)/S,\cK|_{\Fix(D)})\to H_0^{BM}(\Fix(C)/S,\cL|_{\Fix(C)})$ is induced from the map \[\Fix(\alpha):\cL|_{\Fix(C)}\langle-2d\rangle\xrightarrow{\alpha|_{\Fix(C)}} (f^*\cK)|_{\Fix(C)}\isom \Fix(f)^*(\cK|_{\Fix(D)})\] via Example \ref{eg:pullbm}.
\end{thm}

Recall we use $\eta$ to denote the generic point of $\AA^1_F$ and $s$ to denote the zero point of $\AA^1_F$.
\begin{thm}[Compatibility with specialization]\label{thm:trspecial}
     Consider an algebraic stack $S$ over $\AA^1_F$.  Given a correspondence between algebraic stacks over $S$: \[A\xleftarrow{c_1} C \xrightarrow{c_0} A,\] together with sheaves $\cF\in\Shv(A_{\eta})_{\ULA}$, $\cK\in\Shv(C_{\eta})$, $\cL\in\Shv(C_{s})$, and a map $\alpha:\cL\to\Psi\cK$, the diagram \[\begin{tikzcd}
        \Cor_{C_{\eta},\cK}(\cF,\cF) \ar[r, "\Psi"] \ar[d, "\tr_{S_{\eta}}"] & \Cor_{C_s,\cL}(\Psi\cF,\Psi\cF) \ar[d, "\tr_{S_s}"] \\
        H_0^{BM}(\Fix(C_{\eta})/S_{\eta},\cK|_{\Fix(C_{\eta})}) \ar[r, "\Psi"] & H_0^{BM}(\Fix(C_{s})/S_{s},\cL|_{\Fix(C_{s})})
    \end{tikzcd}\] is commutative. Here the map $\Psi:H_0^{BM}(\Fix(C_{\eta})/S_{\eta},\cK|_{\Fix(C_{\eta})})\to H_0^{BM}(\Fix(C_{s})/S_{s},\cL|_{\Fix(C_{s})})$ is induced from the map \[\Fix(\alpha):\cL|_{\Fix(C_s)}\to (\Psi\cK)|_{\Fix(C_s)}\to\Psi(\cK|_{\Fix(C_{\eta})})\] via Example \ref{eg:spbm}.
\end{thm}

\begin{remark}
    Note that theorems above with constant kernel were proved in \cite{varshavsky2007lefschetz} and \cite{Lu_2022} by slightly different methods for schemes. The argument of \cite{Lu_2022} can be easily generalized to Artin stacks (see \cite{feng2024modularityhigherthetaseries}). We will also prove these theorems by generalizing the argument of \cite{Lu_2022}.
\end{remark}

\subsubsection{Functoriality of geometric trace II}
Now we generalize Theorem  \ref{thm:trproperpush} to the case that $f$ admits compactification with contracting boundary. The main result in this section is Theorem  \ref{thm:trcontractpush}.

Our strategy follows \cite{varshavsky2007lefschetz}\cite{gaitsgory2024localtermscategoricaltrace}\cite{feng2024modularityhigherthetaseries}, in which they prove compatibilities of geometric trace with restriction to a closed substack on which the correspondence is contracting in different settings. As they did in \textit{loc.cit}, we use deformation to the normal cone, which has very simple behavior along a substack on which the correspondence is contracting.

We begin with the definition of a correspondence contracting on a substack following \cite[Definition\,1.5.1]{varshavsky2007lefschetz} (in the scheme case) and \cite[Definition\,7.2.1]{feng2024modularityhigherthetaseries} (for Artin stacks):
\begin{defn}
    Consider a correspondence between algebraic stacks over $S$ \[c:(A\xleftarrow{c_1} C \xrightarrow{c_0} A).\] Let $Z\sub A$ be a closed substack defined by an ideal sheaf $\cI_Z\sub\cO_A$. \begin{itemize}
        \item We say that $c$ stabilizes $Z$ if $c_0^*\cI_Z\sub c_1^*\cI_Z$.
        \item We say that $c$ is \emph{contracting near} $Z$ if $c_0^*\cI_Z\sub c_1^*\cI_Z$ and there exists $n\in\ZZ_{\geq 0}$ such that $c_0^*\cI_Z^{n}\sub c_1^*\cI_Z^{n+1}$.
    \end{itemize} 
\end{defn}

\begin{example}\label{eg:frobcontract}
    Given a correspondence between algebraic stacks $c=(A\xleftarrow{c_1} C \xrightarrow{c_0} A)$ defined over $S/\FF_q$. Then $c^{(1)}$ is contracting on $Z$ if $c$ stabilizes $Z$.
\end{example}

For $Z$ stabilized by $c$, denote $C_Z:=C\times_A Z$ in which we are using $c_1:C\to A$. Consider the correspondence \[D_Z(A)\xleftarrow{\tilc_1} D_{C_Z}(C)\xrightarrow{\tilc_0} D_Z (A)\] where $D_Z(A)$ is the deformation to the normal cone of $A$ along the closed substack $Z$ and similarly for $D_{C_Z}(C)$. This correspondence is defined over $\AA^1_S$, whose fiber over any $S$-point of $\AA^1_S$ which is disjoint from zero is identical to $(A\xleftarrow{c_1} C \xrightarrow{c_0} A)$, while the fiber over $0\in\AA^1_S$ is identified with \[N_Z(A)\xleftarrow{N(c_1)} N_{C_Z}(C)\xrightarrow{N(c_0)} N_Z (A)\] where $N_Z(A)$ is the normal cone of $A$ along $Z$ and similarly for $N_{C_Z}(C)$.

The following are some easy facts:
\begin{lemma}\label{lem:contract}
    Consider a correspondence $c$ between algebraic stacks over $S$ \[A\xleftarrow{c_1} C \xrightarrow{c_0} A\] with a closed substack $Z\sub A$ on which $c$ is contracting. The following statements are true:
    \begin{enumerate}
        \item The set-theoretic image of $N(c_0):N_{C_Z}(C)\to N_Z(A)$ is contained in $Z\sub N_Z(A)$.
        \item The natural map defines an open and closed embedding $\Fix(C_Z)_{\red} \sub \Fix(C)_{\red}$.
        \item The natural map defines an isomorphism $\Fix(N_{C_Z}(C))_{\red}\cong \Fix(C_Z)_{\red}$.
        \item We have disjoint decomposition $\Fix(D_{C_Z}(C))_{\red}=\Fix(C_Z)_{\red}\times_S\AA^1_S\coprod (\Fix(C)_{\red}\bs\Fix(C_Z)_{\red})\times_S(\AA^1_S\bs\{0\})$. 
    \end{enumerate}
\end{lemma}
\begin{proof}
    (1) is \cite[Lemma\,7.2.4]{feng2024modularityhigherthetaseries}. (2) is \cite[Proposition\,7.5.4]{feng2024modularityhigherthetaseries}. (3) is a consequence of (1). (4) is \cite[Proposition\,7.5.5]{feng2024modularityhigherthetaseries}.\footnote{The formulation in \textit{loc.cit} is not correct as stated. The proof in \textit{loc.cit} shows what we state here.}
\end{proof}

In this case, we call the nearby cycle functors $\Psi:\Shv(A)\to \Shv(N_Z(A))$ and $\Psi:\Shv(C)\to\Shv(N_{C_Z}(C))$ specializations.

Note the following property of the $\Psi$:
\begin{lemma}\label{lem:spres}
    The natural transformation $\alpha:i_{Z,0}^*\circ\Psi\to i_Z^*$ is an isomorphism. Here $i_{Z,0}:Z\to N_Z(A)$ is the zero section, $i_Z:Z\to A$ is the natural inclusion.
\end{lemma}
\begin{proof}
    This is \cite[\S7.5(b)]{gaitsgory2024localtermscategoricaltrace}.
\end{proof}

The following is a direct consequence of Theorem  \ref{thm:trspecial}:

\begin{prop}\label{prop:trspecial}
    Consider a correspondence $c$ between algebraic stacks over $S$ \[A\xleftarrow{c_1} C \xrightarrow{c_0} A\] with a closed substack $Z\sub A$ stabilized by $c$. For $\cF\in\Shv(A)_{\ULA}$ and $\cK\in\Shv(C)$, the diagram
    \[\begin{tikzcd}
        \Cor_{C,\cK}(\cF,\cF) \ar[r, "\Psi"] \ar[d, "\tr_S"] & \Cor_{N_{C_Z}(C),\Psi\cK}(\Psi\cF,\Psi\cF) \ar[d, "\tr_{S}"] \\
        H_0^{BM}(\Fix(C)/S,\cK_{\Fix(C)}) \ar[r, "\Psi"] & H_0^{BM}(\Fix(N_{C_Z}(C))/S,(\Psi\cK)|_{\Fix(N_{C_Z}(C))})
    \end{tikzcd}\] is commutative.
\end{prop}

We have the following corollary which is generalization of \cite[Corollary\,5.6(c)]{gaitsgory2024localtermscategoricaltrace}:
\begin{cor}\label{cor:trvanlocalterm}
    In the setting of Proposition \ref{prop:trspecial}, assume moreover that $c$ is contracting along $Z$. If $\cF|_Z=0$, then the composition \[\Cor_{C,\cK}(\cF,\cF)\xrightarrow{\tr_S} H_0^{BM}(\Fix(C),\cK|_{\Fix(C)})\xrightarrow{i_{C_Z}^*} H_0^{BM}(\Fix(C_Z),\cK|_{\Fix(C_Z)})\] is zero. Here, the map $i_{C_Z}:\Fix(C_Z)\sub \Fix(C)$ is an open and closed embedding on the reduced substacks by Lemma \ref{lem:contract}(2) and $i_{C_Z}^*$ is the natural restriction map.
\end{cor}

\begin{proof}
    Applying Lemma \ref{lem:spres} to $C_Z\sub C$ and using Lemma \ref{lem:contract}(3), we get a natural isomorphism \[(\Psi\cK)|_{\Fix(N_{C_Z}(C))}\cong \cK|_{\Fix(C_Z)}.\] Under this isomorphism, one can check that the restriction map \[i_{C_Z}^*:H_0^{BM}(\Fix(C)/S,\cK|_{\Fix(C)})\to H_0^{BM}(\Fix(C_Z)/S,\cK|_{\Fix(C_Z)})\] is identified with the map \[\Psi:H_0^{BM}(\Fix(C)/S,\cK|_{\Fix(C)})\to H_0^{BM}(\Fix(N_{C_Z}(C))/S,(\Psi\cK)|_{\Fix(N_{C_Z}(C))})\] used in Proposition \ref{prop:trspecial}. By Proposition \ref{prop:trspecial}, we are reduced to show that the composition \[\Cor_{C,\cK}(\cF,\cF)\xrightarrow{\Psi} \Cor_{N_{C_Z}(C),\Psi\cK}(\Psi\cF,\Psi\cF)\xrightarrow{\tr_S} H_0^{BM}(\Fix(N_{C_Z}(C))/S,(\Psi\cK)|_{\Fix(N_{C_Z}(C))})\] is zero. Note that \[\Cor_{N_{C_Z}(C),\Psi\cK}(\Psi\cF,\Psi\cF)=\Hom^0(N(c_0)^*\Psi\cF\otimes\Psi\cK,N(c_1)^!\Psi\cF).\] By Lemma \ref{lem:contract}(1) and Lemma \ref{lem:spres}, we know \[N(c_0)^*\Psi\cF\cong N(c_0)^*i_{Z,0,!}i_{Z,0}^*\Psi\cF\cong N(c_0)^*i_{Z,0,!}i_Z^*\cF=0.\] This implies $\Cor_{N_{C_Z}(C),\Psi\cK}(\Psi\cF,\Psi\cF)=0$ and we are done.
\end{proof}

\begin{defn}\label{def:corrcompact}
    Consider a diagram of correspondences between algebraic stacks over $S$: \[\begin{tikzcd}
        A \ar[d, "f_A"] & C \ar[l, "c_1"'] \ar[r, "c_0"] \ar[d, "f"] & A \ar[d, "f_A"]  \\
        B & D \ar[l, "d_1"'] \ar[r, "d_0"] & B 
    \end{tikzcd}.\] 
    
    \begin{itemize}
        \item We say that $f$ \emph{admits compactification with contracting boundary} if the diagram can be extended to a diagram \[\begin{tikzcd}
        A \ar[d, "j_A"] & C \ar[l, "c_1"'] \ar[r, "c_0"] \ar[d, "j"] & A \ar[d, "j_A"]  \\
        \overline{A} \ar[d, "\overline{f}_{A}"] & \overline{C} \ar[l, "\overline{c}_1"'] \ar[r, "\overline{c}_0"] \ar[d, "\overline{f}"] & \overline{A} \ar[d, "\overline{f}_{A}"] \\
        B & D \ar[l, "d_1"'] \ar[r, "d_0"] & B 
    \end{tikzcd}\] in which $j_A$ and $j$ are open immersions, $f_{\overline{A}},\overline{f}$ are proper, the two top squares are Cartesian, and there exists an ideal sheaf $\cI_Z\sub\cO_A$ defining a closed substack $Z\sub \overline{A}$ on which the middle correspondence $\overline{c}$ is contracting and $Z_{\red}=\partial A:=\overline{A}\bs A$.

    \item We say that $f$ \emph{admits compactification} if the condition above still holds except we only require $Z$ to be stabilized by $\overline{c}$.

    \item For $\cF\in \Shv(A)_{\ULA}$ and a compactification as above, we say that $\cF$ is \emph{good} for the compactification if $j_{A,!}\cF\in \Shv(\overline{A})_{\ULA}$.

    \end{itemize}
\end{defn}

\begin{thm}\label{thm:trcontractpush}
    In the setting of Theorem  \ref{thm:trproperpush} but only requiring that $f$ admits compactification with contracting boundary and $\cF\in\Shv(A)_{\ULA}$ is good for the the compactification, then the map $\Fix(f):\Fix(C)\to\Fix(D)$ is proper, and the same conclusion holds.
\end{thm}

\begin{proof}
    The conclusion that $\Fix(f)$ is proper follows from Lemma \ref{lem:contract}(2).
    We can assume $f^*\cL=\cK$ since the general situation can be deduced from this case. Consider the diagram \begin{equation}\label{cd:trcontractpush1}\begin{tikzcd}
        \Cor_{C,f^*\cL}(\cF,\cF) \ar[r, "j_!"] \ar[d, "\tr_S"] & \Cor_{\overline{C},\overline{f}^*\cL}(j_{A,!}\cF,j_{A,!}\cF) \ar[r, "\overline{f}_!"] \ar[d, "\tr_S"] & \Cor_{D,\cL}(f_{A,!}\cF,f_{A,!}\cF) \ar[d, "\tr_S"] \\
        H_0^{BM}(\Fix(C)/S,f^*\cL|_{\Fix(C)}) \ar[r, "\Fix(j)_!"] & H_0^{BM}(\Fix(\overline{C})/S,\overline{f}^*\cL|_{\Fix(\overline{C})}) \ar[r, "\Fix(\overline{f})_!"] & H_0^{BM}(\Fix(D)/S,\cL|_{\Fix(D)})
    \end{tikzcd}\end{equation} in which each horizontal map is induced by a push-forward of cohomological correspondence. We want to show that the outer square of \eqref{cd:trcontractpush1} is commutative. Theorem  \ref{thm:trproperpush} implies that the right square is commutative. Therefore, we only need to show the commutativity of the left square.

    Define $\partial\overline{C}=\overline{C}\bs C$ and $i_{\partial C}:\partial C\to C$. Consider the diagram \begin{equation}\label{cd:trcontractpush2}\begin{tikzcd}
        \Cor_{C,f^*\cL}(\cF,\cF)  \ar[d, "\tr_S"] & \Cor_{\overline{C},\overline{f}^*\cL}(j_{A,!}\cF,j_{A,!}\cF) \ar[l, "j^*"'] \ar[d, "\tr_S"] &  \\
        H_0^{BM}(\Fix(C)/S,f^*\cL|_{\Fix(C)})  & H_0^{BM}(\Fix(\overline{C})/S,\overline{f}^*\cL|_{\Fix(\overline{C})}) \ar[r, "\Fix(i_{\partial C})^*"] \ar[l, "\Fix(j)^*"'] & H_0^{BM}(\Fix(C_Z)/S,\overline{f}^*\cL|_{\Fix(C_Z)})
    \end{tikzcd}. \end{equation} 
    By Theorem  \ref{thm:trpull}, the left square of \eqref{cd:trcontractpush2} is commutative. 
    Note that the horizontal maps in the left square of \eqref{cd:trcontractpush1} are splittings of the horizontal maps in the left square of \eqref{cd:trcontractpush2} and $\Fix(i_{\partial C})^*\circ \Fix(j)_!=0$ by Lemma \ref{lem:contract}(2). Therefore, we are reduced to show $\Fix(i_{\partial C})^*\circ \tr_S\circ j_!=0$. This follows from Corollary \ref{cor:trvanlocalterm} since $(j_{A,!}\cF)|_{\partial A}=0$.
\end{proof}

We have the following immediate corollary of Theorem  \ref{thm:trcontractpush} and Example \ref{eg:frobcontract}:

\begin{cor}\label{cor:trshtpush}
Consider a morphism of correspondences between algebraic stacks over $S/\FF_q$ \[\begin{tikzcd}
        A \ar[d, "f_A"] & C \ar[l, "c_1"'] \ar[r, "c_0"] \ar[d, "f"] & A \ar[d, "f_A"]  \\
        B & D \ar[l, "d_1"'] \ar[r, "d_0"] & B 
    \end{tikzcd}.\] Assume $f$ admits compactification and $\cF\in \Shv(A)_{\ULA}$ is good for the compactification. Then the induced map $\Sht(f):\Sht(C)\to \Sht(D)$ is proper. Moreover, given $\cF\in \Shv(A)_{\ULA}$ and $\cK\in\Shv(C)$, together with a map $\alpha:f^*\cL\to\cK$, we have a commutative diagram \[\begin{tikzcd}
        \Cor_{C,\cK}(\cF,\cF) \ar[r, "f_!"] \ar[d, "\tr_{\Sht,S}"] & \Cor_{D,\cL}(f_{A,!}\cF,f_{A,!}\cF) \ar[d, "\tr_{\Sht,S}"] \\
        H_0^{BM}(\Sht(C)/S,\cK|_{\Sht(C)}) \ar[r, "\Sht(f)_!"] & H_0^{BM}(\Sht(D)/S,\cL|_{\Sht(D)})
    \end{tikzcd}.\]
    
\end{cor}

\subsubsection{Lu--Zheng category}\label{sec:lzcat}
In this section, we follow the strategy in \cite{Lu_2022} to prove theorems in \S\ref{sec:ccfunc}. We consider variants of $2$-categories considered in \cite{Lu_2022} adapted to cohomological correspondences with kernel. Note that similar construction is also made in \cite{feng2024modularityhigherthetaseries}.

\begin{defn}
    Let $S$ be an algebraic stack. We define the symmetric monoidal $2$-category \[(\LZ(S)_!,\otimes, 1_{\LZ(S)_!} )\] as follows:
    \begin{itemize}
        \item The objects are pairs $(A,\cF)$ where $A$ is an algebraic stack over $S$, and $\cF\in\Shv(A)$.
        \item A morphism from $(A_0,\cF_0)$ to $(A_1,\cF_1)$ is a triple $(c,\cK,\frc)$ where $c=(A_1 \xleftarrow{c_1} C\xrightarrow{c_0} A_0)$ is a correspondence, $\cK\in \Shv(C)$ is a kernel sheaf, and $\frc\in\Cor_{C,\cK}(\cF_0,\cF_1)$ is a cohomological correspondence with kernel $\cK$.
        \item The composition of morphisms $(c,\cK,\frc):(A_0,\cF_0)\to (A_1,\cF_1)$ and $(d,\cL,\frd):(A_1,\cF_1)\to (A_2,\cF_2)$ is $(e,\cM,\fre)$, where $e$ is the outer correspondence in the diagram \[\begin{tikzcd}
            &&E \ar[dl, "c_1'"'] \ar[dr, "d_0'"]&& \\
            &D \ar[dl, "d_1"'] \ar[dr, "d_0"]&&C \ar[dl, "c_1"'] \ar[dr, "c_0"]& \\
            A_2&&A_1&&A_0
        \end{tikzcd}\] where the diamond is Cartesian, the kernel $\cM=d_0'^*\cK\otimes c_1'^*\cL$, and $\fre=\frd\circ\frc$ is the obvious notion of composition of cohomological correspondences.
        \item Given two morphisms $(c,\cK,\frc)$ and $(d,\cL,\frd)$ from $(A_0,\cK_0)$ to $(A_1,\cK_1)$, a $2$-morphism $\eta:(c,\cK,\frc)\to (d,\cL,\frd)$ is a map of correspondences \begin{equation}\label{diag:lzfund}\begin{tikzcd}
            A_1 \ar[d, equal] & C \ar[l, "c_1"'] \ar[r, "c_0"] \ar[d, "f"] & A_0 \ar[d, equal] \\
            A_1 & D \ar[l, "d_1"'] \ar[r, "d_0"] & A_0
        \end{tikzcd}\end{equation} in which $f$ is proper, together with a map $\alpha:f^*\cL\to \cK$ such that $f_!\frc=\frd$.
        \item The monoidal unit $1_{\LZ(S)_!}=(S,\uk_{S})$. The tensor product of objects $(A_0,\cF_0)$ and $(A_1,\cF_1)$ is defined as \[(A_0,\cF_0)\otimes (A_1,\cF_1):=(A_0\times_S A_1, \cF_0\boxtimes_S \cF_1 ).\] The tensor product of morphisms is defined verbatim as in \cite[\S5.1.3]{feng2024modularityhigherthetaseries}.
    \end{itemize}
\end{defn}

Consider the subcategory $\LZ(S)_!^0$ with the same objects, whose $1$-morphisms are those $1$-morphisms $(c,\cK,\frc)$ in $\LZ(S)_!$ with $\cK=\uk_C$, and whose $2$-morphisms are those $(f,\alpha)$ as above such that $\alpha:f^*\uk_D\to\uk_C$ is the tautological map. Then $\LZ(S)_!^0$ is the original $2$-category considered in \cite{Lu_2022} (but considering algebraic stacks instead of schemes) and \cite{feng2024modularityhigherthetaseries} (but working with \'etale sheaf theory rather than motivic sheaf theory).

\begin{lemma}\label{lem:lzdual}
    Any object $(A,\cF)$ such that $\cF$ is ULA over $S$ is dualizable in $\LZ(S)_!$, and whose dual is given by $(A,\DD_{S}(\cF))$. 
\end{lemma}
\begin{proof}
    The above statement already holds in the subcategory $\LZ(S)_!^0\sub \LZ(S)_!$ by \cite[Theorem\,2.16]{Lu_2022} and any dualizable object in $\LZ(S)_!^0$ is dualizable in $\LZ(S)_!$.
\end{proof}

The following can be checked directly:
\begin{lemma}\label{lem:lzdualmap}
    For a morphism $(c,\cK,\frc):(A_0,\cF_0)\to(A_1,\cF_1)$ in which $\cF_i$ are ULA over $S$ for $i=0,1$, we have $(c,\cK,\frc)^{\vee}=(c',\cK,\frc^{\vee})$ in which $c'=(A_0\xleftarrow{c_0}C\xrightarrow{c_1}A_1)$ and $\frc^{\vee}$ is the image of $\frc$ under the isomorphism \[\begin{split}
        \Hom^0(c_0^*\cF_0\otimes\cK,c_1^!\cF_1) &\cong\Hom^0(\cK,(c_1,c_0)^!(\cF_1\boxtimes_S\DD_{S}(\cF_0)) \\ &\cong \Hom^0(\cK,(c_0,c_1)^!(\DD_{S}(\cF_0)\boxtimes_S\cF_1)) \\ &\cong \Hom^0(c_1^*\DD_{S}(\cF_1)\otimes\cK,c_0^!\DD_{S}(\cF_0))
    \end{split}\]
\end{lemma}

\begin{lemma}\label{lem:lzrightadjoint}
    For a proper map $f:A\to B$ and $\cF\in \Shv(A)$, consider correspondence $c_f=(B\xleftarrow{f}A=A)$ and cohomological correspondence $\frc_f\in \Cor_{A,\uk_A}(\cF,f_!\cF)$ given by the natural adjunction map $\cF\to f^!f_!\cF$. Then the map \[(c_f,\uk_A,\frc_f):(A,\cF)\to(B,f_!\cF)\] admits right adjoint \[(c_f',\uk_A,\frc_f'):(B,f_!\cF)\to (A,\cF)\] in which $c_f'=(A=A\xrightarrow{f}B)$ and $\frc_f'\in\Cor_{A,\uk_A}(f_!\cF,\cF)$ is the natural map $f^*f_!\cF\cong f^*f_*\cF\to \cF$.
\end{lemma}
\begin{proof}
    The happens already in $\LZ(S)_!^0$ by \cite[Lemma\,2.9]{Lu_2022}.
\end{proof}

Note that the $1$-category $\End_{\LZ(S)_!}(1_{\LZ(S)_!})$ has its objects consisting of triples $(A,\cK,\frc)$ where $A$ is an algebraic stack over $S$, $\cK\in\Shv(A)$, and $\frc\in\Cor_{A,\cK}(\uk_S,\uk_S)=H^{BM}_0(A/S,\cK)$. For any pair $(A,\cF)\in\LZ(S)_!$ such that $\cF$ is ULA over $S$ together with an endomorphism $(c,\cK,\frc)\in\End_{\LZ(S)_!}(A,\cF)$, applying the categorical trace construction in \S\ref{sec:cattrace}, one gets an object \[\tr((c,\cK,\frc),(A,\cF))\in\End_{\LZ(S)_!}(1_{\LZ(S)_!}).\] The following follows from unwinding definition:
\begin{prop}\label{prop:lzloop}
    We have $\tr((c,\cK,\frc),(A,\cF))=(\Fix(C),\cK|_{\Fix(C)},\tr_S(\frc))\in \End_{\LZ(S)_!}(1_{\LZ(S)_!}) $.
\end{prop}

\begin{proof}[Proof of Theorem  \ref{thm:trproperpush}]
    We claim that there is a $2$-commutative diagram \[\begin{tikzcd}
    (A,\cF) \ar[r, "{(c,\,\cK,\,\frc)}"]\ar[d, "{(c_f,\,\uk_A,\,\frc_f)}"'] & (A,\cF) \ar[d, "{(c_f,\,\uk_A,\,\frc_f)}"] \ar[dl, Rightarrow, "\eta"]\\
    (B,f_!\cF) \ar[r, "{(d,\,\cL,\,f_!\frc)}"'] & (B,f_!\cF)
    \end{tikzcd}.\] To define the $2$-morphism $\eta$, consider the map between correspondence \[\begin{tikzcd}
        B \ar[d, "\id"] & C \ar[d, "{(f,\,c_0)}"] \ar[l, "f\circ c_1"'] \ar[r, "c_0"] & A \ar[d, "\id"] \\
        B & D\times_B A \ar[l, "a"'] \ar[r, "\pr_2"] & A
    \end{tikzcd}\] in which $a: D\times_B A\to B$ is induced from the map $d_1:D\to B$. Note that $(f,c_0)$ is proper by the pushability assumption. Consider the map $(f,c_0)^*\pr_1^*\cL\cong f^*\cL\xrightarrow{\alpha}\cK$. Unwinding the definition of push-forward of cohomological correspondence in \S\ref{sec:ccpush}, we know $(f,c_0)_!(\frc_f\circ\frc)=f_!\frc\circ\frc_f$. Therefore, the data above defines a $2$-morphism $\eta$.

    By Lemma \ref{lem:lzdual} and Lemma \ref{lem:lzrightadjoint}, we can apply the functoriality of categorical trace Definition \ref{def:funccattr} to obtain a $2$-morphism \[\tr(\eta):(\Fix(C),\cK|_{\Fix(C)},\tr_S(\frc))\to (\Fix(D),\cL|_{\Fix(D)},\tr_S(f_!\frc)).\] Using Lemma \ref{lem:lzdualmap} and unwinding the construction of Definition \ref{def:funccattr}, we know that the natural transformation $\tr(\eta)$ contains the data $\Fix(\alpha):\Fix(f)^*(\cL|_{\Fix(D)})\to\cK|_{\Fix(C)}$. This forces $\Fix(f)_!\tr_S(\frc)=\tr_S(f_!\frc)$ and we are done.
\end{proof}

\begin{proof}[Proof of Theorem  \ref{thm:trpull}]
    One can define a $2$-category $\LZ(S)^*$ which has the same objects and $1$-morphisms as $\LZ(S)_!$ but with $2$-morphism given by the same diagram of correspondence \eqref{diag:lzfund} and kernel but with $f$ quasi-smooth of some dimension $d$ together with a map $\alpha:\cL\langle-2d\rangle\to\cK$ such that $f^*\frd=\frc$. Then Lemma \ref{lem:lzdual}, Lemma \ref{lem:lzdualmap}, and Proposition \ref{prop:lzloop} still hold in this case. The analogue of Lemma \ref{lem:lzrightadjoint} is the following easy lemma:
    \begin{lemma}
        For a smooth map $f:A\to B$ of relative dimension $d$ and $\cF\in \Shv(B)$, consider correspondence $c_f'=(A=A\xrightarrow{f} B)$ and the tautological cohomological correspondence $\frc_f'\in\Cor_{A,\uk_A}(\cF,f^*\cF)$. The map \[(c_f',\uk_A,\frc_f'):(B,\cF)\to (A,f^*\cF)\] admits right adjoint \[(c_f,\uk_A\langle 2d\rangle,\frc_f):(A,f^*\cF)\to (B,\cF)\] in which $\frc_f\in\Cor_{A,\uk_A\langle2d\rangle}(f^*\cF,\cF)$ is the natural map $f^*\cF\langle 2d\rangle\cong f^!\cF$.
    \end{lemma}
    The rest of the proof remains the same as the proper push-forward case.
\end{proof}

\begin{proof}[Proof of Theorem  \ref{thm:trspecial}]
    One can easily reduce to the case $\cL=\Psi\cK$. In this case, one can define $2$-categories $\LZ(S)_{\eta}$ (resp. $\LZ(S)_{s}$) by modifying the definition of $\LZ(S)_!$ by requiring the sheaves defined only on the generic fiber (resp. special fiber) instead. Moreover, on the level of $2$-morphisms, we require the map $f$ in \eqref{diag:lzfund} to be an isomorphism. Then Lemma \ref{lem:lzdual} and Lemma \ref{lem:lzdualmap} are true for $\LZ(S)_{\eta}$ and $\LZ(S)_s$. Consider the (strict) symmetric monoidal colax functor $\Psi:\LZ(S)_{\eta}\to\LZ(S)_s$ defined as follows: On objects, one define $\Psi(A,\cF)=(A,\Psi\cF)$. On $1$-morphisms, one define $\Psi(c,\cK,\frc)=(c,\Psi\cK,\Psi\frc)$. Note that for morphisms \[(A_0,\cF_0)\xrightarrow{(c,\cK,\frc)} (A_1,\cF_1) \xrightarrow{(d,\cL,\frd)} (A_2,\cF_2),\] there is a $2$-morphism \[\begin{split}\Psi((d,\cL,\frd)\circ(c,\cK,\frc))&=(D\times_{A_1}C,\Psi(\cL\boxtimes_{A_{1,\eta}}\cK),\Psi(\frd\circ\frc))\\&\to(D\times_{A_1}C,\Psi(\cL)\boxtimes_{A_{1,s}}\Psi({\cK}),\Psi\frd\circ\Psi\frc)\\&=\Psi(d,\cL,\frd)\circ \Psi(c,\cK,\frc)\end{split}.\] Here, the morphism in the second row is defined by the natural map $\Psi(\cL)\boxtimes_{A_{1,s}}\Psi({\cK})\to \Psi(\cL\boxtimes_{A_{1,\eta}}\cK)$. There is an obvious extension of $\Psi$ to $2$-morphisms. This defines the colax functor, which can be made symmetric monoidal by \cite[Proposition\,3.1]{Lu_2022}.\footnote{This is for schemes, but the same conclusion holds for algebraic stacks since one can check smooth locally.} For $x=(A,\cF)\in\LZ(S)_{\eta}$ where $\cF$ is ULA over $S_{\eta}$, note that $\Psi (x)=(A,\Psi\cF)$ is still dualizable since $\Psi\cF$ is ULA over $S_s$ (by adapting the proof of \cite[Corollary\,3.9]{Lu_2022}). Consider $F=(c,\cK,\frc)\in\End_{\LZ(S)_{\eta}}(x)$. Note that we have $2$-commutative diagram \[\begin{tikzcd}
        1_{\LZ(S)_{s}}\ar[r, bend left, "\Psi(\rmu_x)" {name=F1}] \ar[r,"\rmu_{\Psi(x)}"'{name=G1}] \ar[from=F1.south-|G1, to=G1, Rightarrow, "\gamma"] & \Psi (x)\otimes \Psi (x^{\vee}) \ar[r, "\Psi(F)\otimes\id"] & \Psi (x)\otimes \Psi (x^{\vee}) \ar[r, bend left, "\Psi(\ev_x)"{name=F2}] \ar[r, "\ev_{\Psi(x)}"'{name=G2}] \ar[from=F2.south-|G2, to=G2, Rightarrow, "\delta"] & 1_{\LZ(S)_{s}} \end{tikzcd}\] in which the $2$-morphisms $\gamma$ and $\delta$ are induced from the natural map $\uk_{A_{s}}\to \Psi\uk_{A_{\eta}}$. By the colax property of $\Psi$, we get a $2$-morphism \[\eta:(\Fix(C),\Psi(\cK|_{\Fix(C)_{\eta}}),\Psi\tr_S(\frc))=\Psi(\tr(F,x))\to\tr(\Psi(F),\Psi(x))=(\Fix(C),(\Psi\cK)|_{\Fix(C)_s},\tr_S(\Psi\frc)).\] After checking that this $2$-morphism contains the natural map $\alpha:(\Psi\cK)|_{\Fix(C)_s}\to \Psi(\cK|_{\Fix(C)_{\eta}})$, we get $\Psi\tr_S(\frc)=\tr_S(\Psi\frc)$.

\end{proof}

\subsection{Relative compactification}\label{sec:relativecompactification}
In this section, we assume that the $G$-variety $X$ is affine spherical. We are going to construct a map $\overline{\pi}:\overline{\Bun}_G^X\to \Bun_G$ serving as the \emph{relative compactification} of $\pi:\Bun_G^X\to\Bun_G$. 

\subsubsection{Preliminaries on spherical varieties}
We briefly review the theory of spherical varieties relevant to us. We refer to \cite{sakellaridis2022intersection} for a more detailed treatment.\footnote{The author is grateful to Yiannis Sakellaridis for explaining the theory of affine degeneration of spherical varieties, especially, Lemma \ref{lem:Xcomp}. Any mistake in the section is due to the author.}

In this section, we work with varieties defined over a field $F$. By a spherical variety, we mean a normal $G$-variety with an open $B$-orbit. For an affine algebraic variety $X$ with an action by an algebraic group $G$, we denote $X\sslash G:=\Spec (\cO(X)^G)$. For an affine spherical $G$-variety $X$, the quotient $X\sslash N$ is a $T$-variety, on which the $T$-action factors through a quotient $T\surj T_X$, making $X\sslash N$ a toric $T_X$-variety.

For simplicity, we first discuss the theory when $\mathrm{char} F=0$. In this case, we have a decomposition into irreducible representations $F(X)\cong\bigoplus_{\l\in X^*(T_X)}V_{\l}$. Define $\frc_X\sub X^*(T_X)$ such that $\cO(X)\cong\bigoplus_{\l\in\frc_X}V_{\l}$.

For each $G$-invariant discrete valuation $v$ on $X$, by restricting $v$ to all $B$-eigen functions on $X$, one obtains an element in $X_*(T_X)_{\QQ}$. All elements in $X_*(T_X)_{\QQ}$ obtained in this way generate a cone $\cV\sub X_*(T_X)_{\QQ}$, which is a fundamental domain for the little Weyl group $W_X$ acting on $X_*(T_X)_{\QQ}$.

There exists a canonical filtration $\{F_{\l}\sub \cO(X)\}_{}$ defined as follows: The subspace $F_{\l}\sub \cO(X)$ is the direct sum of all sub irreducible $G$-representations of $\cO(X)$ with highest weight $\mu$ satisfying $\langle \l-\mu,\cV\rangle \leq 0$. Then one can form the Rees algebra $\cO(\sX):=\bigoplus_{\l\in X^*(T_X)}F_{\l}\otimes e^{\l}\sub \cO(X\times T_X)$ and define the affine degeneration of $X$ to be the variety $\sX:=\Spec(\cO(\sX))$. Define $\overline{T_{X,\rmss}}:=\Spec(\bigoplus_{\l\in X^*(T_X), \langle \l,\cV\rangle\leq 0}k\cdot e^{\l})$, which is a toric $T_{X,\rmss}$-variety where $T_{X,\rmss}:=T_X/\cZ(X)^{\circ}$. Here $\cZ(X)^{\circ}$ is the torus satisfying $X_*(\cZ(X)^{\circ})\cong \cV\cap-\cV\cap X_*(T_X)$.

Both varieties $\sX$ and $\overline{T_{X,\rmss}}$ carry natural $G\times T_X$-actions. There is a natural $G\times T_X$-equivariant map $a:\sX\to \overline{T_{X,\rmss}}$. We use $\sX^{\bullet}\sub X$ to denote the open subvariety whose intersection with each fiber of $a$ is the open $G$-orbit of the fiber. Denote $\sX^{\circ}:=a^{-1}(T_X)$. Then there is a canonical isomorphism $\sX^{\circ}\cong X\times T_X/\cZ(X)^{\circ}$ as $G\times T_X$-varieties.

We will need the following fact from \cite[\S2.2.2]{sakellaridis2022intersection}:
\begin{lemma}\label{lem:Xcomp}
    The stack $\sX^{\bullet}/T_X$ is representable by a proper algebraic variety.
\end{lemma}

\begin{proof}
    This is mentioned in \cite[\S2.2.2]{sakellaridis2022intersection} without proof. We provide a proof for completeness. First note that $T_X(\overline{F})$ acts freely on $\sX^{\bullet}(\overline{F})$. Indeed, one only needs to check that $T_X(\overline{F})$ acts freely on the $\overline{F}$-points of the special fiber of $\sX^{\bullet}\to \overline{T_{X,\rmss}}$, which follows from the fact that this special fiber is homogeneous horospherical with associated torus $T_X$. Second, note that $\sX^{\bullet}$ can be covered by $T_X$-stable open affine sub-varieties. Indeed, as a spherical $G\times T_X$-variety, $\sX^{\bullet}$ can be covered by $B\times T_X$-stable open affine sub-varieties, which are automatically $T_X$-stable. These two facts imply that the quotient stack $\sX^{\bullet}/T_X$ is representable by a scheme, which is clearly normal and of finite type. 
    
    We claim that $\sX^{\bullet}/T_X$ is a spherical $G$-variety. In fact, we only need to check that $\sX^{\bullet}/T_X$ is separated. Indeed, since $\sX^{\bullet}$ admits a unique closed $G\times T_X$-orbit
    , we know that $\sX^{\bullet}/T_X$ admits an open affine subset $U$ intersecting non-trivially with every $G$-orbit. For any two maps $a_1,a_2:\Spec(R)\to \sX^{\bullet}/T_X$ inducing the same map on $\Spec(\mathrm{Frac}(R))$ for a discrete valuation ring $R$, via translation by $G$, one can assume that $a_1,a_2$ has their image lying in $U$. This implies $a_1=a_2$, hence, the separatedness of $\sX^{\bullet}/T_X$.
    
    Note that the valuations of the $G$-stable divisors of $\sX^{\bullet}/T_X$ generate exactly its valuation cone. We know that the colored cone of $\sX^{\bullet}/T_X$ contains the valuation cone, hence, it is a proper algebraic variety.
\end{proof}

When $\mathrm{char} F >0$, we make the following assumption:
\begin{assumption}\label{assumption:goodfil}
    The $G$-representation $\cO(X)$ admits a good filtration in the sense of \cite[\S II.4.16]{jantzen2003representations}. Moreover, we assume the good filtration can be chosen to be multiplicative. That is, there exists an increasing filtration $\{F_{\lambda}\sub \cO(X)\}_{\lambda\in X^*(T_X)}$\footnote{The partial order on $X^*(T_X)$ is the same as for the filtration $\{F_{\lambda}\}_{\lambda\in X^*(T_X)}$ in characteristic zero case. That is, $\lambda\leq\mu$ if and only if $\langle\lambda-\mu,\cV\rangle \geq 0$.} such that: \begin{itemize}
        \item For each $\lambda\in X^*(T_X)$, \begin{equation} F_{\lambda}/F_{<\lambda}\cong 
        \left\{ 
        \begin{aligned}
             \nabla_{\lambda} &, \lambda\in \frc_X \\
             0&, \lambda\notin \frc_X
        \end{aligned} 
        \right.  \end{equation} Here, $\nabla_{\lambda}\in\Rep(G)$ is the costandard object\footnote{In the notation of \cite{jantzen2003representations}, one has $\nabla_{\lambda}=H^0(\lambda)$.} with highest weight $\lambda$;
        \item The multiplication on $\cO(X)$ preserves the filtration $\{F_{\lambda}\}_{\lambda\in X^*(T_X)}$, that is, it induces $F_{\lambda}\otimes F_{\mu}\to F_{\lambda+\mu}$ for $\lambda,\mu\in X^*(T_X)$.
    \end{itemize}
\end{assumption}
Under this assumption, the entire argument in characteristic zero case can be modified to work using the filtration $\{F_{\lambda}\sub \cO(X)\}_{\lambda\in X^*(T_X)}$.

Here are some examples that Assumption \ref{assumption:goodfil} is satisfied:

\begin{example}[Group case]\label{eg:groupcasegoodfil}
    When $G=H\times H$ and $X=H=H\backslash H\times H$ for a split reductive group $H$, by \cite[\S II.4.20]{jantzen2003representations}, we know that Assumption \ref{assumption:goodfil} is satisfied. Note that the multiplicativity of the filtration is automatic since $X$ is wavefront.\footnote{A spherical variety $X$ is called \emph{wavefront} if $X_*(T)\to X_*(T_X)$ maps the dominant cone in $X_*(T)_{\QQ}$ onto $-\cV\sub X_*(T_X)_{\QQ}$.}
\end{example}

\begin{example}[Rankin--Selberg case]\label{eg:rsgoodfil}
    When $G=\GL_n\times\GL_{n-1}$ and $X=\GL_n=\GL_{n-1}\backslash \GL_n\times\GL_{n-1}$, by applying \cite[\S II.4.24]{jantzen2003representations}, one knows that $\GL_{n-1}\sub\GL_n$ is a good pair (also called a Donkin pair). Therefore, Assumption \ref{assumption:goodfil} is satisfied. The multiplicativity of the filtration is automatic since $X$ is wavefront.
\end{example}

\begin{example}[Symmetric varieties]
    When $X$ is a symmetric variety and $F$ is algebraically closed, by \cite[Theorem\,2]{bao2024coordinateringssymmetricspaces}, one knows that Assumption \ref{assumption:goodfil} is satisfied when $\mathrm{char} F\neq 2$.
\end{example}

\subsubsection{Relative compactification}\label{sec:compactifiedmoduli}
In this section, we assume Assumption \ref{assumption:goodfil} when $\mathrm{char} F >0$. Moreover, we assume that $X=H \backslash G$ is $G$-homogeneous. In this case, we have $\sX^{\circ}\sub\sX^{\bullet}$.

Define \[\overline{\Bun}_G^X:=\Map(C,\sX^{\bullet}/G\sub \sX/G)/T_X\] \[\cB:=\Map(C,\overline{T_{X,\rmss}})/T_X.\] Note that $\overline{\Bun}_G^X$ contains an open substack \[\Bun_G^{X}/\cZ(X)^{\circ}=\Map(C,\sX^{\circ}/G)/T_X\] and $\cB$ contains an open substack \[*/\cZ(X)^{\circ}=\Map(C,T_{X,\rmss})/T_X.\]

We have a Cartesian diagram \begin{equation}\label{cd:compactificationfundcart}\begin{tikzcd}
    \Bun_G^{X}/\cZ(X)^{\circ} \ar[r]\ar[d] & \overline{\Bun}_G^X \ar[d] \\
    */\cZ(X)^{\circ}\ar[r] & \cB
    \end{tikzcd}
\end{equation}

Note that there is a natural map $\overline{\pi}:\overline{\Bun}^X_G\to\Bun_G$.

\begin{prop}\label{prop:relcompact}
    The map $\overline{\pi}:\overline{\Bun}^X_G\to\Bun_G$ is representable in proper algebraic spaces.
\end{prop}

\begin{proof}
Consider $\overline{\Bun}_{G\times T_X}^X:=\Map(C,\sX^{\bullet}/G\times T_X \sub \sX/G\times T_X)$. We have a Cartesian diagram \[\begin{tikzcd}
    \overline{\Bun}_{G}^X \ar[r] \ar[d] & \overline{\Bun}_{G\times T_X}^X \ar[d] \\
    */T_X \ar[r] & \Bun_{T_X}
\end{tikzcd}
\]
where the map $*/T_X\to\Bun_{T_X}$ is induced by the trivial $T_X$-bundle over $C$. This implies that $\overline{\Bun}_{G}^X \to \overline{\Bun}_{G\times T_X}^X$ is a closed immersion. Therefore, we only need to show that each the natural map \[\overline{\pi}_{T_X}:\overline{\Bun}_{G\times T_X}^X\to \Bun_G\] is representable in proper algebraic spaces after restricting to connected components of $\Bun_{T_X}$. 

We use $\xi\in C$ to denote the generic point of the curve $C$. We first show that $\overline{\pi}_{T_X}:\overline{\Bun}_{G\times T_X}^X\to \Bun_G$ is representable in algebraic spaces. By \cite[0DSL]{stacks-project}, we only need to check that for each $x=\cF_G\in\Bun_G(\overline{F})$, the fiber $\overline{\pi}_{T_X}^{-1}(x)$ has trivial stabilizers. In fact, for $y=(\cF_G,\cF_{T_X},s:\cF_G\times_{C}\cF_{T_X}\to\sX)\in \overline{\pi}_{T_X}^{-1}(x)(\overline{F})$, the stabilizer of $y$ is the subgroup of $T_X$ stabilizing the image of $s$, hence is contained in the stabilizer of $s(\widetilde{\xi})\in\sX^{\bullet}(\xi)$ for any $\widetilde{\xi}\in (\cF_G\times_{C}\cF_{T_X})(\xi)$ projecting to $\xi$. Note that $T_X$ acts freely on $\sX^{\bullet}$. It follows that the stabilizer above must be trivial.

Then we show that the map $\overline{\pi}_{T_X}$ is proper. Since $\overline{\pi}_{T_X}$ is of finite type (after restricting to a connected component of $\Bun_{T_X}$) and quasi-separated, using the valuative criterion for properness \cite[0CLY]{stacks-project}, we only need to show that for any discrete valuation ring $R$ with $D=\Spec R$ and $D^{\circ}=\Spec\mathrm{Frac}(R)$, any diagram \[\begin{tikzcd}
    D^\circ \ar[r] \ar[d] & \overline{\Bun}_{G\times T_X}^X \ar[d] \\
    D \ar[r] \ar[ur, dashed] & \Bun_G
\end{tikzcd}\] admits a unique (and up to a unique isomorphism, same for below) dashed arrow making the diagram commutative. Equivalently, this is to say that given any $G$-torsor $a_G:\cF_G\to D\times C$ and a $G$-equivariant map $a_{D^{\circ}}:\cF_G|_{D^{\circ}\times C}\to \sX/T_X$ whose restriction $a_{D^{\circ}}|_{D^{\circ}\times\xi}:\cF_G|_{D^{\circ}\times\xi}\to \sX/T_X$ has its image lying in $\sX^{\bullet}/T_X\sub \sX/T_X$, there exists a unique map $a:\cF_G \to \sX/T_X$ extending $a_{D^{\circ}}$ and whose restriction $a|_{D\times\xi}:\cF_{G}|_{D\times\xi}\to \sX/T_X$ has image lying in $\sX^{\bullet}/T_X\sub \sX/T_X$. Note that $\sX^{\bullet}/T_X$ is a proper algebraic variety by Lemma \ref{lem:Xcomp}, by valuative criterion for properness, we know that there exists a unique map $a_{\xi}:\cF_G|_{D\times \xi}\to \sX^{\bullet}/T_X$ extending $a_{D^{\circ}}|_{D^{\circ}\times\xi}$. This implies that there exists a unique map $a_{2}:\cF_G|_{D\times\xi \cup D^{\circ}\times C}\to \sX/T_X$ extending $a_{D^{\circ}}$ and whose restriction to $\cF_G|_{D\times\xi}$ has its image lying in $\sX^{\bullet}/T_X\sub\sX/T_X$. Since $\cF_G$ is normal, $\cF_G|_{D\times\xi \cup D^{\circ}\times C}\sub \cF_G$ has codimension two complement, and $\sX$ is affine, we know that there exists a unique desired map $a:\cF_G\to \sX/T_X$.

\end{proof}

\begin{remark}
    When $G=H\times H$ and $X=H$ for a split reductive group $H$ of adjoint type, we obtain the well-known fact that the relative compactification $\overline{\Bun}_H\to \Bun_H\times\Bun_H$ is proper. This special case is proved in \cite[\S A.1]{finkelberg2020drinfeld}. Our proof simplifies and generalizes the proof in \textit{loc.cit}.
\end{remark}

\section{Geometric trace and special cycle classes}\label{sec:geo=cycle}

In this section, we relate the minuscule homogeneous special cycle classes in \S\ref{sec:intro:mincycleclass} and the diagonal cycle classes in \S\ref{sec:intro:diagcycleclass} to the geometric Shtuka construction (introduced in \S\ref{sec:geoshtcons}) of the corresponding cohomological correspondences. The main results in this section are Theorem \ref{thm:mingeo=cyc} and Theorem \ref{thm:diaggeo=cyc}.
\begin{itemize}
    \item In \S\ref{sec:genformaffinehomogeneous}, we prove a general theorem which will be used in the proof of the main results of this section.
    \item In \S\ref{sec:minsc} and \S\ref{sec:diagsc}, we formulate and prove the main results in this section.
\end{itemize}

\subsection{General formalism in affine homogeneous case}\label{sec:genformaffinehomogeneous}
Given an affine homogeneous spherical variety $X$. Consider the map of correspondences \begin{equation}\label{cd:hkpush}\begin{tikzcd}
    \Bun_G^X\times C^I \ar[d, "\pi"] & \Hk_{G,I}^X \ar[l, "\hl^X_I"'] \ar[r, "\hr^X_I"] \ar[d, "\pi_{\Hk,I}"] & \Bun_G^X\times C^I \ar[d, "\pi"] \\
    \Bun_G\times C^I & \Hk_{G,I} \ar[l, "\hl_I"'] \ar[r, "\hr_I"] & \Bun_G\times C^I
\end{tikzcd}.\end{equation} We fix dominant coweight $\lambda_I\in X_*(T)_+^I$ and kernel sheaf $\cK\in\Shv(\Hk_{G,I})$ supported on $\Hk_{G,\lambda_I}$. 

\begin{thm}\label{thm:geo=push}
    Assume Assumption \ref{assumption:goodfil} is satisfied. For any $\cF\in\Shv(\Bun_G^X)_{\rmc}$, the diagram
    \[\begin{tikzcd}\Cor_{\Hk_{G,I}^X,\cK|_{\Hk_{G,I}^X}}(\cF\boxtimes\uk_{C^I},\cF\boxtimes\uk_{C^I}) \ar[r, "\pi_{\Hk,I,!}"] \ar[d, "\tr_{\Sht,C^I}"] & \Cor_{\Hk_{G,I},\cK}(\pi_!\cF\boxtimes\uk_{C^I},\pi_!\cF\boxtimes\uk_{C^I}) \ar[d, "\tr_{\Sht,C^I}"] \\
    H_0^{BM}(\Sht_{G,I}^X/C^I,\cK|_{\Sht_{G,I}^X}) \ar[r, "\pi_{\Sht,I,!}"] & H_0^{BM}(\Sht_{G,I}/C^I,\cK_{\Sht_{G,I}})\end{tikzcd}\] is commutative.
\end{thm}

\begin{proof}
    Since $\cK$ is supported on $\Hk_{G,\lambda_I}\sub \Hk_{G,I}$, we can restrict to the substack $\Hk_{G,\lambda_I}^X\sub \Hk_{G,I}^X$ and $\Hk_{G,\lambda_I}\sub \Hk_{G,I}$, hence all the stacks involved will be algebraic stacks. We will use the same notations for maps restricted to these substacks. By Corollary \ref{cor:trshtpush}, we only need to show that $\pi_{\Hk,I}$ is a composition of maps of correspondences admitting compactification and check the goodness in each step. 
    
    For the compactification, we can factor the map of correspondence \eqref{cd:hkpush} as \begin{equation}
        \begin{tikzcd}
            \Bun_G^X\times C^I \ar[d, "\pi_1\times\id"] & \Hk_{G,\lambda_I}^X \ar[l] \ar[r] \ar[d, "\pi_{1,\Hk,I}"] & \Bun_G^X\times C^I \ar[d, "\pi_1\times\id"] \\
            \Bun_G^X/\cZ(X)^{\circ}\times C^I \ar[d, "\pi_2\times\id"] & \Hk_{G,\lambda_I}^X/\cZ(X)^{\circ} \ar[l] \ar[r] \ar[d, "\pi_{2,\Hk,I}"] & \Bun_G^X/\cZ(X)^{\circ}\times C^I \ar[d, "\pi_2\times\id"] \\
    \Bun_G\times C & \Hk_{G,\lambda_I} \ar[l] \ar[r] & \Bun_G\times C^I
        \end{tikzcd}.
    \end{equation}
    We only need to show that both $\pi_{1,\Hk,I}$ and $\pi_{2,\Hk,I}$ admit compactification. For $\pi_{1,\Hk,I}$, one can choose any proper toric embedding $\cZ(X)^{\circ}\sub W$ and consider \[\begin{tikzcd}
        \Bun_G^X\times C^I \ar[d, "j_1\times\id"] & \Hk_{G,\lambda_I}^X \ar[l] \ar[r] \ar[d, "j_{1,\Hk,I}"] & \Bun_G^X\times C^I \ar[d, "j_1\times\id"] \\
         (\Bun_G^X\times W)/\cZ(X)^{\circ}\times C^I \ar[d, "\overline{\pi}_1\times\id"] & (\Hk_{G,\lambda_I}^X\times W)/\cZ(X)^{\circ} \ar[l] \ar[r] \ar[d, "\overline{\pi}_{1,\Hk,I}"] & (\Bun_G^X\times W)/\cZ(X)^{\circ}\times C^I \ar[d, "\overline{\pi}_1\times\id"] \\
            \Bun_G^X/\cZ(X)^{\circ} \times C^I & \Hk_{G,\lambda_I}^X/\cZ(X)^{\circ} \ar[l] \ar[r]  & \Bun_G^X/\cZ(X)^{\circ}\times C^I 
    \end{tikzcd}\] in which $\cZ(X)^{\circ}$ acts diagonally on each stack in the middle row. Here, each upper vertical map is induced by the inclusion map $*=\cZ(X)^{\circ}/\cZ(X)^{\circ}\sub W/\cZ(X)^{\circ}$, and each lower vertical map is the projection to the first factor. Note that the entire middle correspondence admits a map to $W/\cZ(X)^{\circ}$ and the boundary locus is stable because it is the fiber over $\partial W/\cZ(X)^{\circ}\sub W/\cZ(X)^{\circ}$. This shows that the factorization above satisfies the condition in Definition \ref{def:corrcompact} and $\pi_{1,\Hk,I}$ admits compactification. 

    For the map $\pi_{2,\Hk,I}$, we can consider \[\begin{tikzcd}
    \Bun_G^X/\cZ(X)^{\circ} \times C^I \ar[d, "j\times\id"] & \Hk_{G,\lambda_I}^X/\cZ(X)^{\circ} \ar[l] \ar[r] \ar[d, "j_{\Hk,I}"] & \Bun_G^X/\cZ(X)^{\circ}\times C^I \ar[d, "j\times\id"] \\
    \overline{\Bun}_G^X\times C^I \ar[d, "\overline{\pi}\times\id"] & \overline{\Hk}_{G,\lambda_I}^X \ar[l] \ar[r] \ar[d, "\overline{\pi}_{\Hk,I}"] & \overline{\Bun}_G^X\times C^I \ar[d, "\overline{\pi}\times\id"] \\
    \Bun_G\times C^I & \Hk_{G,\lambda_I} \ar[l] \ar[r] & \Bun_G\times C^I
    \end{tikzcd}.\] Here the left and right columns both come from the relative compactification $\Bun_G^X/\cZ(X)^{\circ}\xrightarrow{j}\overline{\Bun}_G^X\xrightarrow{\overline{\pi}}\Bun_G$ introduced in \S\ref{sec:compactifiedmoduli} (whose existence relies on Assumption \ref{assumption:goodfil}). The middle column is defined as follows: recall that $\overline{\Bun}_G^X$ can be identified with (the quotient by $\cZ(X)^{\circ}$ of) the moduli of $(\cE_G,s)$ where $\cE_G$ is a $G$-bundle on $C$ and $s:\cE_G\to\sX$ is a $G$-equivariant map which generically has its image lying in $\sX^{\bullet}\sub \sX$. We take $\overline{\Hk}_{G,\lambda_I}^X$ to be the closed substack of $\overline{\Bun}_G^X\times_{\Bun_G}\Hk_{G, \lambda_I}$ defined as (the quotient by $\cZ(X)^{\circ}$ of) the moduli of tuples $(\cE_{G,1},\cE_{G,2},c_I,a,s)$ where $c_I\in C^I$, $a:\cE_{G,1}|_{C-\{c_I\}}\isom \cE_{G,2}|_{C-\{c_I\}}$ is an isomorphism of $G$-bundles over $C-\{c_I\}$ with its pole at $c_I$ bounded by $\lambda_I$, $s:\cE_{G,2}\to\sX$ is a $G$-equivariant map generically has its image lying in $\sX^{\bullet}\sub\sX$ such that the composition $s\circ a:\cE_{G,1}|_{C-\{c_I\}}\to \sX$ can be extended to the entire $\cE_{G,1}$. From the definition, it is clear that the middle column, and hence the diagram above, is defined.

    The fact that the diagram above satisfies the conditions in Definition \ref{def:corrcompact} follows from Proposition \ref{prop:relcompact} and the diagram below, in which each square is Cartesian \[\begin{tikzcd}
        \Hk_{G,\lambda_I}^X/\cZ(X)^{\circ} \ar[d]\ar[r] & \overline{\Hk}_{G,\lambda_I}^X\ar[d] \\
        \Bun_G^X/\cZ(X)^{\circ}\ar[d]\ar[r] & \overline{\Bun}_G^X\ar[d] \\
        */\cZ(X)^{\circ}\ar[r] & \cB
        \end{tikzcd}.\] Here, the lower square is \eqref{cd:compactificationfundcart}, the upper vertical map can be either the left or right Hecke map.

    The goodness in each step of compactification follows from the fact that $\pi_{1,!}$, $j_{1,!}$, and $j_!$ preserve constructibility. 
\end{proof}

\subsection{Minuscule homogeneous special cycles classes}\label{sec:minsc}
In this section, assume we are in the setting of \S\ref{sec:intro:mincycleclass}.

Consider the diagram \eqref{diag:mincorfund} and take $X=H \backslash G$, we have the following result comparing geometric trace with special cycle class:
\begin{thm}\label{thm:mingeo=cyc}
    Assuming Assumption \ref{assumption:goodfil}, we have \[\pi_{\Sht,I,!}[\Sht_{H,\lambda_{H,I}}/C^I]=\tr_{\Sht,C^I}(\pi_{\Hk,I,!}[\Hk_{H,\lambda_{H,I}}/\Bun_H\times C^I])\in H_{-d_{\lambda_I}+2d_{\lambda_{H,I}}}^{BM}(\Sht_{G,I}/C^I,\IC_{V^I}|_{\Sht_{G,I}}).\]
\end{thm}

\begin{proof}
    Take $\cK=\IC_{V^I}|_{\Sht_{G,I}}$, $\cF=\uk_{\Bun_H}$, and \begin{equation}\label{eq:geo=cycle:mincc}\frc=\frc_{\lambda_{H,I}}=[\Hk_{H,\lambda_{H,I}}/\Bun_H\times C^I]\in\Cor_{\Hk_{H,\lambda_{H,I}},\uk}(\uk_{\Bun_H}\boxtimes\uk_{C^I},\uk_{\Bun_H}\boxtimes\uk_{C^I}\langle -2d_{\lambda_{H,I}}\rangle )\end{equation} in Theorem  \ref{thm:geo=push}, we are reduced to show \[[\Sht_{H,\lambda_{H,I}}/C^I]=\tr_{\Sht,C^I}([\Hk_{H,\lambda_{H,I}}/\Bun_H\times C^I])\in H_{2d_{\lambda_{H,I}}}^{BM}(\Sht_{H,\lambda_{H,I}}/C^I).\] This follows directly from Theorem  \ref{thm:trpull}.
\end{proof}

\subsection{Diagonal cycle classes}\label{sec:diagsc}

In this section, assume we are in the setting of \S\ref{sec:intro:diagcycleclass} and take $G=H\times H$ and $X=H$.
\begin{thm}\label{thm:diaggeo=cyc}
    We have \[\langle-,-\rangle_{\lambda_{H,I}}=\tr_{\Sht,C^I}(\D_{\Hk,I,!}[\Hk_{H, \lambda_{H,I}}/\Bun_H\times C^I]) \in \Hom^0((l_{H,I,!}(\IC_{V_{\lambda_{H,I}}}|_{\Sht_{H,\lambda_{H,I}}}))^{\otimes 2}, \uk_{C^I}).\]
\end{thm}

\begin{proof}
    By Example \ref{eg:groupcasegoodfil}, Assumption \ref{assumption:goodfil} is satisfied. Therefore, we can apply Theorem \ref{thm:geo=push}. Take $\cK=\IC_{\lambda_I}\in\Shv(\Hk_{G,I})$, $\cF=\uk_{\Bun_H}$, and \begin{equation}\label{eq:geo=cycle:diagcc}\frc=[\Hk_{H, \lambda_{H,I}}/\Bun_H\times C^I]\in\Cor_{\Hk_{H,I},\IC_{V_{\lambda_{H,I}}}^{\otimes 2}}(\uk_{\Bun_H\times C^I},\uk_{\Bun_H\times C^I})\end{equation} in Theorem \ref{thm:geo=push}, we are reduced to show \[[\Sht_{H,\lambda_{H,I}}/C^I]=\tr_{\Sht,C^I}([\Hk_{H,\lambda_{H,I}}/\Bun_H\times C^I])\in H_0^{BM}(\Sht_{H,\lambda_{H,I}}/C^I,\IC_{V_{\lambda_{H,I}}}^{\otimes 2}|_{\Sht_{H,\lambda_{H,I}}})).\] Note that restriction along the open Schubert cell $\Sht_{H,\lambda_{H,I}}^{\circ}\sub\Sht_{H, \lambda_{H,I}}$ gives an isomorphism \[H_0^{BM}(\Sht_{H,\lambda_{H,I}}/C^I,\IC_{V_{\lambda_{H,I}}}^{\otimes 2}|_{\Sht_{H,\lambda_{H,I}}})\cong H_{2d_{\lambda_{H,I}}}^{BM}(\Sht_{H,\lambda_{H,I}}^{\circ}/C).\] Applying Theorem  \ref{thm:trpull} for the open immersion, we are reduced to show \[[\Sht_{H,\lambda_{H,I}}^{\circ}/C^I]=\tr_{\Sht,C^I}([\Hk_{H,\lambda_{H,I}}^{\circ}/\Bun_H\times C^I])\in H_{2d_{\lambda_{H,I}}}^{BM}(\Sht_{H,\lambda_{H,I}}^{\circ}/C^I)\] which again follows from Theorem  \ref{thm:trpull}.
\end{proof}

\section{Categorical trace and geometric trace} \label{sec:cat=geo}

In this section, we interpret the geometric Shtuka construction of a special cohomological correspondence as a categorical trace. The main result in this section is Theorem \ref{thm:geo=cat}.
\begin{itemize}
    \item In \S\ref{sec:geo=cat:gl}, we review some basic facts about the geometric Langlands conjecture.
    \item In \S\ref{sec:geo=cat:ltserre}, we review the interpretation of Shtuka cohomology as a categorical trace.
    \item In \S\ref{sec:geo=cat:duality}, we develop some tools to compute the functoriality of categorical trace in our setting.
    \item In \S\ref{sec:geo=cat:mainresult}, we formulate and prove the main result in this section.
\end{itemize}

\subsection{Recollections on geometric Langlands}\label{sec:geo=cat:gl}
Consider the natural embedding $\iota:\Shv_{\Nilp}(\Bun_G)\to\Shv(\Bun_G)$, it admits a right adjoint $\iota_{\rmR}:\Shv(\Bun_G)\to\Shv_{\Nilp}(\Bun_G)$. Also, consider $\iota_2:\Shv_{\Nilp}(\Bun_G^2)\to\Shv(\Bun_G^2)$ and its right adjoint $\iota_{2,\rmR}:\Shv(\Bun_G^2)\to\Shv_{\Nilp}(\Bun_G^2)$. We sometimes omit the functor $\iota$ (resp. $\iota_2$) and regard $\Shv_{\Nilp}(\Bun_G)$ (resp. $\Shv_{\Nilp}(\Bun_G^2)$) as a subcategory of $\Shv(\Bun_G^2)$ (resp. $\Shv(\Bun_G^2)$).

One has the following fundamental result:
\begin{thm}[\cite{GR}]\label{thm:geo=cat:spectralprojector}
    The functor $\iota_{\rmR}:\Shv(\Bun_G)\to\Shv_{\Nilp}(\Bun_G)$ coincides with the Beilinson's spectral projector $\mathsf{P}$ defined in \cite[\S13.4.4]{arinkin2022stacklocalsystemsrestricted}. In particular, $\iota_{\rmR}$ is continuous.
\end{thm}

Consider $\D_!\uk_{\Bun_G}\in \Shv(\Bun_G^2)$ where $\D:\Bun_G\to\Bun_G^2$ is the diagonal map and define \[\rmu:=\iota_R(\D_!\uk_{\Bun_G})\in\Shv_{\Nilp}(\Bun_G^2).\] Define \[\ev:=\Gamma_c\circ\D^*\circ\iota_2:\Shv_{\Nilp}(\Bun_G^2)\to\Vect.\] By \cite[Theorem\,16.3.3]{arinkin2022stacklocalsystemsrestricted}, the exterior tensor product functor \[\boxtimes:\Shv_{\Nilp}(\Bun_G)^{\otimes 2}\isom\Shv_{\Nilp}(\Bun_G^2)\] is an equivalence of categories. In the following, we usually omit the functor $\boxtimes$ and do not distinguish detween $\Shv_{\Nilp}(\Bun_G)^{\otimes 2}$ and $\Shv_{\Nilp}(\Bun_G^2)$. Also, by \cite[Theorem\,F.9.7]{arinkin2022stacklocalsystemsrestricted}, for any finite set $I$, the exterior product \[\boxtimes:\Shv_{\Nilp}(\Bun_G)\otimes\QLisse(C^I)\to\Shv_{\Nilp}(\Bun_G\times C^I)\] is an equivalence of categories. Here we abbreviate $\Shv_{\Nilp}(\Bun_G\times C^I):=\Shv_{\Nilp\times\{0\}}(\Bun_G\times C^I)$. From now on, we do not distinguish between $\Shv_{\Nilp}(\Bun_G)\otimes\QLisse(C^I)$ and $\Shv_{\Nilp}(\Bun_G\times C^I)$ and often omit the exterior product.

Recall the following result from \cite[Theorem\,3.2.2]{arinkin2022dualityautomorphicsheavesnilpotent}:

\begin{thm}
    The category $\Shv_{\Nilp}(\Bun_G)$ admits a self-duality with unit \[\rmu\in \Shv_{\Nilp}(\Bun_G)^{\otimes 2}\] and counit \[\ev:\Shv_{\Nilp}(\Bun_G)^{\otimes 2}\to\Vect.\] One of the adjunction map is given by \[\alpha:(\id\otimes\ev)(\rmu\otimes-)=\pr_{1,!}\circ(\id\times\D)^*(\iota_2\circ\iota_{2,\rmR}(\D_{!}\uk_{\Bun_G})\boxtimes-)\to \pr_{1,!}\circ(\id\times\D)^*((\D_{!}\uk_{\Bun_G})\boxtimes-)\isom \id\] where the first identity follows from definition, the second natural transformation comes from the adjunction map $\iota_2\circ\iota_{2,\rmR}\to\id$, the third isomorphism is given by base change isomorphisms and projection formulas. Another adjunction map $\beta:(\ev\otimes\id)(-\otimes\rmu)\to\id$ admits a similar description.
\end{thm}

Recall the following properties of Hecke operators introduced in \S\ref{sec:intro:not:geometricsetup}:

\begin{lemma}
\begin{enumerate}
\item The right adjoint of $T_{V^I}:\Shv(\Bun_G\times C^I)\to\Shv(\Bun_G\times C^I)$ is \[T_{c^*(V^{I})}\cong\lh_*(\rh^!(-)\otimes^!\IC_{V^I})\cong \lh_!(\rh^*(-)\otimes\IC_{V^I}):\Shv(\Bun_G\times C^I)\to\Shv(\Bun_G\times C^I).\] Here $c^I\in\End(\Gc^I)$ is the Cartan involution.
\item There is a canonical isomorphism of functors $l_{I,!}(-\otimes T_{V^I}(-))\cong l_{I,!}(T_{c^*(V^{I})}(-)\otimes -)$.
\item The functor $T_{V^I}$ preserves the full-subcategory $\Shv_{\Nilp}(\Bun_G\times C^I)\sub\Shv(\Bun_G\times C^I)$, hence, gives a functor $T_{V^I}:\Shv_{\Nilp}(\Bun_G\times C^I)\to \Shv_{\Nilp}(\Bun_G\times C^I)$.
\end{enumerate}
\end{lemma}
\begin{proof}
    The first is immediate from the usual six-functor formalism (see \cite[1.1.5]{arinkin2022dualityautomorphicsheavesnilpotent}). The second follows directly from projection formulas (see \cite[Lemma\,3.4.8]{arinkin2022dualityautomorphicsheavesnilpotent}). The third is \cite[Theorem\,14.2.4]{arinkin2022stacklocalsystemsrestricted}.
\end{proof}

\subsection{Shtuka cohomology as a categorical trace} \label{sec:geo=cat:ltserre}
These is a natural map \begin{equation}\label{eq:geo=cat:ltserre}
    \tr_{\QLisse(C^I)}((\Frob\times\id)_!\circ T_{V^I},\Shv_{\Nilp}(\Bun_G\times C^I))\xrightarrow{\LT^{\Serre}} l_{I,!}(\IC_{V^I}|_{\Sht_{G,I}})
\end{equation} defined as  \[\begin{split}
    &\tr_{\QLisse(C^I)}((\Frob\times\id)_!\circ T_{V^I},\Shv_{\Nilp}(\Bun_G\times C^I)) \\ \cong& (\ev\otimes\id )\circ (\Frob_{\Bun_G}\times\id)_!\circ T_{V^I\boxtimes\triv}(\rmu\otimes\uk_{C^I} ) \\ =&l_{I,!}\circ (\Frob_{\Bun_G}\times\id)_!\circ T_{V^I\boxtimes\triv}(\iota\circ\iota_{\rmR}(\D_!\uk_{\Bun_G})\boxtimes\uk_{C^I}) \\
    \to& l_{I,!}\circ (\Frob_{\Bun_G}\times\id)_!\circ T_{V^I\boxtimes\triv}((\D_!\uk_{\Bun_G})\boxtimes\uk_{C^I}) \\
    \cong & l_{I,!}(\IC_{V^I}|_{\Sht_{G,I}})
\end{split}\] in which the third map uses the adjunction $\iota\circ\iota_{\rmR}\to\id$.

By \cite[Theorem\,4.1.2, Theorem\,5.5.6]{arinkin2022automorphicfunctionstracefrobenius}, we have \begin{thm}\label{thm:geo=cat:ltserre}
    The map \eqref{eq:geo=cat:ltserre} is an isomorphism.
\end{thm}

\subsection{Duality of functors}\label{sec:geo=cat:duality}
We would like to apply functoriality of trace construction in \S\ref{sec:cattrace} to functors \[\int_{\cP,\Nilp}:=\Gamma_c(\cP\otimes-):\Shv_{\Nilp}(\Bun_G)\to \Vect\] for some $\cP\in\Shv(\Bun_G)_{\rmc}$. For this purpose, we need to study its right adjoint $\int_{\cP,\Nilp,\rmR}$ and dual functor $\int_{\cP,\Nilp,\rmR}^{\vee}$.

We first study the version without nilpotent singular support. For each constructible sheaf \[\cP\in\Shv(\Bun_G)_{\rmc},\] consider the functor \[\int_{\cP}:=\Gamma_c(\cP\otimes-):\Shv(\Bun_G)\to\Vect.\]

\begin{lemma}
    The functor $\int_{\cP}$ preserves compact objects, hence, it admits a continuous right adjoint.
\end{lemma}
\begin{proof}
    By \cite[Proposition\,F.4.7]{arinkin2022stacklocalsystemsrestricted} the category $\Shv(\Bun_G)$ is generated by compact objects of the form $a_!\cF_S$ where $S$ is an affine scheme equipped with a map $a:S\to\Bun_G$ of finite type, and $\cF_S\in\Shv(S)_{\rmc}$ is a constructible sheaf on $S$. Since $\int_{\cP}a_!\cF_S\cong\Gamma_c(\cF_S\otimes a^*\cP)$ is a perfect complex, we know $\int_{\cP}$ preserves compact objects.
\end{proof}

The right adjoint of $\int_{\cP}$ admits the following explicit description: Consider \[\int_{\cP,\rmR}:=-\otimes\DD(\cP):\Vect\to\Shv(\Bun_G),\] we have an adjoint pair $(\int_{\cP},\int_{\cP,\rmR})$. The first adjunction map is given by \begin{equation}
    \alpha_{\cP}:\int_{\cP}\circ\int_{\cP,\rmR}=\Gamma_c(-\otimes\cP\otimes\DD(\cP))\to\Gamma_c(-\otimes\om_{\Bun_G})\to \id.
\end{equation} Here, the second map is given by the natural map $\cP\otimes\DD(\cP)\to\om_{\Bun_G}$, the third arrow uses the adjunction map for the adjoint pair $(f_!,f^!)$ for $f:\Bun_G\to *$. The second adjunction map is given by \begin{equation}\label{eq:shvbeta}
    \beta_{\cP}:\id\to -\otimes(\cP\otimes^!\DD(\cP))\to (-\otimes\cP)\otimes^!\DD(\cP)\to \Gamma_c(-\otimes\cP)\otimes\DD(\cP)=\int_{\cP,\rmR}\circ\int_{\cP}.
\end{equation} Here, the first map uses the adjunction $\uk_{\Bun_G}\to\cP\otimes^!\DD(\cP)$, the second map comes from the base change map $\D^*(\id\times\D)^!\to \D^!(\D\times\id)^*$ for the Cartesian diagram \[\begin{tikzcd}
    \Bun_G \ar[r, "\D"] \ar[d, "\D"] & \Bun_G^2 \ar[d, "\id\times \D"] \\
    \Bun_G^2 \ar[r, "\D\times \id"] & \Bun_G^3
\end{tikzcd},\] the third map uses the adjunction map for the adjoint pair $(f_!,f^!)$.

By restricting to $\Shv_{\Nilp}(\Bun_G)\sub\Shv(\Bun_G)$, one gets functor \[\int_{\cP,\Nilp}=\Gamma_c(-\otimes\cP):\Shv_{\Nilp}(\Bun_G)\to\Vect\] which has continuous right adjoint \[\int_{\cP,\Nilp,\rmR}=-\otimes \iota_{\rmR}(\DD(\cP)).\] 

The left and right lax-commutative squares in \eqref{diag:trfuncfunddiag} for the functor $\int_{\cP,\Nilp}$ becomes
\begin{equation}\label{diag:Pnilpunit}
    \begin{tikzcd}
        \Vect \ar[r, "\rmu"] \ar[d, "\id"'] & \Shv_{\Nilp}(\Bun_G)^{\otimes 2} \ar[d, "\int_{\cP,\Nilp}\otimes \int_{\cP,\Nilp,\rmR}^{\vee}"] \ar[dl, Rightarrow, "\gamma_{\cP,\Nilp}"] \\
        \Vect \ar[r, "\id"'] & \Vect
    \end{tikzcd}
\end{equation}
\begin{equation}\label{diag:Pnilpcounit}
    \begin{tikzcd}
        \Shv_{\Nilp}(\Bun_G)^{\otimes 2} \ar[r, "\ev"] \ar[d, "\int_{\cP,\Nilp}\otimes \int_{\cP,\Nilp,\rmR}^{\vee}"'] & \Vect  \ar[d, "\id"] \ar[dl, Rightarrow, "\delta_{\cP,\Nilp}"] \\
        \Vect \ar[r, "\id"'] & \Vect
    \end{tikzcd}.
\end{equation}

We now introduce versions of \eqref{diag:Pnilpunit} and \eqref{diag:Pnilpcounit} without the nilpotent singular support condition, which are the left and right squares of the diagram \eqref{diag:geotracefund} for $S=*$, $A=\Bun_G$, $\cF=\cP$. Consider lax-commutative diagrams
\begin{equation}\label{diag:Punit}
    \begin{tikzcd}
        \Vect \ar[r, "\D_!\uk_{\Bun_G}"] \ar[d, "\id"'] & \Shv(\Bun_G^2) \ar[d, "\int_{\cP\boxtimes\DD(\cP)}"] \ar[dl, Rightarrow, "\gamma_{\cP}"] \\
        \Vect \ar[r, "\id"'] & \Vect
    \end{tikzcd}
\end{equation}
\begin{equation}\label{diag:Pcounit}
    \begin{tikzcd}
        \Shv(\Bun_G^2) \ar[r, "\Gamma_c\circ \D^*"] \ar[d, "\int_{\cP\boxtimes\DD(\cP)}"'] & \Vect  \ar[d, "\id"] \ar[dl, Rightarrow, "\delta_{\cP}"] \\
        \Vect \ar[r, "\id"'] & \Vect
    \end{tikzcd}.
\end{equation}
Here the first natural transformation $\gamma_{\cP}$ is defined by \[\gamma_{\cP}:\int_{\cP\boxtimes\DD(\cP)}\D_!\uk_{\Bun_G}\isom \Gamma_c(\cP\otimes\DD(\cP))\to\Gamma_c(\om_{\Bun_G})\to k\] and the second natural transformation $\delta_{\cP}$ is defined by \[\begin{split}\delta_{\cP}:\Gamma_c\circ\Delta^*(-)&\isom \Gamma_c(-\otimes\Delta_!\uk_{\Bun_G})\\ &\to\Gamma_c(-\otimes\D_!(\cP\otimes^!\DD(\cP)))\\ &\isom\Gamma_c(-\otimes\D_!\D^!(\cP\boxtimes\DD(\cP)))\\ &\to \Gamma_c(-\otimes(\cP\boxtimes\DD(\cP)))\\ &=\int_{\cP\boxtimes\DD(\cP)}\end{split}.\]

We want to relate diagrams \eqref{diag:Pnilpunit}\eqref{diag:Pnilpcounit} to \eqref{diag:Punit}\eqref{diag:Pcounit}.

\begin{lemma}
    We have a natural commutative square \begin{equation}\label{diag:shvtonilpP}
    \begin{tikzcd}
        \Shv_{\Nilp}(\Bun_G)^{\otimes 2} \ar[r, "\iota_2\circ\boxtimes"] \ar[d, "\int_{\cP,\Nilp}\otimes \int_{\cP,\Nilp,\rmR}^{\vee}"'] & \Shv(\Bun_G^2) \ar[d, "\int_{\cP\boxtimes\DD(\cP)}"] \\
        \Vect \ar[r, "\id"] & \Vect
        \end{tikzcd}
    \end{equation}
\end{lemma}
\begin{proof}
    This following directly from the description $\int_{\cP,\Nilp,\rmR}^{\vee}=\Gamma_c(-\otimes\iota\circ \iota_{\rmR}(\DD(\cP)))\isom \Gamma_c(-\otimes\DD(\cP))$ in which is second isomorphism follows from \cite[Proposition\,3.4.6]{arinkin2022dualityautomorphicsheavesnilpotent}.
\end{proof}

\begin{prop}\label{prop:fundcube}
    We have commutative cubes
    \begin{equation}\label{diag:unitfund}
        \begin{tikzcd}
            \Vect \ar[rr, "\rmu"] \ar[dr] \ar[dd] & & \Shv_{\Nilp}(\Bun_G)^{\otimes 2} \ar[dr, "\iota_2"] \ar[dd, "\int_{\cP,\Nilp}\otimes\int_{\cP,\Nilp,\rmR}^{\vee}"] &  \\
            & \Vect \ar[dd] \ar[rr, "\D_!\uk_{\Bun_G}"] & & \Shv(\Bun_G^2) \ar[dd, "\int_{\cP\boxtimes\DD(\cP)}"] \\
            \Vect \ar[dr] \ar[rr] & & \Vect \ar[dr] & \\
            & \Vect \ar[rr] & & \Vect 
        \end{tikzcd}
    \end{equation} and 
    \begin{equation}\label{diag:counitfund}
        \begin{tikzcd}
            \Shv_{\Nilp}(\Bun_G)^{\otimes 2} \ar[rr, "\ev"] \ar[dr, "\iota_2"] \ar[dd, "\int_{\cP,\Nilp}\otimes\int_{\cP,\Nilp,\rmR}^{\vee}"'] & & \Vect \ar[dr] \ar[dd] &  \\
            & \Shv(\Bun_G^2) \ar[dd, "\int_{\cP\boxtimes\DD(\cP)}"] \ar[rr, "\Gamma_c\circ\D^*"] & & \Vect \ar[dd] \\
            \Vect \ar[dr] \ar[rr] & & \Vect \ar[dr] & \\
            & \Vect \ar[rr] & & \Vect 
        \end{tikzcd}.
    \end{equation}
    In the cube \eqref{diag:unitfund}, the back face is \eqref{diag:Pnilpunit}, the front face is \eqref{diag:Punit}, the right face is \eqref{diag:shvtonilpP}, and the other faces are equipped with the obvious natural transformations. In the cube \eqref{diag:counitfund}, the back face is \eqref{diag:Pnilpcounit}, the front face is \eqref{diag:Pcounit}, the left face is \eqref{diag:shvtonilpP}, and the other faces are equipped with the obvious natural transformations.
\end{prop}
\begin{proof}
    Note that we have commutative diagram \begin{equation}\label{fundcube:1}
        \begin{tikzcd}
            \Gamma_c((\cP\boxtimes\DD(\cP))\otimes\iota_2\iota_{2,\rmR}\D_!\uk_{\Bun_G}) \ar[r] \ar[d, "\sim"] & \Gamma_c((\cP\boxtimes\DD(\cP))\otimes \D_!\uk_{\Bun_G})  \\
            \Gamma_c(\cP\otimes \pr_{1,!}\circ (\id\times\D)^*(\iota_2\iota_{2,\rmR}(\D_!\uk_{\Bun_G})\boxtimes\DD(\cP))) \ar[r] & \Gamma_c(\cP\otimes \pr_{1,!}\circ (\id\times\D)^*(\D_!\uk_{\Bun_G}\boxtimes\DD(\cP))) \ar[u, "\sim"]
        \end{tikzcd}.
    \end{equation}
    Here, the vertical maps are the natural ones only involving six-functor formalism, and both horizontal maps only use the adjunction $\iota_2\circ\iota_{2,\rmR}\to\id$.
    Unwinding definitions, the two natural transformations from the top-left-back corner to the bottom-right-front corner in the cube \eqref{diag:unitfund} can be identified with the two maps $\Gamma_c((\cP\boxtimes\DD(\cP))\otimes\iota_2\iota_{2,\rmR}\D_!\uk_{\Bun_G})\to k$ obtained by composing the two routes from the top-left corner of \eqref{fundcube:1} to the top-right corner of \eqref{fundcube:1} with the map \[\Gamma_c((\cP\boxtimes\DD(\cP))\otimes \D_!\uk_{\Bun_G})\isom \Gamma_c(\cP\otimes\DD(\cP))\to\Gamma_c(\om_{\Bun_G})\to k.\] This proves the commutativity of \eqref{diag:unitfund}. The commutativity of the second square follows from the following commutative square \[\begin{tikzcd}
        \Gamma_c(-\otimes -) \ar[r, "\sim"] \ar[d] & \Gamma_c((-\boxtimes-)\otimes\D_!\uk_{\Bun_G}) \ar[d] \\
        \Gamma_c(\Gamma_c(-\otimes\cP)\otimes\DD(\cP)\otimes-) \ar[r, "\sim"] & \Gamma_c(-\otimes\cP)\otimes\Gamma_c(-\otimes\DD(\cP))
    \end{tikzcd}.\] Here the left vertical map uses the adjunction map \eqref{eq:shvbeta} and the right vertical arrow uses the natural map $\D_!\uk_{\Bun_G}\to\cP\boxtimes\DD(\cP)$. The commutativity of this diagram is encoded in the six-functor formalism and is routine.

\end{proof}

For later use, we need a version of Proposition \ref{prop:fundcube} with legs, which we introduce now. Fix a finite set $I$, consider the map $l_I:\Bun_G\times C^I\to C^I$. Define \[\int_{\cP,I}:=l_{I,!}(-\otimes(\cP\boxtimes\uk_{C^I})):\Shv(\Bun_G\times C^I)\to \Shv(C^I)\] and \[\int_{\cP,\Nilp,I}:=\int_{\cP,I}\circ (\iota\otimes\id)=\int_{\cP,\Nilp}\otimes\id:\Shv_{\Nilp}(\Bun_G)\otimes\QLisse(C^I)\to\QLisse(C^I).\]

\begin{prop}\label{prop:fundcubeI}
    We have commutative cubes
    \begin{equation}\label{diag:unitfundI}
        \begin{tikzcd}
            \QLisse(C^I) \ar[rr, "\rmu\otimes\id"] \ar[dr] \ar[dd] & & \Shv_{\Nilp}(\Bun_G)^{\otimes 2}\otimes\QLisse(C^I) \ar[dr, "\boxtimes\circ(\iota_2\otimes\id)"] \ar[dd, "\int_{\cP,\Nilp}\otimes\int_{\cP,\Nilp,\rmR}^{\vee}\otimes\id"] &  \\
            & \Shv(C^I)\ar[dd] \ar[rr, "(\D\times\id)_!l_{I}^*"] & & \Shv(\Bun_G^2\times C^I) \ar[dd, "\int_{\cP\boxtimes\DD(\cP),I}"] \\
            \QLisse(C^I) \ar[dr] \ar[rr] & & \QLisse(C^I)  \ar[dr] & \\
            & \Shv(C^I) \ar[rr] & & \Shv(C^I)
        \end{tikzcd}
    \end{equation} and 
    \begin{equation}\label{diag:counitfundI}
        \begin{tikzcd}
            \Shv_{\Nilp}(\Bun_G)^{\otimes 2}\otimes\QLisse(C^I) \ar[rr, "\ev\otimes\id"] \ar[dr, "\boxtimes\circ(\iota_2\otimes\id)"] \ar[dd, "\int_{\cP,\Nilp}\otimes\int_{\cP,\Nilp,\rmR}^{\vee}\otimes\id"'] & & \QLisse(C^I) \ar[dr] \ar[dd] &  \\
            & \Shv(\Bun_G^2\times C^I) \ar[dd, "\int_{\cP\boxtimes\DD(\cP),I}"] \ar[rr, "l_{I,!}\circ(\D\times\id)^*"] & & \Vect \ar[dd] \\
            \QLisse(C^I) \ar[dr] \ar[rr] & & \QLisse(C^I) \ar[dr] & \\
            & \Shv(C^I) \ar[rr] & & \Shv(C^I) 
        \end{tikzcd}
    \end{equation}
    in which the natural transformations are natural generalizations of those in Proposition \ref{prop:fundcube}.
\end{prop}

The proof is parallel to the proof of Proposition \ref{prop:fundcube}.

\subsection{Main result}\label{sec:geo=cat:mainresult}

\subsubsection{Special cycle classes as geometric trace}
Given $\cP\in\Shv(\Bun_G)_{\rmc}$ and a cohomological correspondence \[\frc \in \Hom^0(\cP\boxtimes\uk_{C^I},T_{V^I}(\cP\boxtimes\uk_{C^I}))=\Hom^0(T_{c^*(V^{I})}(\cP\boxtimes\uk_{C^I}),\cP\boxtimes\uk_{C^I})=\Cor_{\Hk_{G,I},\IC_{V^I}}(\cP\boxtimes\uk_{C^I},\cP\boxtimes\uk_{C^I}),\] we can apply the second construction of geometric trace via the diagram \eqref{diag:geotracefund}, in which we take $S=C^I$, the correspondence $(A\xleftarrow{c_1} C\xrightarrow{c_0}A)=(\Bun_G\times C^I\xleftarrow{\lh_I}\Hk_{G,I}\xrightarrow{\rh_I} \Bun_G\times C^I)$, $\cK=\IC_{V^I}\in\Shv(\Hk_{G,I})$, $\cF=\cP\boxtimes\uk_{C^I}\in\Shv(\Bun_G\times C^I)$, and the cohomological correspondence $\frc^{(1)}$.

In this case, the diagram \eqref{diag:geotracefund} becomes
\begin{equation}\label{diag:shvfund}\begin{tikzcd}
    \Shv(C^I) \ar[r, "(\D\times\id)_!l_{I}^*"] \ar[d, "\id"] & \Shv(\Bun_G^2\times C^I) \ar[r, "(\Frob_{\Bun_G}\times\id)_!\circ T_{V^I}"] \ar[d, "\int_{\cP\boxtimes\DD{\cP},I}" description] \ar[dl, Rightarrow, "\gamma_{\cP,I}"] & \Shv(\Bun_G^2\times C^I) \ar[r, "l_{I,!}(\D\times\id)^*"] \ar[d, "\int_{\cP\boxtimes\DD{\cP},I}" description] \ar[dl, Rightarrow, "\eta_{\frc^{(1)}}"] & \Shv(C^I) \ar[d, "\id"] \ar[dl, Rightarrow, "\delta_{\cP,I}"] \\
    \Shv(C^I) \ar[r, "\id"] & \Shv(C^I) \ar[r, "\id"] & \Shv(C^I) \ar[r, "\id"] & \Shv(C^I)
    \end{tikzcd}.
\end{equation}

By Lemma \ref{lem:geotrace=diagtrace}, we know that the natural transformation from the upper route to the lower route from the top-left corner to the bottom-right corner of \eqref{diag:shvfund} evaluated at $\uk_{C^I}\in\Shv(C^I)$ gives \[\tr_{\Sht,C^I}(\frc)=\tr_{C^I}(\frc^{(1)})\in H_0^{BM}(\Sht_{G,I}/C^I,\IC_{V^I}|_{\Sht_{G,I}}).\]

\subsubsection{Special cycle classes as categorical trace}
Restricting the natural transformation $\eta_{\frc^{(1)}}$ in \eqref{diag:shvfund} to the full-subcategory $\Shv_{\Nilp}(\Bun_G)\otimes\QLisse(C^I)\sub\Shv(\Bun_G\times C^I)$, we get a natural transformation 
\[\eta_{\frc^{(1)},\Nilp}:\int_{\cP,I,\Nilp}\circ (\Frob\times\id)_!\circ T_{V^I}\to\int_{\cP,I,\Nilp}.\]


Since the functor $\int_{\cP,\Nilp}:\Shv_{\Nilp}(\Bun_G)\to\Vect$ admits continuous right adjoint, using the natural transformation $\eta_{\frc^{(1)},\Nilp}$ and applying the construction Definition \ref{def:funccattr}, we get a natural map \[\tr(\eta_{\frc^{(1)},\Nilp}):\tr_{\QLisse(C^I)}((\Frob\times\id)_!\circ T_{V^I},\Shv_{\Nilp}(\Bun_G\times C^I))\to \tr_{\QLisse(C^I)}(\id, \QLisse(C^I)).\] 

\subsubsection{Relating geometric trace and categorical trace}
The main result in this section is the following:
\begin{thm}\label{thm:geo=cat}
    We have a commutative square
    \[\begin{tikzcd}
        \tr_{\QLisse(C^I)}((\Frob\times\id)_!\circ T_{V^I}  \Shv_{\Nilp}(\Bun_G)\otimes\QLisse(C^I)) \ar[d, "\tr_{\QLisse(C^I)}(\eta^{(1)}_{\frc,\Nilp})"] \ar[r, "\LT^{\Serre}", "
        \sim"'] & l_{I,!}(\IC_{V^I}|_{\Sht_{G,I}}) \ar[d, "\tr_{\Sht,C^I}(\frc)"] \\
        \tr_{\QLisse(C^I)}(\id,\QLisse(C^I)) \ar[r, "\sim"] & \uk_{C^I}
    \end{tikzcd}\]
    
\end{thm}

\begin{proof}
    Using the duality on $\Shv_{\Nilp}(\Bun_G)$ with unit and counit $(\rmu,\ev)$, the element $\tr_{\QLisse(C^I)}(\eta_{\frc^{(1)},\Nilp})$ is computed as the composition of natural transformations in the diagram \begin{equation}\label{diag:nilpfund}\adjustbox{scale=.96}{%
    \begin{tikzcd}
    \QLisse(C^I) \ar[r, "\rmu\otimes\id"] \ar[d, "\id"] & \Shv_{\Nilp}(\Bun_G)^{\otimes 2}\otimes\QLisse(C^I) \ar[r, "(\Frob_{\Bun_G}\times\id)_!\circ T_{V^I}\otimes\id"] \ar[d, "\int_{\cP,I,\Nilp}\otimes\int_{\cP,I,\Nilp,\rmR}\otimes\id" description] \ar[dl, Rightarrow, "\gamma_{\cP,\Nilp}\otimes\id"] & \Shv_{\Nilp}(\Bun_G)^{\otimes 2}\otimes\QLisse(C^I) \ar[r, "\ev\otimes\id"] \ar[d, "\int_{\cP,I,\Nilp}\otimes\int_{\cP,I,\Nilp,\rmR}\otimes\id" description] \ar[dl, Rightarrow, "\eta_{\frc^{(1)},\Nilp}\otimes\id"] & \QLisse(C^I) \ar[d, "\id"] \ar[dl, Rightarrow, "\delta_{\cP,\Nilp}\otimes\id"] \\
    \QLisse(C^I) \ar[r, "\id"] & \QLisse(C^I) \ar[r, "\id"] & \QLisse(C^I) \ar[r, "\id"] & \QLisse(C^I)
    \end{tikzcd}
    }
\end{equation} in which the left and right squares are the two back squares in Proposition \ref{prop:fundcubeI}. By Proposition \ref{prop:fundcubeI}, there is a natural map from the diagram \eqref{diag:nilpfund} to \eqref{diag:shvfund} (such a map means three commutative cubes), which gives our desired identity.

\end{proof}

\subsection{Proof of main result}\label{sec:ss=catproof}
\begin{proof}[Proof of Theorem  \ref{thm:minmain} and Theorem  \ref{thm:diagmain}]
    Take $\cP=\cP_X$ in Theorem  \ref{thm:geo=cat} for the spherical variety $X$ in each case.
    Theorem  \ref{thm:minmain} follows from Theorem  \ref{thm:geo=cat} and Theorem  \ref{thm:mingeo=cyc}. Theorem  \ref{thm:diagmain} follows from Theorem  \ref{thm:geo=cat} and Theorem  \ref{thm:diaggeo=cyc}.
\end{proof}

\section{Isotypic part of special cycle classes}\label{sec:isotypicpartsc}

In this section, we study the restriction of special cycle classes on the isotypic part of the cohomology of Shtukas. 
\begin{itemize}
    \item In \S\ref{sec:isotypicpartsc:geop}, we review some basic properties of the geometric isotypic part.
    \item In \S\ref{sec:isosc1}, we study the isotypic part of special cycle classes for split semisimple groups.
    \item In \S\ref{sec:isotypicpartscc2}, we study the isotypic part of special cycle classes for split reductive groups.
    \item In \S\ref{sec:gmdcase}, we study the isotypic part of special cycle classes which have middle-dimension on the generic fiber.
    \item In \S\ref{sec:isotypicdiag}, we study the isotypic part of diagonal cycle classes.
\end{itemize}

\subsection{Isotypic part of geometric period}\label{sec:isotypicpartsc:geop}
In this section, we recall the geometric isotypic part introduced in \cite{liu2025higherperiodintegralsderivatives} and its basic properties.

For a $\Gc$-local system $\s\in\Loc_{\Gc}(k)$ and a Hecke eigensheaf $\LL_{\s}\in\Shv_{\Nilp}(\Bun_G)$, given an affine smooth $G$-variety $X$, the geometric isotypic part of the $X$-period integral is defined to be the complex $\int_{X,\Nilp}\LL_{\s}\in\Vect$, which we simply call the \emph{geometric isotypic part}.

\subsubsection{Hecke action}\label{sec:heckeaction}
The geometric isotypic part is equipped with Hecke actions, which we now recall. For $V^I\in\Rep(\Gc^I)$, given a cohomological correspondence $\frc_{V^I}\in \Cor_{\Hk_{G,I},\IC_{V^I}\langle -d_I\rangle}(\cP_X\boxtimes\uk_{C^I},\cP_X\boxtimes\uk_{C^I})$, there is an induced natural transformation constructed in \S\ref{sec:cctrnt} \[\eta_{\frc_{V^I}}:\int_{X,\Nilp,I}\circ T_{V^I}\to \int_{X,\Nilp,I}.\] Evaluating the natural transformation $\eta_{\frc_{V^I}}$ at $\LL_{\s}\boxtimes\uk_{C^I}\in\Shv_{\Nilp}(\Bun_G)\otimes\QLisse(C^I)$, we get a map \[a_{\frc_{V^I},\s}:V_{\s}^I \langle -d_I\rangle\otimes \int_{X,\Nilp}\LL_{\s}\to \uk_{C^I}\otimes \int_{X,\Nilp}\LL_{\s} .\] Since there is an obvious isomorphism \[\Hom^0(V_{\s}^I\langle -d_I\rangle\otimes\int_{X,\Nilp}\LL_{\s},\uk_{C^I}\otimes \int_{X,\Nilp}\LL_{\s} )\cong \Hom^0(\Gamma(V_{\s}^I)\langle -d_I + 2r \rangle\otimes \int_{X,\Nilp}\LL_{\s}  ,\int_{X,\Nilp}\LL_{\s}),\] we can also treat $a_{\frc_{V^I},\s}$ as a map \begin{equation}\label{eq:heckeactiononisotypicpart}
    a_{\frc_{V^I},\s}:\Gamma(V_{\s}^I)\langle -d_I + 2r \rangle\otimes \int_{X,\Nilp}\LL_{\s}  \to \int_{X,\Nilp}\LL_{\s}.
\end{equation} We say that the map $a_{\frc_{V^I},\s}$ equips the geometric isotypic part $\int_{X,\Nilp}\LL_{\s}$ with a Hecke action by $\Gamma(V_{\s}^I)\langle -d_I+2r\rangle$. Equivalently, we can write the data above as a map \[a_{\frc_{V^I},\s}:\Gamma(V_{\s}^I)\langle -d_I + 2r \rangle\to\End(\int_{X,\Nilp}\LL_{\s}).\]

One can also fix a point $c^I\in C^I(\overline{\FF}_q)$ and consider the Hecke operator with leg restricted to $c^I$ which is $T_{V^I,c^I}:\Shv_{\Nilp}(\Bun_G)\to \Shv_{\Nilp}(\Bun_G)$. In this case, one can restrict the cohomological correspondence $\frc_{V^I}$ to $\frc_{V^I,c^I}\in \Cor_{\Hk_{G,c^I},\IC_{V^I}\langle -d_I\rangle}(\cP_X,\cP_X)$. Here, we define $\Hk_{G,c^I}:=\Hk_{G,I}\times_{C^I}\{c^I\}$.  It gives a map \begin{equation}\label{eq:heckeactiononisotypicpartfixleg}
    a_{\frc_{V^I,c^I},\s}:V_{\s,c^I}^I\langle -d_I\rangle\otimes\int_{X,\Nilp}\LL_{\s}\to \int_{X,\Nilp}\LL_{\s}.
\end{equation} However, the data of \eqref{eq:heckeactiononisotypicpartfixleg} is completely contained in \eqref{eq:heckeactiononisotypicpart} as follows: Note that there is a map obtained from $(i_!,i^!)$-adjunction for $i:\{c^I\}\to C^I$ \[[c^I]:V_{\s,c^I}^I\to \Gamma(V_{\s}^I)\langle 2r\rangle.\] Then we have $a_{\frc_{V^I,c^I},\s}=a_{\frc_{V^I},\s}\circ ([c^I]\otimes\id)$.

\subsubsection{Associativity}\label{sec:associativity}
The Hecke action \eqref{eq:heckeactiononisotypicpart} enjoys associativity for composition of cohomological correspondences.

Take $I=\{1,2\cdots,r\}$. Given $V_i\in\Rep(\Gc)$ and $\frc_{V_i}\in \Cor_{\Hk_{G,\{i\}},\IC_{V_i}\langle -d_i\rangle}(\cP_X\boxtimes\uk_C,\cP_X\boxtimes\uk_C)$ for $i\in I$, consider the cohomological correspondence arising from composition \[\frc_{V^I}=\frc_{V_1}\circ \frc_{V_2}\circ\cdots\circ \frc_{V_r}\] for $V^I=\boxtimes_{i\in I}V_i$ and $d_I=\sum_{i\in I}d_i$ whose definition is as follows: Note that for each $i\in I$ there is a map \[\Cor_{\Hk_{G,\{i\}},\IC_{V_i}\langle -d_i\rangle}(\cP_X\boxtimes\uk_{C},\cP_X\boxtimes\uk_{C})\to \Cor_{\Hk_{G,I},\IC_{V_i}\langle -d_i\rangle}(\cP_X\boxtimes \uk_{C^I},\cP_X\boxtimes\uk_{C^I}).\] Here the sheaf $\IC_{V_i}\in \Shv(\Hk_{G,I})$ is understood as $\IC_{\triv\boxtimes\cdots\boxtimes V_i\boxtimes\cdots\boxtimes \triv}\in\Shv(\Hk_{G,I})$ for $\triv\boxtimes\cdots\boxtimes V_i\boxtimes\cdots\boxtimes \triv\in \Rep(\Gc^I)$. Putting this together with the usual composition of cohomological correspondences, we get a map \[\begin{split}
\circ:\prod_{i=1}^r \Cor_{\Hk_{G,\{i\}},\IC_{V_i}\langle -d_i\rangle}(\cP_X\boxtimes\uk_{C},\cP_X\boxtimes\uk_{C}) &\to \prod_{i=1}^r \Cor_{\Hk_{G,I},\IC_{V_i}\langle -d_i\rangle}(\cP_X\boxtimes \uk_{C^I},\cP_X\boxtimes\uk_{C^I}) \\
& \to \Cor_{\Hk_{G,I},\IC_{V^I}\langle -d_I\rangle}(\cP_X\boxtimes\uk_{C^I},\cP_X\boxtimes\uk_{C^I})
\end{split}.\]

In this case, we have \[a_{\frc_{V^I},\s}=a_{\frc_{V_1},\s}\circ\cdots\circ a_{\frc_{V_r},\s}\in \Hom^0(\Gamma(V_{\s}^I)\langle -d_I + 2r \rangle\otimes \int_{X,\Nilp}\LL_{\s}  ,\int_{X,\Nilp}\LL_{\s}) \] where the later is understood as the composition \[\begin{split}
    a_{\frc_{V_1},\s}\circ\cdots\circ a_{\frc_{V_r},\s}:\Gamma(V_{\s}^I)\langle -d_I+2r\rangle \otimes \int_{X,\Nilp}\LL_{\s} &= \otimes_{i=1}^r\Gamma(V_{i,\s})\langle-d_i+2\rangle  \otimes \int_{X,\Nilp}\LL_{\s}\\
    &\xrightarrow{\id \otimes a_{\frc_{V_r}}} \otimes_{i=1}^{r-1}\Gamma(V_{i,\s})\langle-d_i+2\rangle  \otimes   \int_{X,\Nilp}\LL_{\s}  \\
    &\cdots \\
    &\xrightarrow{\id \otimes a_{\frc_{V_1}}} \int_{X,\Nilp}\LL_{\s}
\end{split}.\]

\subsubsection{Commutator relation}\label{sec:commutatorrelation}
When the cohomological correspondences $\frc_{V_i}$ come from local special cohomological correspondences as defined in \cite[Definition\,4.39]{liu2025higherperiodintegralsderivatives}, the Hecke action \eqref{eq:heckeactiononisotypicpart} enjoys commutator relations coming from the commutator relations in the local Plancherel algebra. By associativity, it is enough to consider the case $I=\{1,2\}$. Assume each $\frc_{V_i}$ comes from a local special cohomological correspondence $\frc_{V_i}^{\loc}$. Moreover, assume $[\frc_{V_1}^{\loc},\frc_{V_{2}}^{\loc}]=\hbar \cdot \frc^{\loc}_{V_1\otimes V_{2}}$ for some local special cohomological correspondence $\frc^{\loc}_{V_1\otimes V_{2}}$. We use \[\frc_{V_1\otimes V_{2}}\in\Cor_{\Hk_{G,\{1\},\IC_{V_1\otimes V_{2}}}\langle -d_1-d_2+2 \rangle}(\cP_X\boxtimes\uk_{C},\cP_X\boxtimes\uk_{C})\] to denote the corresponding global special cohomological correspondence. In this case, assuming \cite[Conjecture\,4.45]{liu2025higherperiodintegralsderivatives}, we have \[a_{\frc_{V_1}}\circ a_{\frc_{V_2}}-a_{\frc_{V_2}}\circ a_{\frc_{V_1}}=a_{\frc_{V_1\otimes V_2}}\circ (\cup\otimes\id)\in \Hom^0(\Gamma(V_{1,\s})\otimes\Gamma(V_{2,\s})\langle -d_1-d_2+4 \rangle\otimes\int_{X,\Nilp}\LL_{\s},\int_{X,\Nilp}\LL_{\s})\] where $\cup:\Gamma(V_{1,\s})\otimes\Gamma(V_{2,\s})\to\Gamma((V_{1}\otimes V_2)_{\s})$ is the cup product.

\subsection{Isotypic part of special cycle classes I}\label{sec:isosc1}
From now on, assume that $\s\in\Loc_{\Gc}^{\arith}(k)$ is a $\Gc$-Weil local system and the Hecke eigensheaf $\LL_{\s}$ is equipped with a compatible Weil sheaf structure. We drop the subscript $V^I$ in $\frc_{V^I}$ and write $\frc=\frc_{V^I}$.

\subsubsection{Fake special cycle classes}

From the Hecke action \eqref{eq:heckeactiononisotypicpart}, one can construct the \emph{fake special cycle class} \[z_{\frc,\s}\in H^0(\Gamma(V_{\s}^I)\langle-d_I+2r\rangle)^*\cong \Hom^0(V_{\s}^I\langle-d_I\rangle,\uk_{C^I})\] defined as the element satisfying \begin{equation}\label{eq:fakespecialcycleclasses} z_{\frc,\s}(m)=\tr( a_{\frc,\s}(m)\circ\Frob,\int_{X,\Nilp}\LL_{\s})\end{equation} for any $m\in H^0(\Gamma(V_{\s}^I)\langle-d_I+2r\rangle)$. See \eqref{eq:intro:fakespecialcycleclasses} for an alternative but equivalent definition. Here, we are assuming the trace written above is convergent. The ideal case is the following:

\begin{assumption}\label{assumption:vectcompact}
    The complex $\int_{X,\Nilp}\LL_{\s}$ is perfect.
\end{assumption}

We will assume Assumption \ref{assumption:vectcompact} until the end of \S\ref{sec:isosc1}. Under this assumption, the fake special cycle classes admit the following interpretation via categorical trace: Consider the lax-commutative square \begin{equation}\label{diag:semisimplediag1}
    \begin{tikzcd}
        \QLisse(C^I) \ar[r, "-\otimes V_{\s}^I\langle -d_I\rangle "] \ar[d, "-\otimes \int_{X,\Nilp}\LL_{\s}"'] & \QLisse(C^I) \ar[d, "-\otimes \int_{X,\Nilp}\LL_{\s}"] \ar[dl, Rightarrow, "\eta_{a_{\frc,\s}}"] \\
        \QLisse(C^I) \ar[r, "\id"] & \QLisse(C^I)
    \end{tikzcd}
\end{equation} in which the natural transformation is defined as \[\eta_{a_{\frc,\s}}:-\otimes V_{\s}^I\langle-d_I\rangle\otimes\int_{X,\Nilp}\LL_{\s}\xrightarrow{-\otimes a_{\frc,\s}} -\otimes \uk_{C^I}\otimes\int_{X,\Nilp}\LL_{\s}.\] The natural transformation $\eta_{a_{\frc,\s}}$ admits a Frobenius twist \[\eta_{a_{\frc,\s}}^{(1)}:-\otimes V_{\s}^I\langle-d_I\rangle\otimes\int_{X,\Nilp}\LL_{\s}\xrightarrow{-\otimes  a_{\frc,\s}\circ(\id\otimes\Frob)} -\otimes \uk_{C^I}\otimes\int_{X,\Nilp}\LL_{\s}.\] Then one has a commutative diagram \begin{equation}\label{diag:semisimplediag1tr}\begin{tikzcd}\tr_{\QLisse(C^I)}(-\otimes V_{\s}^I\langle-d_I\rangle,\QLisse(C^I)) \ar[d, "\tr_{\QLisse(C^I)}(\eta_{a_{\frc},\s}^{(1)})"] \ar[r, "\sim"] & V_{\s}^I\langle-d_I\rangle \ar[d, "z_{\frc,\s}"] \\
\tr_{\QLisse(C^I)}(\id,\QLisse(C^I)) \ar[r, "\sim"] & \uk_{C^I}\end{tikzcd}.\end{equation} Here, the horizontal maps are the most obvious isomorphisms. The left vertical map is induced by the functoriality of categorical trace in Definition \ref{def:funccattr}.

\subsubsection{Fake versus real I}\label{sec:fakevsreal1}
For a cohomological correspondence \[\frc\in\Cor_{\Hk_{G,I},\IC_{V^I}\langle-d_I\rangle}(\cP_X\boxtimes\uk_{C^I},\cP_X\boxtimes\uk_{C^I})\] and $\s\in \Loc_{\Gc}^{\arith}(k)$. On the one hand, one has the fake special cycle classes \[z_{\frc,\s}\in \Hom^0(V_{\s}^I\langle-d_I\rangle,\uk_{C^I})\] whenever it is well-defined. On the other hand, one has the geometric trace \[\tr_{\Sht,C^I}(\frc)\in \Hom^0(l_{I,!}(\IC_{V^I}|_{\Sht_{G,I}})\langle -d_I\rangle,\uk_{C^I}).\] We would like to relate these two.

The most natural way of doing this is to construct a map \begin{equation}\label{eq:isotypicpart}\xi_{\s,I}:V_{\s}^I\to l_{I,!}(\IC_{V^I}|_{\Sht_{G,I}})\end{equation} which only depends on the Hecke eigensheaf $\LL_{\s}$ (in particular, independent of the $G$-variety $X$) and ask if \begin{equation} \label{eq:computationofisotypicss} z_{\frc,\s}=\tr_{\Sht,C^I}(\frc)\circ \xi_{\s,I}.\end{equation} It turns out that one can achieve this under good assumptions.

In this subsection, we will achieve this under a stronger assumption compared to Assumption \ref{assumption:vectcompact}, which is:
\begin{assumption}\label{assumption:eigenstronglycompact}
    The group $G$ is split semisimple, and the object $\LL_{\s}\in \Shv_{\Nilp}(\Bun_G)$ is compact.
\end{assumption}

\begin{remark}
    Note that since the functor $\int_{X,\Nilp}:\Shv_{\Nilp}(\Bun_G)\to \Vect$ preserves compact objects, Assumption \ref{assumption:vectcompact} is a consequence of Assumption \ref{assumption:eigenstronglycompact}.
\end{remark}

\begin{remark}
    By \cite{GR}, Assumption \ref{assumption:eigenstronglycompact} is true when $G$ is semisimple and $\s\in\Loc_{\Gc}(k)$ is (geometrically) irreducible. Unfortunately, in our application \S\ref{sec:rs}, we are interested in the case $G=\GL_n\times\GL_{n-1}$, which is not semisimple. We will loosen Assumption \ref{assumption:eigenstronglycompact} in \S\ref{sec:isotypicpartscc2}.
\end{remark}

Until the end of \S\ref{sec:fakevsreal1}, we assume Assumption \ref{assumption:eigenstronglycompact}. In this case, there is a natural choice of $\xi_{\s,I}$ which we are going to explain now.

Note that the diagram \eqref{diag:semisimplediag1} naturally factors as \begin{equation}\label{diag:semisimplediag2}\begin{tikzcd}
    \QLisse(C^I) \ar[d, "(-\otimes\LL_{\s})\otimes\id"'] \ar[r, "-\otimes V_{\s}^I\langle-d_I\rangle "] & \QLisse(C^I) \ar[d, "(-\otimes\LL_{\s})\otimes\id"] \ar[dl, Rightarrow, "\eta_{\s,I}"] \\
    \Shv_{\Nilp}(\Bun_G)\otimes\QLisse(C^I) \ar[d, "\int_{X,\Nilp}\otimes\id"'] \ar[r, "T_{V^I\langle-d_I\rangle}"] & \Shv_{\Nilp}(\Bun_G)\otimes\QLisse(C^I) \ar[d, "\int_{X,\Nilp}\otimes\id"] \ar[dl, Rightarrow, "\eta_{\frc,\Nilp}"] \\
    \QLisse(C^I) \ar[r, "\id"] & \QLisse(C^I)
    \end{tikzcd}.
\end{equation} Here, the natural transformation $\eta_{\s,I}$ is the natural map coming from the Hecke eigen-property of $\LL_{\s}$: \[\eta_{\s,I}:-\otimes(\LL_{\s}\boxtimes V_{\s}^I\langle-d_I\rangle)\cong -\otimes T_{V^I\langle-d_I\rangle}(\LL_{\s}\boxtimes\uk_{C^I}).\] One easily checks that \begin{equation}\label{eq:compositionofnt} \eta_{a_{\frc},\s}=\eta_{\frc,\Nilp}\circ \eta_{\s,I}.\end{equation}

By Assumption \ref{assumption:eigenstronglycompact}, all vertical maps in \eqref{diag:semisimplediag2} preserve compact objects. Therefore, by the formalism of Definition \ref{def:funccattr}, the diagram \eqref{diag:semisimplediag2} induces a factorization of the diagram \eqref{diag:semisimplediag1tr}:
\begin{equation}\label{diag:semisimplediag2tr}\begin{tikzcd}\tr_{\QLisse(C^I)}(-\otimes V_{\s}^I\langle-d_I\rangle,\QLisse(C^I)) \ar[d, "\tr_{\QLisse(C^I)}(\eta_{\s,I}^{(1)})"] \ar[r, "\sim"] & V_{\s}^I\langle-d_I\rangle \ar[d, "\xi_{\s,I}"] \\
\tr_{\QLisse(C^I)}((\Frob\times\id)_!\circ T_{V^I\langle-d_I\rangle},\Shv_{\Nilp}(\Bun_G)\otimes\QLisse(C^I)) \ar[r, "\LT^{\Serre}", "\sim"'] \ar[d, "\eta_{\frc,\Nilp}^{(1)}"] & l_{I,!}(\IC_{V^I}|_{\Sht_{G,I}})\langle-d_I\rangle \ar[d, "\tr_{\Sht,C^I}(
\frc
)"] \\
\tr_{\QLisse(C^I)}(\id,\QLisse(C^I)) \ar[r, "\sim"] & \uk_{C^I}\end{tikzcd}.\end{equation} Here, the upper commutative diagram can be taken as the definition of the map $\xi_{\s, I}$. The lower commutative diagram is Theorem  \ref{thm:geo=cat}. By \eqref{eq:compositionofnt}, we know that the outer square of \eqref{diag:semisimplediag2tr} is the square in \eqref{diag:semisimplediag1tr}. This gives the desired identity $z_{\frc,\s}=\tr_{\Sht,C^I}(\frc)\circ \xi_{\s,I}$.

\subsubsection{Refinement}\label{sec:isotypicrefinement}
Note that $l_{I,!}(\IC_{V^I}|_{\Sht_{G,I}})=0$ unless the central characters of different components of $V^I\in\Rep(\Gc^I)$ add up to zero, which we assume in this section. In this case, all the constructions above make sense after adding a superscript $e\in \pi_0(\Bun_G)\cong \pi_1(G)$, which means restricting to the connected component $\Bun_G^e\sub \Bun_G$. This gives us refined notions including Hecke eigensheaf on the connected component $\LL_{\s}^e=\LL_{\s}|_{\Bun_G^e}$, cohomological correspondence \[\frc^e\in\Cor_{\Hk_{G,I}^e,\IC_{V^I}\langle -d_I\rangle}(\cP_X\boxtimes\uk_{C^I},\cP_X\boxtimes\uk_{C^I}),\] Hecke action on the geometric isotypic part \[a^e_{\frc,\s}=a_{\frc^e,\s}:=V_{\s}^I\langle-d_I\rangle \otimes\int_{X,\Nilp}\LL_{\s}^e\to\int_{X,\Nilp}\LL_{\s}^e, \] isotypic part map \begin{equation}\label{eq:refinedisotypicpart}
    \xi_{\s,I}^e:V_{\s}^I\to l_{I,!}(\IC_{V^I}|_{\Sht_{G,I}^e}),
\end{equation} the fake special cycle classes \begin{equation}\label{eq:fakespecialcycleclassesref}z^e_{\frc,\s}:V_{\s}^I\langle-d_I\rangle\to \uk_{C^I}\end{equation} constructed from $a_{\frc,\s}^e$, and the identity \begin{equation} \label{eq:computationofisotypicssrefined} z^e_{\frc,\s}=\tr_{\Sht,C^I}(\frc)\circ \xi_{\s,I}^e.\end{equation}

\subsection{Isotypic part of special cycle classes II}\label{sec:isotypicpartscc2}
In this section, we will generalize the technique in \S\ref{sec:fakevsreal1} to general split reductive groups $G$. Let $Z(G)$ be the center of $G$ and $S=G/[G,G]$. We have a natural map on the Langlands dual groups $\Sc\to\Gc$. Define $\Shv_{\triv}(\Bun_S) \sub \Shv_{\Nilp}(\Bun_S)$ to be the direct summand containing the constant sheaf $\uk_{\Bun_S}$. We have an obvious map $f:\Bun_G\to\Bun_S$.

The main difference between this section and \S\ref{sec:fakevsreal1} is that we replace Assumption \ref{assumption:eigenstronglycompact} by the following which we assume throughout \S\ref{sec:isotypicpartscc2}:

\begin{assumption}\label{assumption:eigencompact}
    The functor $f^*(-)\otimes\LL_{\s}:\Shv_{\triv}(\Bun_S)\to \Shv_{\Nilp}(\Bun_G)$ preserves compact objects.
\end{assumption}

\begin{remark}
    By \cite{GR}, Assumption \ref{assumption:eigencompact} is true whenever $\s\in\Loc_{\Gc}(k)$ is (geometrically) irreducible.
\end{remark}

\subsubsection{Construction of the isotypic part}\label{sec:geoisotypic}

In this section, we are going to construct the map \eqref{eq:isotypicpart} for $V^I\in \Rep(\Gc^I)$ and $\s\in \Loc_{\Gc}^{\arith}(k)$, which only depends on the choice of a Hecke eigensheaf $\LL_{\s}\in\Shv_{\Nilp}(\Bun_G)$ with a Weil sheaf structure compatible with that of $\s$. 

Consider the functor $(-)_{\s}: \Rep(\Gc^I)\to \Shv(\Hk_{S,I})$ defined by $V^I\mapsto V^I_{\s}$. We explain the notation as follows: Note that $\Hk_{S,I}=\coprod_{\lambda_{S,I}\in X_*(S)^I}\Hk_{S,\l_{S,I}}$. When $V^I$ admits a unique central character $\lambda_{S,I}\in X^*(\Sc)^I$, we abuse the notation and denote $V^I_{\s}=l_{S,I}^*V^I_{\s}|_{\Hk_{S,\lambda_{S,I}}}$ in which $l_{S,I}:\Hk_{S,I}\to C^I$ is the map remembering only the legs. In the expression, the later $V^I_{\s}$ is the local system $V_{\s}^I\in\QLisse(C^I)$ while the former $V_{\s}^I$ is a sheaf $V_{\s}^I\in\Shv(\Hk_{S,\lambda_{S,I}})\sub \Shv(\Hk_{S,I})$. The meaning of $V_{\s}^I$ will be clear from the context. In general, $V^I$ is a direct sum of representations with different central characters. We take the direct sum of the construction above for each direct summand of $V^I$.

When $G$ is split reductive but not necessarily semisimple, it is necessary to make the refinement in \S\ref{sec:isotypicrefinement}. We will fix $e\in \pi_0(\Bun_G)$ until the end of this section. We use $\overline{e}\in\pi_0(\Bun_S)$ to denote the image of $e$ under the map $\pi_0(\Bun_G)\to \pi_0(\Bun_S)$. From now on, without loss of generality, we always assume $V^I\in\Rep(\Gc^I)$ is irreducible with highest weight $\lambda_I$ and central character $\lambda_{S,I}=(\lambda_{S,1},\cdots,\lambda_{S,r})\in X^*(S)^I$. Moreover, we can assume \begin{equation}\label{eq:sum0} \sum_{i\in I}\lambda_{S,i}=0\end{equation} since otherwise $\Sht_{G,\lambda_I}=\varnothing$.

Consider the correspondence \[\begin{tikzcd}
    \Bun_S\times C^I & \Hk_{S,I} \ar[l, "\lh_{S,I}"'] \ar[r, "\rh_{S,I}"] & \Bun_S\times C^I
\end{tikzcd}.\] Define \[T_{V^I_{\s}}:\Shv(\Bun_S\times C^I)\to \Shv(\Bun_S\times C^I)\] by \[T_{V^I_{\s}}(-):=\rh_{S,I,!}(\lh_{S,I}^*(-)\otimes V^I_{\s})\] where $V_{\s}^I\in\Shv(\Hk_{G,\lambda_{S,I}})$. This functor preserves the subcategory $\Shv_{\triv}(\Bun_S)\otimes\QLisse(C^I)$ and induces a functor \[T_{V^I_{\s}}:\Shv_{\triv}(\Bun_S)\otimes\QLisse(C^I)\to \Shv_{\triv}(\Bun_S)\otimes\QLisse(C^I).\]

We have a natural transformation \begin{equation}\label{diag:constructisotypic1}
    \begin{tikzcd}
        \Shv_{\triv}(\Bun_S)\otimes \QLisse(C^I) \ar[r, "T_{V_{\s}^I}"] \ar[d, "(f^*(-)\otimes\LL_{\s})\otimes\id"'] &\Shv_{\triv}(\Bun_S)\otimes \QLisse(C^I) \ar[d, "(f^*(-)\otimes\LL_{\s})\otimes\id"] \ar[dl, Rightarrow, "\eta_{\s}"] \\
        \Shv_{\Nilp}(\Bun_G)\otimes\QLisse(C^I) \ar[r, "T_{V^I}"] & \Shv_{\Nilp}(\Bun_G)\otimes\QLisse(C^I)
    \end{tikzcd}
\end{equation}
To define the natural transformation $\eta_{\s}$, we consider the diagram \begin{equation}\label{diag:centralpull1}\begin{tikzcd}
    \Bun_G\times C^I \ar[d, "f\times\id"] & \Hk_{G, \lambda_{I}} \ar[d, "f_{\Hk}"] \ar[l, "\lh_{I}"'] \ar[r, "\rh_{I}"] & \Bun_G\times C^I \ar[d, "f\times\id"] \\
    \Bun_S\times C^I & \Hk_{S,\lambda_{S,I}} \ar[l, "\lh_{S,I}"'] \ar[r, "\rh_{S,I}"] & \Bun_S\times C^I
\end{tikzcd}\end{equation} in which $\lh_{S,I}$ and $\rh_{S,I}$ are isomorphisms, the natural transformation $\eta_{\s}$ is defined as

\[\begin{split}
    ((f^*(-)\otimes\LL_{\s})\otimes\id)\circ T_{V^I_{\s}} &=(f\times\id)^*(\rh_{S,I,!}(\lh_{S,I}^*(-)\otimes V^I_{\s}))\otimes(\LL_{\s}\boxtimes\uk_{C^I}) \\
    &\cong (f\times \id)^*(\rh_{S,I,!}\lh_{S,I}^*(-))\otimes(\LL_{\s}\boxtimes V_{\s}^I) \\
    &\cong (f\times \id)^*(\rh_{S,I,!}\lh_{S,I}^*(-))\otimes\rh_{I,!}(\lh_{I}^*(\LL_{\s}\boxtimes\uk_{C^I})\otimes \IC_{V^I}) \\
    &\cong \rh_{I,!}(\rh_{I}^*(f\times\id)^*\rh_{S,I,!}\lh_{S,I}^*(-)\otimes \lh_{I}^*(\LL_{\s}\boxtimes\uk_{C^I})\otimes\IC_{V^I}) \\
    &\cong \rh_{I,!}(f_{\Hk}^*\rh_{S,I}^*\rh_{S,I,!}\lh_{S,I}^*(-)\otimes \lh_I^*(\LL_{\s}\boxtimes\uk_{C^I})\otimes\IC_{V^I}) \\
    &\cong \rh_{I,!}(f_{\Hk}^*\lh_{S,I}^*(-)\otimes \lh_I^*(\LL_{\s}\boxtimes\uk_{C^I})\otimes\IC_{V^I}) \\
    &\cong \rh_{I,!}(\lh_I^*((f\times\id)^*(-)\otimes\LL_{\s}\boxtimes\uk_{C^I})\otimes\IC_{V^I})\\
    &= T_{V^I}\circ ((f^*(-)\otimes\LL_{\s})\otimes\id)
\end{split}.\]
Adding a Frobenius twist, we get a natural transformation \[\eta_{\s}^{(1)}:((f^*(-)\otimes\LL_{\s})\otimes\id)\circ (\Frob\times\id)_!\circ T_{V^I_{\s}}\to (\Frob\times\id)_! \circ T_{V^I}\circ ((f^*(-)\otimes\LL_{\s})\otimes\id).\] Since the vertical maps in \eqref{diag:constructisotypic1} preserve compact objects by Assumption \ref{assumption:eigencompact}, by Definition \ref{def:funccattr}, we get the left vertical map of the following commutative diagram \begin{equation}\label{diag:constructisotypic1tr}\begin{tikzcd}
    \tr_{\QLisse(C^I)}((\Frob\times\id)_!\circ T_{V^I_{\s}},\Shv_{\triv}(\Bun_S^{\overline{e}})\otimes\QLisse(C^I)) \ar[r, "\LT^{\true}", "\sim"'] \ar[d, "\tr_{\QLisse(C^I)}(\eta_{\s}^{(1)})"] & V_{\s}^I \ar[d, "\xi_{\s,I}^e"] \\
    \tr_{\QLisse(C^I)}((\Frob\times\id)_! \circ T_{V^I},\Shv_{\Nilp}(\Bun_G^e)\otimes\QLisse(C^I)) \ar[r, "\LT^{\Serre}", "\sim"'] & l_{I,!}(\IC_{V^I}|_{\Sht_{G,I}^e})
\end{tikzcd}.
\end{equation}

Here, the map $\xi_{\s,I}^e$ is defined such that the diagram above is commutative. We need to explain the isomorphism \begin{equation}\label{eq:lttrueS}\LT^{\true}:\tr_{\QLisse(C^I)}((\Frob\times\id)_!\circ T_{V^I_{\s}},\Shv_{\triv}(\Bun_S^{\overline{e}})\otimes\QLisse(C^I))\isom V_{\s}^I.\end{equation}

The inclusion \[\iota_{S,\triv}:\Shv_{\triv}(\Bun_S^{\overline{e}})\to \Shv(\Bun_S^{\overline{e}})\] admits a continuous right adjoint \[\iota_{S,\triv,\rmR}:\Shv(\Bun_S^{\overline{e}})\to\Shv_{\triv}(\Bun_S^{\overline{e}}),\] which is the composition of the Beilinson's spectral projector $\iota_{S,\rmR}:\Shv(\Bun_S^{\overline{e}})\to\Shv_{\Nilp}(\Bun_S^{\overline{e}})$ and the projection to the direct summand $\Shv_{\triv}(\Bun_S^{\overline{e}})\sub \Shv_{\Nilp}(\Bun_S^{\overline{e}})$. The category $\Shv_{\triv}(\Bun_S^{\overline{e}})$ is self-dual under the Verdier duality with unit \[\rmu_S=\iota_{S,\triv,\rmR}(\D_{S,*}\om_{\Bun_S^{\overline{e}}})\in \Shv_{\triv}(\Bun_S^{\overline{e}})^{\otimes 2}\] and counit \[\ev_S=\Gamma_{\blacktriangle}\circ\D_S^!:\Shv_{\triv}(\Bun_S^{\overline{e}})^{\otimes 2}\to\Vect.\] Here $\D_S:\Bun_S^{\overline{e}}\to\Bun_S^{\overline{e}}\times\Bun_S^{\overline{e}}$ is the diagonal map, and $\Gamma_{\blacktriangle}:\Shv(\Bun_S^{\overline{e}})\to\Vect$ is the unique continuous functor whose restriction to compact objects is $\Gamma$. This is a particular case of the renormalized functor of direct image, which is denoted $f_{\blacktriangle}$ in \cite[\S A.2.3]{arinkin2022stacklocalsystemsrestricted}.

The map $\LT^{\true}$ is defined as \begin{equation}\label{eq:lttrueSprim}
    \begin{split}
        \LT^{\true}:&\tr_{\QLisse(C^I)}((\Frob\times\id)_!\circ T_{V^I_{\s}},\Shv_{\triv}(\Bun_S^{\overline{e}})\otimes\QLisse(C^I)) \\
        \cong & \pr_{2,\blacktriangle}(\D_S\times\id)^!(\Frob_{\Bun_S^{\overline{e}}}\times \id)_*\circ T_{(V^I\boxtimes\triv)_\s}(\iota_{S,\triv}\iota_{S,\triv,\rmR}\D_{S,*}\om_{\Bun_S^{\overline{e}}}\boxtimes\uk_{C^I}) \\
        \to & \pr_{2,\blacktriangle}(\D_S\times\id)^!(\Frob_{\Bun_S^{\overline{e}}}\times \id)_*\circ T_{(V^I\boxtimes\triv)_\s}(\D_{S,*}\om_{\Bun_S^{\overline{e}}}\boxtimes\uk_{C^I}) \\
        \cong & l_{S,I,*}(V_{\s}^I\otimes\om_{\Sht_{S,\lambda_{S,I}}^{\overline{e}}}) \\
        \cong & l_{S,I,*}(V_{\s}^I|_{\Sht_{S,\lambda_{S,I}}^{\overline{e}}})
    \end{split}.
\end{equation} Here, $\pr_{2}:\Bun_S\times C^I\to C^I$ is the projection to the second coordinate, the second map uses the adjunction $\iota_{S,\triv}\iota_{S,\triv,\rmR}\to\id$, the map $l_{S,I}:\Sht_{S,I}\to C^I$ is the map remembering only the legs. The map $\LT^{\true}$ is a special case of the refined true local term map defined in \cite[\S4.11]{gaitsgory2024localtermscategoricaltrace} in which the same map is denoted by $\LT^{\true}_{c,\blacktriangle}$.

Note that there is a natural map $V_{\s}^I\to l_{S,I,*}(V_{\s}^I|_{\Sht_{S,\lambda_{S,I}}^{\overline{e}}})  $ induced by $(l_{S,I}^*,l_{S,I,*})$-adjunction, which realizes $V_{\s}^I$ as a direct summand of the later. One easily checks that the map \eqref{eq:lttrueSprim} factors through $V_{\s}^I$ and defines the isomorphism \eqref{eq:lttrueS}. This explains the construction of the $\s$-isotypic part $\xi_{\s,I}^e:V_{\s}^I\to l_{I,!}(\IC_{V^I}|_{\Sht_{G,I}^{e}})$.

The $\s$-isotypic part map $\xi_{\s,I}^e$ admits a less canonical but equivalent and simpler construction, which is more convenient to use. For simplicity in explanation, we assume that the curve $C$ admits a rational point $c\in C(\F_q)$.\footnote{In general, one can choose a finite subscheme of $C$ and run a parallel argument.}  In this case, instead of considering the map $f:\Bun_G\to\Bun_S$, one chooses a point $c\in C(\FF_q)$ and consider the map $f_c=(-)_c\circ f:\Bun_G\to [*/S]$ where $(-)_c:\Bun_S\to [*/S]$ is taking stalk at $c\in C$. In this case, one replaces the diagram \eqref{diag:centralpull1} by \begin{equation}\label{diag:centralpull2}
    \begin{tikzcd}
        \Bun_G\times C^I \ar[d, "f_c\times\id"] & \Hk_{G, \lambda_{I}} \ar[d, "f_{c,\Hk}"] \ar[l, "\lh_{I}"'] \ar[r, "\rh_{I}"] & \Bun_G\times C^I \ar[d, "f_c\times\id"] \\
    {[*/S]\times C^I} & {[*/S]\times C^I} \ar[l, "\id"'] \ar[r, "t_{\lambda_{S,I}}"] & {[*/S]\times C^I}
    \end{tikzcd}
\end{equation}
where $t_{\lambda_{S,I}}(\cE_0,c^I)=(\cE_0\otimes\cO(-\lambda_{S,I}\cdot c^I)_c,c^I)\in [*/S]\times C^I$. Define \[T_{c,V_{\s}^I}=t_{\lambda_{S,I},!}(-\otimes V_{\s}^I):\Shv([*/S])\otimes\QLisse(C^I)\to \Shv([*/S])\otimes\QLisse(C^I).\] 

In this case, the diagram \eqref{diag:constructisotypic1} can be composed with the upper square in the following diagram, and we arrive at  \begin{equation}\label{diag:constructisotypic2}
        \begin{tikzcd}
        \Shv([*/S])\otimes\QLisse(C^I) \ar[d, "(-)_c^*\otimes\id "'] \ar[r, "T_{c,V_{\s}^I}"] & \Shv([*/S])\otimes\QLisse(C^I) \ar[d, "(-)_c^*\otimes\id "] \ar[dl, Rightarrow, "\eta_{c,S,\s}^{\overline{e}}"] \\
        \Shv_{\triv}(\Bun_S^{\overline{e}})\otimes \QLisse(C^I) \ar[r, "T_{V_{\s}^I}"] \ar[d, "(f^*(-)\otimes\LL_{\s})\otimes\id"'] &\Shv_{\triv}(\Bun_S^{\overline{e}})\otimes \QLisse(C^I) \ar[d, "(f^*(-)\otimes\LL_{\s})\otimes\id"] \ar[dl, Rightarrow, "\eta_{\s}"] \\
        \Shv_{\Nilp}(\Bun_G^e)\otimes\QLisse(C^I) \ar[r, "T_{V^I}"] & \Shv_{\Nilp}(\Bun_G^e)\otimes\QLisse(C^I)
    \end{tikzcd}
\end{equation} Here, the natural transformation $\eta_{c,S,\s}^{\overline{e}}$ is the obvious one. We denote $\eta_{c,\s}^e=\eta_{\s}\circ\eta_{c,S,\s}^{\overline{e}}$. Since all the vertical maps in \eqref{diag:constructisotypic2} preserves compact objects, by the formalism of Definition \ref{def:funccattr}, we arrive at a commutative square \begin{equation}\label{diag:constructisotypic2tr}
    \begin{tikzcd}
    \tr_{\QLisse(C^I)}((\Frob\times\id)_!\circ T_{c,V_{\s}^I},\Shv([*/S])\otimes\QLisse(C^I)) \ar[r, "\LT^{\true}", "\sim"'] \ar[d, "\tr_{\QLisse(C^I)}(\eta_{c,S,\s}^{\overline{e}(1)})"] & V_{\s}^I \ar[d, "\id"]  \\
         \tr_{\QLisse(C^I)}((\Frob\times\id)_!\circ T_{V^I_{\s}},\Shv_{\triv}(\Bun_S^{\overline{e}})\otimes\QLisse(C^I)) \ar[r, "\LT^{\true}", "\sim"'] \ar[d, "\tr_{\QLisse(C^I)}(\eta_{\s}^{(1)})"] & V_{\s}^I \ar[d, "\xi_{\s,I}^e"]  \\
    \tr_{\QLisse(C^I)}((\Frob\times\id)_! \circ T_{V^I},\Shv_{\Nilp}(\Bun_G^e)\otimes\QLisse(C^I)) \ar[r, "\LT^{\Serre}", "\sim"'] & l_{I,!}(\IC_{V^I}|_{\Sht_{G,I}^e})
    \end{tikzcd}
\end{equation} Here the definition of the top horizontal isomorphism $\LT^{\true}$ is similar to but simpler than \eqref{eq:lttrueS}. The commutativity of the upper square follows from a routine generalization of \cite[Theorem\,0.4(b)(i)]{gaitsgory2024localtermscategoricaltrace} to cohomological correspondences with kernel. Therefore, the outer square of \eqref{diag:constructisotypic2tr} gives a simpler definition of the map $\xi_{\s,I}^e:V_{\s}^I\to l_{I,!}(\IC_{V^I}|_{\Sht_{G,I}^e})$.

\subsubsection{Restriction of special cycle classes to the isotypic part} \label{sec:restricttoisotypic}

Keep the same assumptions as in the previous section. Suppose we are given an affine smooth $G$-variety $X$ and a cohomological correspondence \[\frc\in\Cor_{\Hk_{G,I},\IC_{V^I}\langle -d_I\rangle}(\cP_X\boxtimes\uk_{C^I},\cP_X\boxtimes\uk_{C^I} ),\] we can consider the restriction of the geometric Shtuka construction of special cohomological correspondence $\tr_{\Sht,C^I}(\frc):l_{I,!}(\IC_{V^I}|_{\Sht_{G,I}}\langle-d_I\rangle)\to\uk_{C^I}$ along the isotypic part map $\xi_{\s,I}^e:V_{\s}^I\to l_{I,!}(\IC_{V^I}|_{\Sht_{G,I}^e})$, which is the composition  \[\tr_{\Sht,C^I}(\frc)\circ \xi_{\s,I}^e:V_{\s}^I\langle-d_I\rangle \to \uk_{C^I}.\] Let us revisit its construction. Consider the diagram
\begin{equation}\label{diag:isotypicspecialcyclediag1}
    \begin{tikzcd}
        \Shv([*/S])\otimes \QLisse(C^I) \ar[d, "(f_c^*(-)\otimes\LL_{\s}^e)\otimes\id"'] \ar[r, "T_{c,V_{\s}^I\langle-d_I\rangle}"] & \Shv([*/S])\otimes \QLisse(C^I) \ar[d, "(f_c^*(-)\otimes\LL_{\s}^e)\otimes\id"] \ar[dl, Rightarrow, "\eta_{c,\s}^e"]  \\
        \Shv_{\Nilp}(\Bun_G)\otimes\QLisse(C^I) \ar[d, "\int_{X,\Nilp,I}"'] \ar[r, "T_{V^I\langle -d_I\rangle}"] & \Shv_{\Nilp}(\Bun_G)\otimes\QLisse(C^I) \ar[d, "\int_{X,\Nilp,I}"] \ar[dl, Rightarrow, "\eta_{\frc,\Nilp}"]   \\
        \QLisse(C^I) \ar[r, "\id"] & \QLisse(C^I)
    \end{tikzcd}.
\end{equation}
Since all the vertical maps in the diagram above preserve compact objects, applying the formalism in Definition \ref{def:funccattr}, we arrive at the left column of the commutative diagram \begin{equation}\label{diag:isotypicspecialcyclediag1tr}
    \begin{tikzcd}
        \tr_{\QLisse(C^I)}((\Frob\times\id)_!\circ T_{c,V_{\s}^I\langle-d_I\rangle},\Shv([*/S])\otimes\QLisse(C^I)) \ar[r, "\LT^{\true}", "\sim"'] \ar[d, "\tr_{\QLisse(C^I)}(\eta_{c,\s}^{e(1)})"] & V_{\s}^I \langle-d_I\rangle \ar[d, "\xi_{\s,I}^e"]  \\
    \tr_{\QLisse(C^I)}((\Frob\times\id)_! \circ T_{V^I\langle-d_I\rangle},\Shv_{\Nilp}(\Bun_G^e)\otimes\QLisse(C^I)) \ar[r, "\LT^{\Serre}", "\sim"'] \ar[d, "\tr_{\QLisse(C^I)}(\eta_{\frc,\Nilp}^{(1)})"] & l_{I,!}(\IC_{V^I}|_{\Sht_{G,I}^e}\langle-d_I\rangle) \ar[d, "\tr_{\Sht,C^I}(\frc)"]\\
    \tr_{\QLisse(C^I)}(\id,\QLisse(C^I)) \ar[r, "\sim"] & \uk_{C^I}
    \end{tikzcd}
\end{equation}

Define \begin{equation}\label{eq:etacs}
    \eta_{\frc_{\s}^e}:=\eta_{\frc,\Nilp}\circ \eta_{c,\s}^e.
\end{equation} This gives the identity \begin{equation}
    \tr_{\Sht,C^I}(\frc)\circ \xi_{\s,I}^e=\tr_{\QLisse(C^I)}(\eta_{\frc_{\s}^e}^{(1)}).
\end{equation}

We give another description of the natural transformation $\eta_{\frc_{\s}^e}$. Note that \[\int_{X,\Nilp,I}\circ ((f_c^*(-)\otimes\LL_{\s}^e)\otimes\id)\cong \Gamma_c(-\otimes f_{c,!}(\cP_X\otimes\LL_{\s}^e))\otimes\id.\] We claim that the natural transformation $\eta_{\frc_{\s}^e}$ comes from a cohomological correspondence $\frc_{\s}^e$ via the formalism in \S\ref{sec:cctrnt} where  \begin{equation}\label{eq:isotypiccc}\begin{split}\frc_{\s}^e &\in \Cor_{[*/S]\times C^I,V_{\s}^I\langle-d_I\rangle}(f_{c,!}(\cP_X\otimes\LL_{\s}^e)\boxtimes\uk_{C^I},f_{c,!}(\cP_X\otimes\LL_{\s}^e)\boxtimes\uk_{C^I}) \\ &=\Hom^0(V_{\s}^I\langle -d_I\rangle\otimes t_{\lambda_{S,I}}^*(f_{c,!}(\cP_X\otimes\LL_{\s}^e)\boxtimes\uk_{C^I}),f_{c,!}(\cP_X\otimes\LL_{\s}^e)\boxtimes\uk_{C^I})\end{split}\end{equation} constructed as:
\begin{equation}
    \begin{split}
        V_{\s}^I\langle -d_I\rangle\otimes t_{\lambda_{S,I}}^*(f_{c,!}(\cP_X\otimes\LL_{\s}^e)\boxtimes\uk_{C^I}) &\cong t_{\lambda_{S,I}}^*(f_{c}\times\id)_!(\cP_X\otimes \rh_{I,!}(\lh_I^*(\LL_{\s}\boxtimes\uk_{C^I})\otimes\IC_{V^I}\langle -d_I\rangle) ) \\
        &\cong t_{\lambda_{S,I}}^*(f_{c}\times\id)_!\rh_{I,!}(\rh_{I}^*(\cP_X\boxtimes\uk_{C^I})\otimes\lh_{I}^*(\LL_{\s}\boxtimes\uk_{C^I})\otimes\IC_{V^I}\langle -d_I\rangle) \\
        &\cong (f_{c}\times\id)_!\lh_{I,!}(\rh_{I}^*(\cP_X\boxtimes\uk_{C^I})\otimes\lh_{I}^*(\LL_{\s}\boxtimes\uk_{C^I})\otimes\IC_{V^I}\langle -d_I\rangle) \\
        &\cong (f_{c}\times\id)_!(\lh_{I,!}(\rh_I^*(\cP_X\boxtimes\uk_{C^I})\otimes\IC_{V^I}\langle -d_I\rangle)\otimes(\LL_{\s}\boxtimes\uk_{C^I})) \\
        &\to f_{c,!}(\cP_X\otimes\LL_{\s})\boxtimes\uk_{C^I}
    \end{split}
\end{equation} in which the last step uses the cohomological correspondence $\frc:\lh_{I,!}(\rh_I^*(\cP_X\boxtimes\uk_{C^I})\otimes\IC_{V^I}\langle -d_I\rangle)\to \cP_X\boxtimes\uk_{C^I}$. 

\subsubsection{Fake versus real II}\label{sec:fakevsreal2}
In this section, we would like to generalize the identity \eqref{eq:computationofisotypicssrefined}. We work under the following assumption, which is weaker than Assumption \ref{assumption:vectcompact}:

\begin{assumption}\label{assumption:vectorcompactrefined}
    The complex $\int_{X,\Nilp}\LL_{\s}^e$ is perfect.
\end{assumption}

In this case, we have the fake special cycle classes \[z_{\frc,\s}^e:V_{\s}^I\langle-d_I\rangle \to \uk_{C^I},\] the special cycle classes \[\tr_{\Sht,C^I}(\frc): l_{I,!}(\IC_{V^I}|_{\Sht_{G,I}}\langle-d_I\rangle)\to\uk_{C^I},\] and the isotypic part map \[\xi_{\s,I}^e:V_{\s}^I\to l_{I,!}(\IC_{V^I}|_{\Sht_{G,I}^e}).\] The main result in this subsection is the following:
\begin{prop}\label{prop:computationofisotypicss}
    Assuming Assumption \ref{assumption:vectorcompactrefined} and Assumption \ref{assumption:eigencompact}, we have $z_{\frc,\s}^e=\tr_{\Sht,C^I}(\frc) \circ \xi_{\s,I}^e$.
\end{prop}

\begin{proof}[Proof of Proposition \ref{prop:computationofisotypicss}]
    
The basic idea is to replace the axillary category $\Shv([*/S])$ by $\Vect$, and we will get back to the situation in \S\ref{sec:fakevsreal1}. As a naive tempt, we consider the diagram \begin{equation}\label{diag:isotypicspecialcyclediag2}\begin{tikzcd}
    \QLisse(C^I) \ar[r, "-\otimes V_{\s}^I\langle -d_I\rangle"] \ar[d, "p_c^*\otimes\id"'] & \QLisse(C^I) \ar[d, "p_c^*\otimes\id"] \ar[dl, Rightarrow, "\eta_{p_c}"] \\
    \Shv([*/S])\otimes\QLisse(C^I) \ar[r, "T_{c,V_{\s}^I\langle -d_I\rangle}"] \ar[d, "(\int_{X,\Nilp}\circ (f_c^*(-)\otimes\LL_{\s}^e))\otimes\id"'] & \Shv([*/S])\otimes\QLisse(C^I) \ar[d, "(\int_{X,\Nilp}\circ (f_c^*(-)\otimes\LL_{\s}^e))\otimes\id"] \ar[dl, Rightarrow, "\eta_{\frc_{\s}^e}"]\\
    \QLisse(C^I) \ar[r, "\id"] & \QLisse(C^I)
\end{tikzcd}\end{equation} Where we use the map $p_c:[*/S]\to *$. Here, the lower square is the outer square of \eqref{diag:isotypicspecialcyclediag1}, the natural transformation in the upper square is the obvious one.

Pretending that the functor $p_c^*$ preserves compact objects, one would hope for a commutative diagram
\begin{equation}
    \begin{tikzcd}
        \tr_{\QLisse(C^I)}(-\otimes V_{\s}^I\langle -d_I\rangle,\QLisse(C^I)) \ar[r, "\sim"] \ar[d, "\tr_{\QLisse(C^I)}(\eta_{p_c}^{(1)})"'] & V_{\s}^I\langle -d_I\rangle \ar[d, "\id"] \\
    \tr_{\QLisse(C^I)}((\Frob\times\id)_!\circ T_{c,V_{\s}^I\langle -d_I\rangle},\Shv([*/S])\otimes\QLisse(C^I)) \ar[r, "\LT^{\true}", "\sim"']  \ar[d, "\tr_{\QLisse(C^I)}(\eta_{\frc_{\s}^e}^{(1)})"'] & V_{\s}^I\langle -d_I\rangle \ar[d, "\tr_{\Sht,C^I}(\frc)\circ \xi_{\s,I}^e"]  \\
    \tr_{\QLisse(C^I)}(\id, \QLisse(C^I)) \ar[r, "\sim"] & \uk_{C^I} \\
    \end{tikzcd}.
\end{equation}
Since $\eta_{\frc_{\s}^e}\circ \eta_{p_c}=\eta_{a_{\frc,\s}^e}$, we get the desired identity in Proposition \ref{prop:computationofisotypicss}. However, the functor $p_c^*$ does not preserve compact objects. We cannot directly apply the formalism in Definition \ref{def:funccattr}.

We bypass this point by considering the category $\Shv([*/S])^{\ren}$, which is the renormalization of $\Shv([*/S])$ such that constructible complexes are compact. There is a fully faithful embedding $\ren:\Shv([*/S])\to\Shv([*/S]^{\ren})$ preserving compact objects. We refer to \cite[\S F.5]{arinkin2022stacklocalsystemsrestricted} for a thorough introduction to this renormalized category. 

One can bypass the problem by considering the diagram \begin{equation}\label{diag:isotypicspecialcyclediag3}\begin{tikzcd}
    \QLisse(C^I) \ar[r, "-\otimes V_{\s}^I\langle -d_I\rangle"] \ar[d, "p_c^*\otimes\id"'] & \QLisse(C^I) \ar[d, "p_c^*\otimes\id"] \ar[dl, Rightarrow, "\eta_{p_c}"] \\
    \Shv(*/S)^{\ren}\otimes\QLisse(C^I) \ar[r, "T_{c,V_{\s}^I\langle -d_I\rangle}"]  & \Shv(*/S)^{\ren}\otimes\QLisse(C^I) \\
    \Shv(*/S)\otimes\QLisse(C^I) \ar[r, "T_{c,V_{\s}^I\langle -d_I\rangle}"] \ar[d, "(\int_{X,\Nilp}\circ (f_c^*(-)\otimes\LL_{\s}^e))\otimes\id"'] \ar[u, "\ren\otimes\id"] & \Shv(*/S)\otimes\QLisse(C^I) \ar[d, "(\int_{X,\Nilp}\circ (f_c^*(-)\otimes\LL_{\s}^e))\otimes\id"] \ar[u, "\ren\otimes\id"'] \ar[ul, Rightarrow, "\eta_{\ren}"]  \ar[dl, Rightarrow, "\eta_{\frc_{\s}^e}"] \\
    \QLisse(C^I) \ar[r, "\id"] & \QLisse(C^I)
\end{tikzcd}\end{equation}
in which all the natural transformations are the obvious ones. In this diagram, all vertical arrows preserve compact objects. We can apply the formalism in Definition \ref{def:funccattr} to obtain the diagram
\begin{equation}\label{diag:isotypicspecialcyclediag4}\begin{tikzcd}
    \tr_{\QLisse(C^I)}(-\otimes V_{\s}^I\langle -d_I\rangle,\QLisse(C^I)) \ar[r, "\sim"] \ar[d, "\tr_{\QLisse(C^I)}(\eta_{p_c}^{(1)})"'] & V_{\s}^I\langle -d_I\rangle \ar[d, "\id"] \\
    \tr_{\QLisse(C^I)}((\Frob\times\id)_!\circ T_{c,V_{\s}^I\langle -d_I\rangle},\Shv(*/S)^{\ren}\otimes\QLisse(C^I)) \ar[r, "\LT^{\true}", "\sim"']  & V_{\s}^I\langle -d_I\rangle  \\
    \tr_{\QLisse(C^I)}((\Frob\times\id)_!\circ T_{c,V_{\s}^I\langle -d_I\rangle},\Shv(*/S)\otimes\QLisse(C^I)) \ar[r, "\LT^{\true}", "\sim"'] \ar[u, "\tr_{\QLisse(C^I)}(\eta_{\ren}^{(1)})"] \ar[d, "\tr_{\QLisse(C^I)}(\eta_{\frc_{\s}^e}^{(1)})"'] & V_{\s}^I\langle -d_I\rangle \ar[d, "\tr_{\Sht,C^I}(\frc)\circ \xi_{\s,I}^e"] \ar[u, "\id"'] \\
    \tr_{\QLisse(C^I)}(\id, \QLisse(C^I)) \ar[r, "\sim"] & \uk_{C^I} \\
\end{tikzcd}
\end{equation}

Here, the commutativity of the upper square follows from a direct generalization of \cite[Theorem\,0.4(b)(i)]{gaitsgory2024localtermscategoricaltrace} to cohomological correspondences with kernel. The commutativity of the middle square follows from the same kind of generalization of \cite[Proposition\,4.12]{gaitsgory2024localtermscategoricaltrace}. The bottom square is the outer square of \eqref{diag:isotypicspecialcyclediag1tr}.

By the commutativity of \eqref{diag:isotypicspecialcyclediag4}, we are reduced to show  \begin{equation}\label{eq:technicalreduction1} \tr_{\QLisse(C^I)}(\eta_{\frc_{\s}^e}^{(1)}) \circ \tr_{\QLisse(C^I)}(\eta_{\ren}^{(1)})^{-1}\circ \tr_{\QLisse(C^I)}(\eta_{p_c}^{(1)})=z_{\frc,\s}^e .\end{equation} For this purpose, consider the strictly commutative diagram \begin{equation}\label{diag:technicaldiag1}
    \begin{tikzcd}
         &\QLisse(C^I) \ar[d, "p_c^*\otimes\id"]  \\
        \Shv(*/S)\otimes\QLisse(C^I) \ar[r, "\ren\otimes\id"] \ar[dr, "(\int_{X,\Nilp}\circ (f_c^*(-)\otimes\LL_{\s}^e))\otimes\id"'] & \Shv(*/S)^{\ren}\otimes\QLisse(C^I) \ar[d, "(\int_{X,\Nilp}\circ (f_c^*(-)\otimes\LL_{\s}^e))\otimes\id"]\\
        & \QLisse(C^I)
    \end{tikzcd}
.\end{equation} Moreover, each category in the diagram is equipped with an endomorphism which has been specified in \eqref{diag:isotypicspecialcyclediag3}, among which the only new natural transformation is \begin{equation}
    \begin{tikzcd}
        \Shv(*/S)^{\ren}\otimes\QLisse(c^I) \ar[r, "T_{c,V_{\s}^I\langle -d_I\rangle}"] \ar[d, "(\int_{X,\Nilp}\circ (f_c^*(-)\otimes\LL_{\s}^e))\otimes\id"'] & \Shv(*/S)^{\ren}\otimes\QLisse(c^I) \ar[d, "(\int_{X,\Nilp}\circ (f_c^*(-)\otimes\LL_{\s}^e))\otimes\id"] \ar[dl, Rightarrow, "\eta_{\frc_{\s}^e,\ren}"] \\
        \QLisse(C^I) \ar[r, "\id"] & \QLisse(C^I)
    \end{tikzcd}
\end{equation} such that $ \eta_{\frc_{\s}^e,\ren}\circ\eta_{\ren}=\eta_{\frc_{\s}^e}$. 

By Assumption \ref{assumption:vectorcompactrefined}, we know $(\int_{X,\Nilp}\circ f_c^*(-)\otimes\LL_{\s}^e)(\uk_{[*/S]}) \in \Vect$ is compact. Since $\uk_{[*/S]}\in\Shv([*/S])^{\ren}$ is a compact generator, we know all maps in \eqref{diag:technicaldiag1} preserve compact objects. Therefore, it gives a commutative diagram \begin{equation}\label{diag:isotypicspecialcyclediag5}
    \adjustbox{scale=.77}{%
    \begin{tikzcd}
         &\tr_{\QLisse(C^I)}(-\otimes V_{\s}\langle -d_I\rangle,\QLisse(C^I)) \ar[d, "\tr_{\QLisse(C^I)}(\eta_{p_c}^{(1)})"]  \\
        \tr_{\QLisse(C^I)}((\Frob\times\id)_!\circ T_{c,V_{\s}^I\langle -d_I\rangle},\Shv([*/S])\otimes\QLisse(C^I)) \ar[r, "\tr_{\QLisse(C^I)}(\eta_{\ren}^{(1)})"] \ar[dr, "\tr_{\QLisse(C^I)}(\eta_{\frc_{\s}^e}^{(1)})"'] & \tr_{\QLisse(C^I)}((\Frob\times\id)_!\circ T_{c,V_{\s}^I\langle -d_I\rangle},\Shv([*/S])^{\ren}\otimes\QLisse(C^I)) \ar[d, "\tr_{\QLisse(C^I)}(\eta_{\frc_{\s}^e,\ren}^{(1)})"] \\ & \tr_{\QLisse(C^I)}(\id, \QLisse(C^I))
    \end{tikzcd}
    }
.\end{equation} Since we have $\eta_{\frc_{\s}^e,\ren}\circ\eta_{p_c}=\eta_{a_{\frc,\s}^e}$, we know \[\begin{split} &\tr_{\QLisse(C^I)}(\eta_{\frc_{\s}^e}^{(1)}) \circ \tr_{\QLisse(C^I)}(\eta_{\ren}^{(1)})^{-1}\circ \tr_{\QLisse(C^I)}(\eta_{p_c}^{(1)}) \\
=&\tr_{\QLisse(C^I)}(\eta_{\frc_{\s}^e,\ren}^{(1)})\circ \tr_{\QLisse(C^I)}(\eta_{p_c}^{(1)}) \\
=&\tr_{\QLisse(C^I)}(\eta_{a_{\frc,\s}^e}^{(1)})\\
=&z_{\frc,\s}^e\end{split}.\] This concludes the proof of Proposition \ref{prop:computationofisotypicss}.

\end{proof}

\begin{remark}
    Suppose one replaces $\LT^{\true}$ by $\LT^{\Serre}$ for all the horizontal maps in \eqref{diag:isotypicspecialcyclediag4}, we do not know how to prove the commutativity of the top square.\footnote{It might be possible to generalize the argument in \cite[\S6]{arinkin2022automorphicfunctionstracefrobenius} and show that $\LT^{\true}=\LT^{\Serre}$.} This is the reason that we use $\LT^{\true}$ instead of $\LT^{\Serre}$ even though the latter seems more natural in our setting.
\end{remark}

\subsection{Generically middle-dimensional case}\label{sec:gmdcase}
In this section, we study the case $d_I=0$ and $V^I\in\Rep(\Gc^I)^{\heartsuit}$, which means that the cohomological correspondence is $\frc\in\Cor_{\Hk_{G,I},\IC_{V^I}}(\cP_X\boxtimes\uk_{C^I},\cP_X\boxtimes\uk_{C^I})$. In this case, the geometric trace is $\tr_{\Sht,C^I}(\frc):l_{I,!}(\IC_{V^I}|_{\Sht_{G,I}})\to\uk_{C^I}$. We develop a tool to study its restriction to the isotypic part $\tr_{\Sht,C^I}(\frc)\circ \xi_{\s,I}^e:V_{\s}^I\to k_{C^I}$. In this section, we assume $V^I\in\Rep(\Gc^I)^{\heartsuit}$ is irreducible with highest weight $\lambda_{I}$ whose central character $\lambda_{S,I}\in X_*(S)^I$ satisfies $\sum_{i\in I}\lambda_{S,i}=0$.

In this case, taking stalk at $c^I\in C^I$ induces an injection \[(-)_{c^I}:\Hom^0(V_{\s}^I,\uk_{C^I})\to \Hom^0(V_{\s,c^I}^I,k).\] Therefore, we only need to understand the induced map on the stalk $(\tr_{\Sht,C^I}(\frc)\circ \xi_{\s,I}^e)_{c^I}:V_{\s,c^I}^I\to k$. 

Note that there is another natural map \[\tr_{\Sht,C^I}(\frc_{\s,c^I}^e):V_{\s,c^I}^I\to k.\] We now explain this map. Consider the restriction along $c^I\to C^I$ of the cohomological correspondence \eqref{eq:isotypiccc} which is \[\frc_{\s,c^I}^e\in \Cor_{[*/S],V_{\s,c^I}^I}(f_{c,!}(\cP_X\otimes\LL_{\s}^e)\boxtimes\uk_{C^I}, f_{c,!}(\cP_X\otimes\LL_{\s}^e)\boxtimes\uk_{C^I}).\] Since $\sum_{i\in I}\lambda_{S,i}=0$, we know that the underlying correspondence of $\frc_{\s,c^I}^e$ is the identity correspondence on $[*/S]$. The geometric trace above is the map \[\tr_{\Sht,C^I}(\frc_{\s,c^I}^e): V_{\s,c^I}^I\cong  \Fun([*/S(\FF_q)])\otimes V_{\s,c^I}^I\cong \Gamma_c(\uk_{[*/S(\FF_q)]})\otimes V_{\s,c^I}^I \to k\] in which we are using the identification $\Fun([*/S(\FF_q)])\cong k$ given by evaluation at the unique point.\footnote{Note that this isomorphism is $|S(\FF_q)|$ times the natural adjunction map $\Gamma_c(\uk_{[*/S(\FF_q)]})\to k$.} Note that the trace construction makes sense since $f_{c,!}(\cP_X\otimes\LL_{\s}^e)$ is constructible by Assumption \ref{assumption:eigencompact}.

We have the following proposition:
\begin{prop}\label{prop:gmdcase}
    We have $(\tr_{\Sht,C^I}(\frc)\circ \xi_{\s,I}^e)_{c^I}=\tr_{\Sht,C^I}(\frc_{\s,c^I}^e)\in \Hom^0(V_{\s,c^I}^I,k)$.
\end{prop}
\begin{proof}[Proof of Proposition \ref{prop:gmdcase}]
We first note that there is a commutative diagram \begin{equation}\label{diag:diagisotypiccc1tr}
    \begin{tikzcd}
        \tr(\Frob_!\circ(-\otimes V_{\s,c^I}^I),\Shv([*/S])) \ar[d, "\tr(\eta_{\frc_{\s,c^I}^e}^{(1)})"] \ar[r, "\LT^{\Serre}", "\sim"'] & V_{\s,c^I}^I \ar[d, "\tr_{\Sht,C^I}(\frc_{\s,c^I}^e)"] \\
        \tr(\id,\Vect) \ar[r, "\sim"] & k
    \end{tikzcd}.
\end{equation} in which that the local term map $\LT^{\Serre}$ is defined by \[\begin{split}\LT^{\Serre}:\tr(\Frob_!\circ(-\otimes V_{\s,c^I}^I),\Shv([*/S]))&\cong \Gamma_c(\D_c^*(\Frob\times\id)_!(\D_{c,!}\uk_{[*/S]}\otimes V_{\s,c^I}^I))\\ &\cong \Gamma_c(\uk_{[*/S(\FF_q)]})\otimes V_{\s,c^I}^I \\ &\cong \Fun([*/S(\FF_q)])\otimes V_{\s,c^I}^I \\ &\cong  V_{\s,c^I}^I \end{split}.\] Here we use the map $\D_c:[*/S]\to[*/S]\times[*/S]$, and we use the isomorphism $\Fun([*/S(\FF_q)])\cong k$ given by evaluation at the unique point.

On the other hand, by restricting the outer square of \eqref{diag:isotypicspecialcyclediag1tr} to $c^I$, we have a commutative diagram \begin{equation}\label{diag:diagisotypiccc2tr}
    \begin{tikzcd}
        \tr(\Frob_!\circ(-\otimes V_{\s,c^I}^I),\Shv([*/S])) \ar[d, "\tr(\eta_{\frc_{\s,c^I}^e}^{(1)})"] \ar[r, "\LT^{\true}", "\sim"'] & V_{\s,c^I}^I \ar[d, "(\tr_{\Sht,C^I}(\frc)\circ \xi_{\s,I}^e)_{c^I}"] \\
        \tr(\id,\Vect) \ar[r, "\sim"] & k
    \end{tikzcd}
\end{equation} in which the local term map $\LT^{\true}$ is defined by \[\begin{split}
    \LT^{\true}:\tr(\Frob_!\circ(-\otimes V_{\s,c^I}^I),\Shv([*/S]))&\cong \Gamma_{\blacktriangle}(\D_c^!(\Frob\times\id)_*(\D_{c,*}\om_{[*/S]}\otimes V_{\s,c^I}^I))\\ &\cong \Gamma(\om_{[*/S(\FF_q)]})\otimes V_{\s,c^I}^I \\ &\cong \Fun([*/S(\FF_q)])\otimes V_{\s,c^I}^I \\ &\cong  V_{\s,c^I}^I \end{split}.\] In which we are also using the isomorphism $\Fun([*/S(\FF_q)])\cong k$ given by evaluation at the unique point for the last step. Indeed, the inverse of this map is the adjunction map $k\to \Gamma(\om_{[*/S(\FF_q)]})$.

By \cite[Theorem\,6.1.4]{arinkin2022automorphicfunctionstracefrobenius}, we know $\LT^{\true}=\LT^{\Serre}$.\footnote{Strictly speaking, there is no tensor product $-\otimes V_{\s,c^I}^I$ involved in \cite[Theorem\,6.1.4]{arinkin2022automorphicfunctionstracefrobenius}. However, it is easy to see that this factor is innocuous and goes through the proof in \textit{loc.cit}.} Combining this with \eqref{diag:diagisotypiccc1tr}\eqref{diag:diagisotypiccc2tr}, we obtain \[(\tr_{\Sht,C^I}(\frc)\circ \xi_{\s,I}^e)_{c^I}=\tr(\eta_{\frc_{\s,c^I}^e}^{(1)})=\tr_{\Sht,C^I}(\frc_{\s,c^I}^e).\]
\end{proof}

\subsection{Diagonal cycles}\label{sec:isotypicdiag}
In this section, we study the restriction of the diagonal cycle on the isotypic part.

\subsubsection{Conjectural description}

We take $G=H\times H$, $X=H\backslash H\times H$, and $\s=(\s_H,c^*\s_H)\in \Loc_{\Hc}^{\arith}(k)\times \Loc_{\Hc}^{\arith}(k)$ where $c:\Hc\to \Hc$ is the Cartan involution (same letter for the induced map $c:\Loc_{\Hc}\to\Loc_{\Hc}$). The $\Hc$-local system $c^*\s_H$ is characterized by $V^*_{H,c^*\s}\cong V_{H,\s}$ for any irreducible $V_H\in\Rep(\Hc)$.\footnote{Strictly speaking, we have $V_{H,c^*\s}\cong c^*V_{H,\s}$. However, one has the natural isomorphism $c^*V_H\cong V_H^*$ given by the Geometric Satake equivalence.} In this case, we take a Hecke eigensheaf $\LL_{\s_H}\in\Shv_{\Nilp}(\Bun_H)$ for $\s_H$. One can consider $\DD(\LL_{\s_H})\in\Shv_{\Nilp}(\Bun_H)$ which is naturally a Hecke eigensheaf with eigenvalue $c^*\s_H$. Here, $\DD:\Shv(\Bun_G)_c\to\Shv(\Bun_G)_c^{\mathrm{op}}$ is the Verdier duality. We take \[\LL_{\s}=\LL_{\s_H}\boxtimes \DD(\LL_{\s_H})\in\Shv_{\Nilp}(\Bun_G).\] 

Fix $e\in\pi_1(H)\cong\pi_0(\Bun_H)$. We use $\Bun_H^e\sub \Bun_H$ and $\Bun_G^e\sub \Bun_G$ to denote the corresponding connected component. Define $S_H=H/[H,H]$, then $S=S_H\times S_H$. Consider irreducible representation $V_H\in\Rep(\Hc^I)$ with highest weight $\lambda_{H,I}$. Take $V^I=V_H^I\boxtimes V_H^I\in\Rep(\Gc^I)$. Assume $V_{H}^I$ has central character $\lambda_{S_H,I}=(\lambda_{S_H,1},\cdots,\lambda_{S_H,r})\in X^*(S_H)^I$ and $\sum_{i\in I}\lambda_{S_H,i}=0$. 

Define $\LL_{\s_H}^e=\LL_{\s_{H}}|_{\Bun_H^e}$ and $\LL_{\s}^e=\LL_{\s}|_{\Bun_G^e}$. We have an isotypic part map \[\xi_{\s,I}^e=\xi_{\s_H,I}^e\otimes\xi_{c^*\s_H,I}^e: V_{H,\s_H}^I \otimes V_{H,\s_H}^{I*}\cong  V_{\s}^I\to l_{I,!}(\IC_{V^I}|_{\Sht_{G,I}^e})\cong l_{H,I,!}(\IC_{V_H^I}|_{\Sht_{H,I}^e})\otimes l_{H,I,!}(\IC_{V_H^I}|_{\Sht_{H,I}^e}).\]

Take the diagonal cohomological correspondence \[\frc=\D_{\Hk,I,!}[\Hk_{H, \lambda_{H,I}}/\Bun_H\times C^I]\in\Cor_{\Hk_{G,I},\IC_{V^I}}(\cP_X\boxtimes\uk_{C^I},\cP_X\boxtimes\uk_{C^I})\] as in Theorem  \ref{thm:diaggeo=cyc}. We would like to understand the restriction of intersection pairing on the isotypic part \[\tr_{\Sht,C^I}(\frc)\circ \xi_{\s,I}^e:V_{\s}^I\to\uk_{C^I}.\] Note that there is the natural evaluation map $\ev_{V_{H,\s_H}^I}:V_{\s}^I\cong V_{H,\s_H}^I\otimes V_{H,\s_H}^{I*}\to \uk_{C^I}$. It is natural to ask about the relation between these two maps. The following is a conjectural answer:
\begin{conj}\label{conj:intersection}
    We have $\tr_{\Sht,C^I}(\frc)\circ \xi_{\s,I}^e=\tr(\Frob,\Gamma_c(\LL_{\s_H}^e\otimes\DD(\LL_{\s_H}^e)))\cdot \ev_{V_{H,\s_H}^I} $.
\end{conj}

We need to explain the meaning of the number $\tr(\Frob,\Gamma_c(\LL_{\s_H}^e\otimes\DD(\LL_{\s_H}^e)))$. When $H$ is semisimple, assuming Assumption \ref{assumption:eigenstronglycompact}, the complex $\Gamma_c(\LL_{\s_H}^e\otimes\DD(\LL_{\s_H}^e))$ is perfect, hence, the trace is a well-defined number. When $H$ is not semisimple, the vector space $\Gamma_c(\LL_{\s_H}^e\otimes\DD(\LL_{\s_H}^e))$ is usually infinite-dimensional. In this case, we assume Assumption \ref{assumption:eigencompact}, and the sum $\tr(\Frob,\Gamma_c(\LL_{\s_H}^e\otimes\DD(\LL_{\s_H}^e)))$ will be convergent. In fact, consider the map $f_{c,H}:\Bun_H\to [*/S_H]$ and the diagonal map $\D_c:[*/S_H]\to[*/S]$. Since the functor \[\int_{X,\Nilp}\circ (f_c^*(-)\otimes\LL_{\s})\cong \Gamma_c(-\otimes \D_{c,!} f_{c,H,!}(\LL_{\s_H}^e\otimes\DD(\LL_{\s_H}^e))):\Shv([*/S])\to\Vect \] preserves compact objects, we know $f_{c,H,!}(\LL_{\s_H}^e\otimes\DD(\LL_{\s_H}^e))\in\Shv([*/S_H])$ is constructible. Note that \[f_{c,H,!}(\LL_{\s_H}^e\otimes\DD(\LL_{\s_H}^e))\] is a constant sheaf. Indeed, consider the Cartesian diagram \[\begin{tikzcd}
    {[*/Z(H)]\times\Bun_H} \ar[r, "m"] \ar[d, "\id\times f_{c,H}"] & \Bun_H \ar[d, "f_{c,H}"] \\
    {[*/Z(H)]\times[*/S_H]} \ar[r, "m_c"] & {[*/S_H]}
\end{tikzcd}\] where $m$ is induced by the closed embedding $[*/Z(H)]\sub\Bun_{Z(H)}$ and the natural action of $\Bun_{Z(H)}$ on $\Bun_H$, $m_c$ is induced by the natural map $[*/Z(H)]\to[*/S_H]$ and the natural multiplication on $[*/S_H]$. By base change and Hecke eigen-property of $\LL_{\s_H}$, one has \[m_c^*f_{c,H,!}(\LL_{\s_H}^e\otimes\DD(\LL_{\s_H}^e))\cong (\id\times f_{c,H})_!m^*(\LL_{\s_H}^e\otimes\DD(\LL_{\s_H}^e))\cong \uk_{[*/Z(H)]}\boxtimes f_{c,H,!}(\LL_{\s_H}^e\otimes\DD(\LL_{\s_H}^e)).\] This implies that $f_{c,H,!}(\LL_{\s_H}^e\otimes\DD(\LL_{\s_H}^e))$ is constant.

Therefore, consider the map $i:*\to [*/S_H]$, we have \[\Gamma_c(\LL_{\s_H}^e\otimes\DD(\LL_{\s_H}^e))\cong i^* f_{c,H,!}(\LL_{\s_H}^e\otimes\DD(\LL_{\s_H}^e))\otimes\Gamma_c([*/S_H])\] in which the first factor is a perfect complex. Therefore, we have \begin{equation}\label{eq:diaggeoisotypictr}\tr(\Frob,\Gamma_c(\LL_{\s_H}^e\otimes\DD(\LL_{\s_H}^e)))=\tr(\Frob,i^* f_{c,H,!}(\LL_{\s_H}^e\otimes\DD(\LL_{\s_H}^e)))\cdot |S_H(\FF_q)|^{-1},\end{equation} which is a well-defined number.

\subsubsection{Conjectural description of cohomological correspondence}

By Proposition \ref{prop:gmdcase}, to prove Conjecture \ref{conj:intersection}, we only need to understand the cohomological correspondence \[\frc_{\s,c^I}^e\in \Cor_{[*/S], V_{\s,c^I}^I}(\D_{c,!}f_{c,H,!}(\LL_{\s_H}^e\otimes\DD(\LL_{\s_H}^e)),\D_{c,!}f_{c,H,!}(\LL_{\s_H}^e\otimes\DD(\LL_{\s_H}^e)))\] whose general definition is given in \S\ref{sec:gmdcase}. Note that there is another natural cohomological correspondence \[\begin{split}\ev_{V_{H,\s_H,c^I}^I}\otimes\id &\in \Hom^0(V_{\s,c^I}^I\otimes \D_{c,!}f_{c,H,!}(\LL_{\s_H}^e\otimes\DD(\LL_{\s_H}^e)),\D_{c,!}f_{c,H,!}(\LL_{\s_H}^e\otimes\DD(\LL_{\s_H}^e))) \\ & \cong \Cor_{[*/S], V_{\s,c^I}^I}(\D_{c,!}f_{c,H,!}(\LL_{\s_H}^e\otimes\DD(\LL_{\s_H}^e)),\D_{c,!}f_{c,H,!}(\LL_{\s_H}^e\otimes\DD(\LL_{\s_H}^e)))\end{split}\] in which $\ev_{V_{H,\s_H,c^I}^I}:V_{\s,c^I}^I\cong V_{H,\s_H,c^I}^I\otimes V_{H,\s_H,c^I}^{I*} \to k$ is the natural evaluation map. We have the following conjecture:

\begin{conj}\label{conj:intersectioncc}
    We have $\frc_{\s,c^I}^e=\ev_{V_{H,\s_H,c^I}^I}\otimes\id$.
\end{conj}

\begin{proof}[Proof of Conjecture \ref{conj:intersection} assuming Conjecture \ref{conj:intersectioncc}]
Note that \[ \begin{split}
(\tr_{\Sht,C^I}(\frc)\circ \xi_{\s,I}^e)_{c^I}&=\tr_{\Sht,C^I}(\frc_{\s,c^I}^e)\\
&=\tr_{\Sht,C^I}(\ev_{V_{H,\s_H,c^I}^I}\otimes\id) \\&=\ev_{V_{H,\s_H,c^I}^I}\otimes\tr_{\Sht,C^I}(\id_{\D_{c,!}f_{c,H,!}(\LL_{\s_H}^e\otimes\DD(\LL_{\s_H}^e))}) \\
&=\ev_{V_{H,\s_H,c^I}^I}\otimes \tr(\Frob,i^* f_{c,H,!}(\LL_{\s_H}^e\otimes\DD(\LL_{\s_H}^e)))\cdot |S_H(\FF_q)|^{-1} \\
&= \tr(\Frob,\Gamma_c(\LL_{\s_H}^e\otimes\DD(\LL_{\s_H}^e)))\cdot \ev_{V_{H,\s_H,c^I}^I}
\end{split} .\] Here, the first identity is Proposition \ref{prop:gmdcase}, the second identity is Conjecture \ref{conj:intersectioncc}, the fourth identity follows from \cite[Corollary\,0.9]{gaitsgory2024localtermscategoricaltrace}, the last identity follows from \eqref{eq:diaggeoisotypictr}. Combined with the injectivity of $(-)_{c^I}:\Hom^0(V_{\s}^I,\uk_{C^I})\to\Hom^0(V_{\s,c^I}^I,k)$, we get \[\tr_{\Sht,C^I}(\frc)\circ \xi_{\s,I}^e=\tr(\Frob,\Gamma_c(\LL_{\s_H}^e\otimes\DD(\LL_{\s_H}^e)))\cdot \ev_{V_{H,\s_H}^I} .\]

\end{proof}

\subsubsection{Evidence for Conjecture \ref{conj:intersectioncc}}
In this section, we will prove Conjecture \ref{conj:intersectioncc} under Assumption \ref{assumption:diagonalgenerate}, which we can verify in case $H=\GL_n$. This provides evidence for Conjecture \ref{conj:intersectioncc}.

Recall that for each irreducible representation $V_{H}'\in \Rep(\Hc)$ and $V'=V_H'\boxtimes V_H'\in\Rep(\Gc)$, we have the diagonal cohomological correspondence $\frc_{V_H'}\in\Cor_{\Hk_{G,\{1\}},\IC_{V'}}(\cP_X\boxtimes\uk_{C},\cP_X\boxtimes\uk_{C})$  which gives a map \[a_{\frc_{V_H'},\s}:\Gamma(V_{H,\s_H}'\otimes V_{H,c^*\s_H}')\langle 2\rangle\otimes \int_{X,\Nilp}\LL_{\s}\to \int_{X,\Nilp}\LL_{\s}.\] Putting all these maps for $V_H'\in\Irr(\Rep(\Hc))$ together, one obtains a map \[a_{\s}:\bigoplus_{V_H'\in\Irr(\Rep(\Hc))}\Gamma(V_{H,\s_H}'\otimes V_{H,c^*\s_H}')\langle 2\rangle\otimes\int_{X,\Nilp}\LL_{\s}\to \int_{X,\Nilp}\LL_{\s}.\] Furthermore, one obtains an action map for the free tensor algebra \[a^{ \otimes }_{\s}:(\bigoplus_{V_H'\in\Irr(\Rep(\Hc))}\Gamma(V_{H,\s_H}'\otimes V_{H,c^*\s_H}')\langle 2\rangle )^{\otimes}\otimes\int_{X,\Nilp}\LL_{\s}\to \int_{X,\Nilp}\LL_{\s}.\]

We make the following assumption:
\begin{assumption}\label{assumption:diagonalgenerate}
    The following statements are true:
    \begin{enumerate}
        \item For every $e'\in \pi_1(H)$, we have $\dim H^0(\int_{X,\Nilp}\LL_{\s}^{e'})=1$;
        \item The map \[H^*(a^{\otimes}_{\s})|_{H^0}:H^*((\bigoplus_{V_H'\in\Irr(\Rep(\Hc))}\Gamma(V_{H,\s_H}'\otimes V_{H,c^*\s_H}')\langle 2\rangle)^{\otimes})\otimes H^0(\int_{X,\Nilp}\LL_{\s})\to H^*(\int_{X,\Nilp}\LL_{\s})\] is surjective.
    \end{enumerate}
\end{assumption}

The main result in this section is the following:
\begin{prop}\label{prop:intersectioncc}
    Under Assumption \ref{assumption:diagonalgenerate}, Conjecture \ref{conj:intersectioncc} is true; hence, Conjecture \ref{conj:intersection} is true.
\end{prop}

\begin{proof}[Proof of Proposition \ref{prop:intersectioncc}]
Since the functor $\Gamma_c:\Shv([*/S])_c\to \Vect$ is faithful, we only need to check that \[\Gamma_c(\frc_{\s,c^I}^e)=\ev_{V_{H,\s_H,c^I}^I}\otimes\id \in \Hom^0(V_{\s,c^I}^I\otimes \int_{X,\Nilp}\LL_{\s}^e, \int_{X,\Nilp}\LL_{\s}^e ).\] Note that $\Gamma_c(\frc_{\s,c^I}^e)=a_{\frc_{c^I},\s}^e$ in which the later is defined in \eqref{eq:heckeactiononisotypicpartfixleg}, we are reduced to show \begin{equation} \label{eq:diagccreducedtoshow}a_{\frc_{c^I},\s}^e=\ev_{V_{H,\s_H,c^I}^I}\otimes\id \in \Hom^0(V_{\s,c^I}^I\otimes \int_{X,\Nilp}\LL_{\s}^e, \int_{X,\Nilp}\LL_{\s}^e ).\end{equation}

Since $\int_{X,\Nilp}\LL_{\s}^e=\Gamma_c(\LL_{\s_H}^e\otimes\DD(\LL_{\s_H}^e))$, there is a canonical element \begin{equation}\label{eq:canonicalelement}\ev_{\LL_{\s}^e}\in H^0(\int_{X,\Nilp}\LL_{\s}^e)^*\end{equation} defined as \[\ev_{\LL_{\s}^e}:\int_{X,\Nilp}\LL_{\s}^e \cong \Gamma_c(\LL_{\s_H}^e\otimes\DD(\LL_{\s_H}^e))\to \Gamma_c(\om_{\Bun_H^e})\to k.\] We have the following lemma:

\begin{lemma}\label{lem:diagcccalh0}
We have $a_{\frc_{c^I},\s}^{e,*}\ev_{\LL_{\s}^e}=\ev_{\LL_{\s}^e}\otimes \ev_{V_{H,\s_H,c^I}^I}\in H^0(\int_{X,\Nilp}\LL_{\s}^e)^*\otimes V_{\s,c^I}^{I*}$.
\end{lemma} 

\begin{proof}[Proof of Lemma \ref{lem:diagcccalh0}]
    Note that $\ev_{\LL_{\s}^e}\otimes \ev_{V_{H,\s_H,c^I}^I}=\ev_{T_{V_{H}^I,c^I}\LL_{\s_H}^e}$ where \[\ev_{T_{V_{H}^I,c^I}\LL_{\s_H}^e}\in \Hom^0(T_{V_{H}^I,c^I}\LL_{\s_H}^e\otimes\DD(T_{V_{H}^I,c^I}\LL_{\s_H}^e),\om_{\Bun_H})\cong\Hom^0(\Gamma_c(T_{V_{H}^I,c^I}\LL_{\s_H}^e\otimes\DD(T_{V_{H}^I,c^I}\LL_{\s_H}^e)),k)\] is the natural map. The claim follows from the following fact: For any $\cF\in\Shv(\Bun_H)_c$, evaluating $\eta_{\frc_{c^I}}$ on $\cF\boxtimes\DD(\cF)$ gives a natural map \[\eta_{\frc_{c^I}}(\cF\boxtimes\DD(\cF)):\Gamma_c(T_{V_H^I,c^I}\cF\otimes\DD(T_{V_H^I,c^I}\cF))\to\Gamma_c(\cF\otimes\DD(\cF)) . \] This map satisfies $\eta_{\frc_{c^I}}(\cF\boxtimes\DD(\cF))^*\ev_{\cF}=\ev_{T_{V_H^I,c^I}\cF}$.
\end{proof}

By Lemma \ref{lem:diagcccalh0}, and Assumption \ref{assumption:diagonalgenerate}(1), we know \begin{equation}\label{eq:diagccreducedtoshowh0} a_{\frc_{c^I},\s}|_{H^0}=(\ev_{V_{H,\s_H,c^I}^I}\otimes\id)|_{H^0} \in \Hom (V_{\s,c^I}^I\otimes H^0(\int_{X,\Nilp}\LL_{\s}), H^0(\int_{X,\Nilp}\LL_{\s})).\end{equation}

We have the following lemma:
\begin{lemma}\label{lem:diagcccom}
    We have \[a_{\frc_{c^I},\s}\circ a_{\s}^{\otimes}=a_{\s}^{\otimes}\circ a_{\frc_{c^I},\s}\in\Hom^0(V_{\s,c^I}^I\otimes (\bigoplus_{V_H'\in\Irr(\Rep(\Hc))}\Gamma(V_{H,\s_H}'\otimes V_{H,c^*\s_H}')\langle 2\rangle )^{\otimes}\otimes\int_{X,\Nilp}\LL_{\s}, \int_{X,\Nilp}\LL_{\s}) .\] Here the composition is understood as in \S\ref{sec:associativity}.
\end{lemma}

\begin{proof}
    Note that we are in the situation of \S\ref{sec:commutatorrelation}. Indeed, the cohomological correspondences involved here all come from local special cohomological correspondences. Since the local Plancherel algebra has non-negative degrees, \cite[Assumption\,4.46]{liu2025higherperiodintegralsderivatives} is satisfied if the genus of the curve $g(C)\neq 1$. In this case, the statement is true by the discussion in \S\ref{sec:commutatorrelation}. One can easily remove the condition $g(C)\neq 1$ since the leg is fixed in one of the Hecke actions.
\end{proof}

Now, \eqref{eq:diagccreducedtoshow} follows from a combination of \eqref{eq:diagccreducedtoshowh0}, Lemma \ref{lem:diagcccom}, and Assumption \ref{assumption:diagonalgenerate}(2). This concludes the proof of Proposition \ref{prop:intersectioncc} under Assumption \ref{assumption:diagonalgenerate}.

\end{proof}

\section{Application: higher Rankin--Selberg integrals}\label{sec:rs}

In this section, we work towards a proof of Theorem \ref{thm:intro:main}.
\begin{itemize}
    \item In \S\ref{sec:app:isotypicpart}, we review the construction of the $\s$-isotypic part in the cohomology of Shtukas for $G=\GL_n$.
    \item In \S\ref{sec:app:rscycle}, we study the Rankin--Selberg cycle classes for $G=\GL_n\times\GL_{n-1}$.
    \item In \S\ref{sec:application:diagonalcycle}, we study the diagonal cycle classes for $G=\GL_n$.
    \item In \S\ref{sec:app:compute}, we complete the proof of Theorem \ref{thm:intro:main}.
\end{itemize}

\subsection{Isotypic part}\label{sec:app:isotypicpart}
We first define the $\s$-isotypic part map \eqref{eq:intro:spectralsub3} used in the formulation of Theorem  \ref{thm:intro:main}. They will be defined using techniques of \S\ref{sec:isotypicpartscc2}.

For $G=\GL_n$, we have $S=G/[G,G]=\Gm$. We use $\Bun_{\GL_n}^d\sub\Bun_{\GL_n}$ to denote the connected component consisting of vector bundles of degree $d$.

We fix a geometrically irreducible Weil local system $\s_n\in\Loc_{\GL_n}^{\arith}(k)$. There is an associated Hecke eigensheaf $\LL_{\s_n}^{\FGV}\in\Shv_{\Nilp}(\Bun_{\GL_n})$ constructed in \cite{frenkel2002geometric}. We refer to \cite[\S7.2]{liu2025higherperiodintegralsderivatives} for a summary of its properties and normalization (which is slightly different from the normalization in \cite{frenkel2002geometric} by a twist).

In this case, the construction in \S\ref{sec:isotypicpartscc2} makes sense. Indeed, consider the map $f_n=\det:\Bun_{\GL_n}\to\Bun_{\Gm}$. The following proposition verifies Assumption \ref{assumption:eigencompact}.
\begin{prop}\label{prop:glneigencompact}
    The functor \[f_n^*(-)\otimes\LL_{\s_n}^{\FGV}:\Shv_{\Nilp}(\Bun_{\Gm})\to\Shv_{\Nilp}(\Bun_{\GL_n})\] preserves compact objects.
\end{prop}
The proof will be given in \S\ref{sec:glneigencompactproof}.

For each $d\in\ZZ$ and $V^I\in\Rep(\GL_n^I)$, the construction in \S\ref{sec:isotypicpartscc2} gives a map \[\xi_{\s_n,I}^d:V_{\s_n}^I\to l_{I,!}(\IC_{V^I}|_{\Sht_{\GL_n,I}^d}).\] In particular, taking $V^I=\Std_n^{\underline{\e}}$, for each $\underline{\e}\in\{\pm 1\}^r_0$, we get a map \begin{equation}
    \xi_{\s_n,\underline{\e}}=(\xi_{\s_n,\underline{\e}}^d)_{d\in\ZZ}:=(\xi_{\s_n,I}^d)_{d\in\ZZ}:\s_n^{\underline{\e}}\to \prod_{d\in\ZZ}l_{I,!}(\IC_{\Std_{n}^{\underline{\e}}}|_{\Sht_{\GL_n,I}^d})
.\end{equation}

In our application, we consider $G=\GL_n\times\GL_{n-1}$. For geometrically irreducible Weil local system $\s=(\s_n,\s_{n-1})\in\Loc_{\GL_n\times \GL_{n-1}}^{\arith}(k)$, we consider the functor \begin{equation}
    f_n\times f_{n-1}:\Bun_{\GL_n\times\GL_{n-1}}\to\Bun_{\Gm^2}
\end{equation} and use the Hecke eigensheaf \[\LL_{\s}:=\LL_{\s_n}^{\FGV}\boxtimes\LL_{\s_{n-1}}^{\FGV}\in\Shv_{\Nilp}(\Bun_{\GL_n\times\GL_{n-1}}).\] The same construction gives \begin{equation}\label{eq:app:shtcohisotypicpart}
    \xi_{\s,\underline{\e}}=(\xi_{\s,\underline{\e}}^{(d_n,d_{n-1})})_{(d_n,d_{n-1})\in\ZZ^2}:(\s_n\otimes\s_{n-1})^{\underline{\e}}\to \prod_{(d_n,d_{n-1})\in\ZZ^2}l_{I,!}(\IC_{(\Std_{n}\boxtimes\Std_{n-1})^{\underline{\e}}}|_{\Sht_{\GL_n\times\GL_{n-1},I}^{(d_n,d_{n-1})}})
.\end{equation} This is an enhancement of the map \eqref{eq:intro:spectralsub3} which recovers the map \eqref{eq:intro:spectralsub3} by taking global section. The injectivity of the map \eqref{eq:intro:spectralsub3} is a consequence of Theorem  \ref{thm:geoLMglntrsht}.

\subsubsection{Proof of Proposition \ref{prop:glneigencompact}}\label{sec:glneigencompactproof}
In this section, we prove Proposition \ref{prop:glneigencompact}. 
For an irreducible $\SL_n$-local system $\s_n\in\Loc_{\SL_n}(k)$, regarding it as a $\GL_n$-local system, the associated Hecke eigensheaf $\LL_{\s_n}^{\FGV}$ descends to a Hecke eigensheaf $\LL_{\s_n,\PGL_n}^{\FGV}\in \Shv_{\Nilp}(\Bun_{\PGL_n})$. That is, consider the map $h:\Bun_{\GL_n}\to\Bun_{\PGL_n}$, we have $\LL_{\s_n}^{\FGV}\cong h^*\LL_{\s_n,\PGL_n}^{\FGV}$.

We have the following lemma:
\begin{lemma}\label{lem:pglneigencompact}
    For any irreducible $\SL_n$-local system $\s_n\in\Loc_{\SL_n}(k)$, the corresponding Hecke eigensheaf $\LL_{\s_n,\PGL_n}^{\FGV}\in\Shv_{\Nilp}(\Bun_{\PGL_n})$ is compact.
\end{lemma}

\begin{proof}[Proof of Lemma \ref{lem:pglneigencompact}]
    This is a direct consequence of \cite{GR}. To make our result minimally depend on characteristic zero techniques, we provide an independent proof. Consider the map $i_{\s_n}:*\to \Loc_{\SL_n}$ induced by the point $\s_n\in\Loc_{\SL_n}(k)$. Since $\LL_{\s_n,\PGL_n}^{\FGV}$ agrees with $(i_{\s_n}k)*\cP_{\mathrm{Whit}}$ up to a cohomological shift, we only need to show that $(i_{\s_n}k)*\cP_{\mathrm{Whit}}$ is compact. Here \[\cP_{\mathrm{Whit}}\in\Shv(\Bun_{\PGL_n})\] is the Whittaker sheaf and \[*:\QCoh(\Loc_{\SL_n})\otimes\Shv(\Bun_{\PGL_n})\to\Shv_{\Nilp}(\Bun_{\PGL_n})\] is the spectral action. Since $i_{\s_n}k$ is compact, $\cP_{\mathrm{Whit}}$ is compact, and $-*-$ preserves compact objects, we know that $(i_{\s_n}k)*\cP_{\mathrm{Whit}}$ is compact. This concludes the proof of Lemma \ref{lem:pglneigencompact}.
\end{proof}

\begin{proof}[Proof of Proposition \ref{prop:glneigencompact}]
    By tensoring with a rank $1$ local system on $C$, we can assume $\det\s_n\cong\uk_{C}$. By Theorem  \ref{thm:geo=cat:spectralprojector}, compactness in $\Shv$ and $\Shv_{\Nilp}$ are the same, and we can safely discard the singular support condition everywhere. Consider the commutative diagram \[\begin{tikzcd}\Bun_{\Gm}\times\Bun_{\SL_n}\ar[r, "m"] \ar[d, "\id\times h"] & \Bun_{\GL_n} \ar[r, "f_n"] \ar[d, "h"] & \Bun_{\Gm} \\ \Bun_{\Gm}\times\Bun_{\PGL_n} \ar[r, "\pr_2"] & \Bun_{\PGL_n}
    \end{tikzcd}\] in which $m(\cL,\cE_n)=\cE_n\otimes\cL$. Since for any $\cF\in\Shv(\Bun_{\GL_n})$ one has $m_*m^*\cF\cong m_!m^*\cF\cong \cF\otimes m_!\uk_{\Bun_{\Gm}\times\Bun_{\SL_n}}$ containing $\cF$ as a direct summand, we know that the object $\cF$ is compact whenever $m^*\cF$ is compact. Therefore, we only need to check that \[m^*\circ (f_n^*(-)\otimes\LL_{\s_n}^{\FGV})\cong [n]^*(-)\boxtimes h^*\LL_{\s_n,\PGL_n}^{\FGV} :\Shv(\Bun_{\Gm})\to\Shv(\Bun_{\Gm}\times\Bun_{\SL_n})\] preserves compact objects. Here $[n]:\Bun_{\Gm}\to\Bun_{\Gm}$ is the $n$-th power map. This follows from Lemma \ref{lem:pglneigencompact} because both $[n]^*:\Shv(\Bun_{\Gm})\to \Shv(\Bun_{\Gm})$ and $h^*:\Shv(\Bun_{\PGL_n})\to\Shv(\Bun_{\SL_n})$ preserve compact objects. 
\end{proof}

\subsection{Rankin--Selberg cycles}\label{sec:app:rscycle}

Now we come to study the Rankin--Selberg special cycle classes \eqref{eq:intro:rscycles}. In particular, we would like to study their restriction along the map \eqref{eq:app:shtcohisotypicpart} using tools in \S\ref{sec:fakevsreal2}.

In this section, we take $G=\GL_n\times\GL_{n-1}$, $H=\GL_{n-1}$, $X=H\backslash G$. Define $L=\Gamma(\s_n\otimes\s_{n-1})\langle 1\rangle$, which is an odd vector space.  Let $L^{\underline{\e}}=L^{\e_1}\otimes\cdots\otimes L^{\e_r}$. Define $M=L\oplus L^{*}$. We have $M^{\otimes r}=\bigoplus_{\underline{\e}\in \{\pm 1\}^r} L^{\underline{\e}}$. Define $(M^{\otimes r})_0=\bigoplus_{\underline{\e}\in \{\pm 1\}^r_0} L^{\underline{\e}}\sub M^{\otimes r}$. We would like to understand the elements \begin{equation}\label{eq:app:isotypicrscycle1}
(\pi_{\Sht,I,!}[\Sht_{\GL_{n-1},\Std_{n-1}^{\underline{\e}}}^d])_{\s}=
    \pi_{\Sht,I,!}[\Sht_{\GL_{n-1},\Std_{n-1}^{\underline{\e}}}^d]\circ{\xi_{\s,\underline{\e}}}= \pi_{\Sht,I,!}[\Sht_{\GL_{n-1},\Std_{n-1}^{\underline{\e}}}]\circ\xi_{\s,\underline{\e}}^{(d,d)}\in L^{\underline{\e},*}
\end{equation} defined in \eqref{eq:intro:isotypicrscycles}. Or slightly weaker, we would like to understand the element \begin{equation}\label{eq:app:isotypicrscycle2}
    ((\pi_{\Sht,I,!}[\Sht_{\GL_{n-1},\Std_{n-1}^{\underline{\e}}}])_{\s})_{\underline{\e}\in\{\pm 1\}^r_0}\in (M^{\otimes r})^*_0
\end{equation} whose components are defined in \eqref{eq:intro:totalrscycle}.

Take the cohomological correspondence \begin{equation}\label{eq:app:rscc}\begin{split}\frc_{\Std_{n-1}^{\underline{\e}}}=\pi_{\Hk,I,!}[\Hk_{\GL_{n-1},\Std_{n-1}^{\underline{\e}}}/\Bun_{\GL_{n-1}}\times C^I] \\ \in \\ \Cor_{\Hk_{\GL_n\times\GL_{n-1},I},\IC_{(\Std_n\boxtimes\Std_{n-1})^{\underline{\e}}}\langle -r\rangle}(\cP_X\boxtimes\uk_{C^I},\cP_X\boxtimes\uk_{C^I})\end{split}\end{equation} in Theorem \ref{thm:mingeo=cyc}. In this case, Assumption \ref{assumption:goodfil} is satisfied by Example \ref{eg:rsgoodfil}. Therefore, Theorem \ref{thm:mingeo=cyc} gives the following:
\begin{prop}\label{prop:app:geo=cyclers}
    We have \[\pi_{\Sht,I,!}[\Sht_{\GL_{n-1},\Std_{n-1}^{\underline{\e}}}]=\tr_{\Sht,C^I}(\frc_{\Std_{n-1}^{\underline{\e}}}).\] In particular, we have \[(\pi_{\Sht,I,!}[\Sht_{\GL_{n-1},\Std_{n-1}^{\underline{\e}}}^d])_{\s}=\tr_{\Sht,C^I}(\frc_{\Std_{n-1}^{\underline{\e}}})\circ \xi_{\s,\underline{\e}}^{(d,d)} \in L^{\underline{\e},*}.\] 
\end{prop}

Since Assumption \ref{assumption:vectcompact} is true by Theorem  \ref{thm:geors}, we have the fake special cycle classes \begin{equation} z_{\frc_{\Std_{n-1}^{\underline{\e}}},\s}^{(d,d)}\in L^{\underline{\e},*}\end{equation} defined in \eqref{eq:fakespecialcycleclasses} (see refinement in \eqref{eq:fakespecialcycleclassesref}). Note that Assumption \ref{assumption:eigencompact} is true by Proposition \ref{prop:glneigencompact}, we can apply Proposition \ref{prop:computationofisotypicss} and get the following:
\begin{prop}\label{prop:app:geo=fakers}
    We have $\tr_{\Sht,C^I}(\frc_{\Std_{n-1}^{\underline{\e}}})\circ \xi_{\s,\underline{\e}}^{(d,d)}=z_{\frc_{\Std_{n-1}^{\underline{\e}}},\s}^{(d,d)}\in L^{\underline{\e},*}$.
\end{prop}

Combining Proposition \ref{prop:app:geo=cyclers} and Proposition \ref{prop:app:geo=fakers}, we are reduced to understand the fake special cycle classes $z_{\frc_{\Std_{n-1}^{\underline{\e}}},\s}^{(d,d)}\in L^{\underline{\e},*}$. This is the subject of \cite{liu2025higherperiodintegralsderivatives}, where we recollect the key results below.

Consider \[z_{\frc_{\Std_{n-1}^{\underline{\e}}},\s}=\sum_{d\in\ZZ}z_{\frc_{\Std_{n-1}^{\underline{\e}}},\s}^{(d,d)}\in L^{\underline{\e},*}\] and \begin{equation}\label{eq:app:faketotal}z_{\s,r}:=\sum_{\underline{\e}\in\{\pm 1\}_0^r}z_{\frc_{\Std_{n-1}^{\underline{\e}}},\s}\in (M^{\otimes r})_0^*\sub (M^{\otimes r})^*.\end{equation} By the definition in \eqref{eq:fakespecialcycleclasses}, these elements are defined by \begin{equation}\label{eq:app:fakespecialcyclersdef}
    z_{\s,r}(m_1\otimes\cdots \otimes m_r)=\tr( a^{\otimes}(m_1\otimes\cdots\otimes m_r)\circ\Frob, \int_{X,\Nilp}\LL_{\s})
\end{equation} in which the map $a^{\otimes}$ is \eqref{eq:rsisotypicheckeaction12}.

By Theorem  \ref{thm:geors}, the sequence of elements $\{z_{\s,\bullet}\}$ form a Kolyvagin system in the sense of \cite[\S3]{liu2025higherperiodintegralsderivatives}. Note that the (possibly infinite) sum above involves only finitely many non-zero terms since $\int_{X,\Nilp}\LL_{\s_n}^{\FGV,d}=0$ for all but finitely many $d$ by Theorem \ref{thm:geors}. This confirms the first claim in Conjecture \ref{conj:intro:main} in this setting.

Now we recollect some key properties of this Kolyvagin system. Let $K=\Std_n\boxtimes\Std_{n-1}\oplus\Std_{n}^*\boxtimes\Std_{n-1}^*$. It is equipped with a natural symplectic pairing \begin{equation}\label{eq:app:omk}\om_{K}=\ev_{\Std_n\boxtimes\Std_{n-1}}-\ev_{\Std_{n}^*\boxtimes\Std_{n-1}^*}\end{equation} in which $\ev_{\Std_n\boxtimes\Std_{n-1}}:(\Std_n\boxtimes\Std_{n-1})\otimes(\Std_{n}^*\boxtimes\Std_{n-1}^*)\to k$ is the natural evaluation map and similarly for $\ev_{\Std_{n}^*\boxtimes\Std_{n-1}^*}$. This gives rise to a symplectic pairing on the odd vector space $M$, which we denote by $\om_{M}$. Note that symplecticity here means $\om_{M}(m_1,m_2)=\om_{M}(m_2,m_1)$ since $M$ is odd. The bilinear form $\om_{M}$ induces bilinear form $\om_{M^*}$ by identifying $M^*$ with $M$ using $\om_M$. This induces a bilinear form $\om_{(M^{\otimes r})^*}$ on $(M^{\otimes r})^*$.

We have the following result:
\begin{thm}\cite[Theorem\,1.2]{liu2025higherperiodintegralsderivatives}\label{thm:app:kolydermain}
    We have \[\om_{(M^{\otimes r})^*}(z_{\s,r},z_{\s,r})=\b_{\s}(\ln q)^{-r}\left(\frac{d}{ds}\right)^{r}\Big|_{s=1/2}\widetilde{L}(\s_n\otimes\s_{n-1}\oplus\s_n^*\otimes\s_{n-1}^*,s)\] where $\beta_{\s}=(-1)^{r/2}q^{-n^2(g-1)}\chi^{-n}_{\det\s_{n-1}}(\Om)\chi^{-n+1}_{\det\s_n}(\Om)\e(\s_n\otimes\s_{n-1})$ and \[\widetilde{L}(\s_n\otimes\s_{n-1}\oplus\s_n^*\otimes\s_{n-1}^*,s)=q^{2n(n-1)(g-1)(s-1/2)}L(\s_{n}\otimes\s_{n-1},s)L(\s_{n}^*\otimes\s_{n-1}^*,s)\] is the normalized $L$-function defined in \eqref{eq:intro:normlfunc}.
\end{thm}

\subsubsection{Geometric result}
In this section, we collect all relevant geometric results relating to Rankin--Selberg integrals for the convenience of readers.
Take $\underline{\e}=(1),(-1)$ in \eqref{eq:app:rscc}, we have cohomological correspondences $\frc_{\Std_{n-1}}$ and $\frc_{\Std_{n-1}^*}$. They give rise to Hecke actions on $\int_{X,\Nilp}\LL_{\s}^{(d,d)}$ defined in \eqref{eq:heckeactiononisotypicpart}
\begin{equation} \label{eq:rsisotypicheckeaction1} a_{\frc_{\Std_{n-1}},\s}:L\otimes\int_{X,\Nilp}\LL_{\s}^{(d,d)}\to \int_{X,\Nilp}\LL_{\s}^{(d+1,d+1)}.\end{equation}\begin{equation}\label{eq:rsisotypicheckeaction2} a_{\frc_{\Std_{n-1}^*},\s}:L^*\otimes\int_{X,\Nilp}\LL_{\s}^{(d,d)}\to \int_{X,\Nilp}\LL_{\s}^{(d-1,d-1)}.\end{equation} Combining these two actions, we get an action \begin{equation}\label{eq:rsisotypicheckeaction12}
    a^{\otimes}: M^{\otimes} \to \End( \int_{X,\Nilp}\LL_{\s}).
\end{equation}

Let $\om_{M,X}=(-1)^{n-1}\om_{M}$.\footnote{The role of $\om_M$ and $\om_{M,X}$ are different here: The bilinear form $\om_M$ (independent of $X$) is directly related to the intersection pairing on the isotypic part. The bilinear form $\om_{M,X}$ (dependent on $X$) arises from the Poisson structure on the Plancherel algebra attached to $X$ and controls the Kolyvagin system. The sign in the difference between these two forms can be traced back to \cite[Proposition\,7.1]{liu2025higherperiodintegralsderivatives}.}Define the Clifford algebra $\Cl(M)$ to be the quotient of the tensor algebra $M^{\otimes}$ by the two-sided ideal generated by elements of the form $m_1\otimes m_2+m_2\otimes m_1-\om_{M,X}(m_1,m_2)$. We have the following result:

\begin{thm}\cite[Theorem\,7.8]{liu2025higherperiodintegralsderivatives}\label{thm:geors}
    The action map \eqref{eq:rsisotypicheckeaction12} factors through the Clifford algebra $\Cl(M)$. Moreover, we have \[\int_{X,\Nilp}\LL_{\s}^{d}\cong (\Sym^{d+n(n-1)(g-1)}L) \otimes (\LL_{\det\s_{n-1}})_{\Om^{-n/2}}\otimes (\LL_{\det\s_n})_{\Om^{-(n-1)/2}}\langle(n^2-2)(g-1) \rangle\] for each $d\in\ZZ$. Here, $\LL_{\det\s_n},\LL_{\det\s_{n-1}}\in\Shv(\Bun_{\Gm})$ are the Hecke eigensheaves associated to the rank one local systems $\det\s_n,\det\s_{n-1}$.
\end{thm}

This theorem has the following corollary:
\begin{cor}\label{cor:app:rscliffmodule}
    We have \[\int_{X,\Nilp}\LL_{\s}^{-n(n-1)(g-1)}\cong (\LL_{\det\s_{n-1}})_{\Om^{-n/2}}\otimes (\LL_{\det\s_n})_{\Om^{-(n-1)/2}}\langle(n^2-2)(g-1) \rangle\] as a one-dimensional vector space with Frobenius action. Moreover, the action map \eqref{eq:rsisotypicheckeaction12} induces an isomorphism \[\Sym^{\bullet}L\otimes \int_{X,\Nilp}\LL_{\s}^{-n(n-1)(g-1)}\cong \int_{X,\Nilp}\LL_{\s}. \]
\end{cor}

Using the language of \cite[\S3.4]{liu2025higherperiodintegralsderivatives}, we have the following corollary:
\begin{cor}\label{cor:app:highestwt}
    The $\Cl(M)$-module $\int_{X,\Nilp}\LL_{\s}$ is a lowest weight module with Frobenius eigenvalue on the lowest weight vector equal to $q^{-n^2(g-1)/2}\chi^{-n}_{\det\s_{n-1}}(\Om^{1/2})\chi^{-n+1}_{\det\s_n}(\Om^{1/2})$. Here, $\chi_{\det\s_{n-1}},\chi_{\det\s_n}:\Pic(\FF_q)\to k^{\times}$ are the Hecke characters associated to local systems $\det\s_{n-1},\det\s_n$.
\end{cor}

\begin{remark}\label{rmk:app:lowestwt}
    Dually, one has an isomorphism \[\Sym^{\bullet}L^*\otimes \int_{X,\Nilp}\LL_{\s}^{n(n-1)(g-1)}\cong \int_{X,\Nilp}\LL_{\s} \] where \[\int_{X,\Nilp}\LL_{\s}^{n(n-1)(g-1)}\cong \det(\Gamma(\s_n\otimes\s_{n-1})\langle 1\rangle)\otimes (\LL_{\det\s_{n-1}})_{\Om^{-n/2}}\otimes (\LL_{\det\s_n})_{\Om^{-(n-1)/2}}\langle(n^2-2)(g-1) \rangle.\] This implies that the $\Cl(M)$ is a highest weight module with Frobenius eigenvalue on the highest weight vector equal to $q^{-n^2(g-1)/2}\chi^{-n}_{\det\s_{n-1}}(\Om^{1/2})\chi^{-n+1}_{\det\s_n}(\Om^{1/2})\e(\s_n\otimes\s_{n-1})$
\end{remark}

We would also like to mention a relevant result, which will be used in the study of diagonal cycle in \S\ref{sec:application:diagonalcycle}.

For any $\cL\in\Bun_{\Gm}(\FF_q)$, let $\Bun_{\GL_n}'^{\cL}=\{(\Om^{(n-1)/2}\otimes\cL\sub\cE)|\cE\in\Bun_{\GL_n}\}$ be the moduli stack of a rank $n$ vector bundle $\cE$ together with an injection of coherent sheaves $\Om^{(n-1)/2}\otimes\cL\sub\cE$. Consider the natural forgetful map $\rho_{n}^{\cL}:\Bun_{\GL_n}'^{\cL}\to \Bun_{\GL_n}$ defined by $\rho_{n}^{\cL}(\Om^{(n-1)/2}\otimes\cL\sub\cE)=\cE$. It restricts to maps on connected components $\rho_{n}^{\cL,d}:\Bun_{\GL_n}'^{\cL,d}\to\Bun_{\GL_n}^d$.
\begin{thm}\cite{lysenko2002local}\label{thm:equalrkrs}
    For any geometrically irreducible local systems $\s_n,\s_n'\in\Loc_{\GL_n}(k)$, there is a canonical isomorphism \[\begin{split}
    \Gamma_c(\rho_{n}^{\cL,d,*}(\LL_{\s_n}^{\FGV,d}\otimes \LL_{\s_n'}^{\FGV,d}))\\ \cong \\ \Sym^{d_{\cL}}(\Gamma(\s_n\otimes\s_n')\langle 2\rangle)\langle-2(d_{\cL}-n(g-1))\rangle\otimes (\LL_{\det\s_n})_{\cL}\langle-(g-1)\rangle\otimes(\LL_{\det\s_n'})_{\cL}\langle-(g-1)\rangle
    \end{split}.\] where $d_{\cL}=d-n\deg\cL\otimes\Om^{(n-1)/2}$.
\end{thm}

\subsection{Diagonal cycles}\label{sec:application:diagonalcycle}
We are now going to use techniques from \S\ref{sec:isotypicdiag} to study the diagonal cycle classes. In this section, we take $H=\GL_n$, $G=H\times H$, and $X=H \backslash G$. For any $V^I\in\Rep(\GL_n^I)^{\heartsuit}$ and $d\in\ZZ$, we
would like to understand the diagonal cycle \begin{equation}
    \langle-,-\rangle_{V^I}^{d}=\D_{\Hk,I,!}[\Sht_{\GL_n,V^I}^d/C^I]\in \Hom^0((l_{I,!}(\IC_{V^I}|_{\Sht_{\GL_n,I}^d}))^{\otimes 2},\uk_{C^I} ).
\end{equation} That is, for any geometrically irreducible Weil local system $\s_n\in\Loc_{\GL_n}^{\arith}(k)$, we want to understand its restriction to the $\s_n$-isotypic part  \begin{equation}\begin{split}
    \langle-,-\rangle_{V^I,\s_n}^{d}&=\D_{\Hk,I,!}[\Sht_{\GL_n,V^I}^d/C^I]\circ (\xi_{\s_n,I}\otimes\xi_{\s_n^*,I}) \\&=\D_{\Hk,I,!}[\Sht_{\GL_n,V^I}/C^I]\circ (\xi_{\s_n,I}^d\otimes\xi_{\s_n^*,I}^d)\\&\in\Hom^0(V_{\s_n}^{I}\otimes V_{\s_n^*}^I,\uk_{C^I})
\end{split}.\end{equation} Note that when $V^I=\Std_n^{\underline{\e}}$, this element is \eqref{eq:intro:diagisotypiccycle}.

For this, we take the diagonal cohomological correspondence defined in \eqref{eq:geo=cycle:diagcc}  \begin{equation}\frc_{V^I}\in\Cor_{\Hk_{G,I},\IC_{V^I\boxtimes V^I}}(\cP_X\boxtimes\uk_{C^I},\cP_X\boxtimes\uk_{C^I}).\end{equation} The result of Theorem  \ref{thm:diaggeo=cyc} gives us \begin{prop}\label{prop:app:geo=cyclediaggln}
    We have \[\D_{\Hk,I,!}[\Sht_{\GL_n,V^I}/C^I]=\tr_{\Sht,C^I}(\frc_{V^I}).\] In particular, we have \[\langle-,-\rangle_{V^I,\s_n}^{d}=\tr_{\Sht,C^I}(\frc_{V^I})\circ (\xi_{\s_n,I}^d\otimes\xi_{\s_n^*,I}^d).\]
\end{prop}

Since $\LL_{\s_n^*}^{\FGV}\cong\DD(\LL_{\s_n}^{\FGV})$, we can apply results in \S\ref{sec:isotypicdiag}, in particular Conjecture \ref{conj:intersection}. By Proposition \ref{prop:intersectioncc}, we only need to verify Assumption \ref{assumption:diagonalgenerate}. In our case, this is provided by Theorem  \ref{thm:geoLMgln}. Therefore, we obtain the following:

\begin{thm}\label{thm:geoLMglntrsht}
    For every $d\in\ZZ$ and $V^I\in\Rep(\GL_n^I)$, we have \begin{equation}\label{eq:app:geoLMnn-1trsht}\tr_{\Sht,C^I}(\frc_{ V^I})\circ(\xi_{\s_n,I}^d\otimes\xi_{\s_n^*,I}^d)=(\ln q)\cdot \Res_{s=1} L(\s_n\otimes\s_n^*,s)\cdot  \ev_{V_{\s_n}^I} \in \Hom^0(V_{\s_n}^I\otimes V_{\s_n^*}^I,\uk_{C^I}).\end{equation} Here $ \ev_{V_{\s_n}^I}:V_{\s_n}^I\otimes V_{\s_n^*}^I\to \uk_{C^I}$ is the evaluation map.
\end{thm}

The proof will be given in \S\ref{sec:app:proofdiag}.

Combing Proposition \ref{prop:app:geo=cyclediaggln} and Theorem  \ref{thm:geoLMglntrsht}, we obtain
\begin{cor}\label{cor:app:nondegenracy}
    For any irreducible representation $V^I\in\Rep(\GL_n^I)^{\heartsuit}$ and $d\in\ZZ$, the pairing \[\langle-,-\rangle_{{V^I},\s_n}^d:V_{\s_n}^I\otimes V_{\s_n^*}^I\to\uk_{C^I}\] is non-degenerate. 
\end{cor}

In our application, we would like to take $H=\GL_n\times\GL_{n-1}$. In this case, for each $(d_n,d_{n-1})\in\ZZ^2$, we study the diagonal cycle \begin{equation}\begin{split}
    \langle-,-\rangle_{\underline{\e}}^{(d_n,d_{n-1})}=\D_{\Hk,I,!}[\Sht_{\GL_n\times\GL_{n-1},(\Std_n\boxtimes\Std_{n-1})^{\underline{\e}}}^{(d_n,d_{n-1})}/C^I] \\ \in  \Hom^0((l_{I,!}(\IC_{(\Std_n\boxtimes\Std_{n-1})^{\underline{\e}}}|_{\Sht_{\GL_n\times\GL_{n-1},I}}^{(d_n,d_{n-1})}))^{\otimes 2},\uk_{C^I} )
\end{split}\end{equation} and the diagonal cohomological correspondence \[\frc_{(\Std_n\boxtimes\Std_{n-1})^{\underline{\e}}}\] as well as the restriction of the intersection pairing to $\s$-isotypic part \begin{equation}
    \langle-,-\rangle_{\underline{\e},\s}^{(d_n,d_{n-1})}=\langle-,-\rangle_{\underline{\e}}^{(d_n,d_{n-1})}\circ(\xi_{\s,\underline{\e}}\otimes\xi_{\s^*,\underline{\e}})\in\Hom^0((\s_n\otimes\s_{n-1})^{\underline{\e}}\otimes (\s_n^*\otimes\s_{n-1}^*)^{\underline{\e}},\uk_{C^I})
\end{equation} where $\s=(\s_n,\s_{n-1})$ and $\s^*=(\s_n^*,\s_{n-1}^*)$. Then results similar to Proposition \ref{prop:app:geo=cyclediaggln} and Theorem  \ref{thm:geoLMglntrsht} hold with Theorem  \ref{thm:geoLMglntrsht} replaced by \begin{equation}\label{eq:app:nn-1LM}\begin{split}
    \tr_{\Sht,C^I}(\frc_{(\Std_n\boxtimes\Std_{n-1})^{\underline{\e}}})\circ(\xi_{\s,\underline{\e}}^{(d_n,d_{n-1})}\otimes\xi_{\s^*,\underline{\e}}^{(d_n,d_{n-1})}) \\ =(\ln q)^2\cdot \Res_{s=1} L(\s_n\otimes\s_n^*,s)\Res_{s=1} L(\s_{n-1}\otimes\s_{n-1}^*,s)\cdot  \ev_{(\s_n\otimes\s_{n-1})^{\underline{\e}}} \\ \in \Hom^0((\s_n\otimes\s_{n-1})^{\underline{\e}}\otimes (\s_n^*\otimes\s_{n-1}^*)^{\underline{\e}},\uk_{C^I}).
\end{split}\end{equation}

\subsubsection{Geometric result}
To apply Conjecture \ref{conj:intersection}, we need to verify Assumption \ref{assumption:diagonalgenerate} to apply Proposition \ref{prop:intersectioncc}. We also need to compute the scalar $\tr(\Frob,\Gamma_c(\LL_{\s_n}^{\FGV,d}\otimes\LL_{\s_n^*}^{\FGV,d}))$ involved in Conjecture \ref{conj:intersection}. All of these will follow from a description of $\Gamma_c(\LL_{\s_n}^{\FGV}\otimes\LL_{\s_n^*}^{\FGV})$ which we are going to give.

Take $V=\Std_n\in\Rep(\GL_n)$ and the diagonal cohomological correspondence $\frc_{\Std_n\boxtimes\Std_n}$, we obtain a Hecke action map introduced in \S\ref{sec:heckeaction} \begin{equation}\label{eq:glndiagheckeaction} a_{\Std_n\boxtimes\Std_n,\s_n\boxtimes\s_n^*}:\Gamma(\s_n\otimes\s_n^*)\langle 2\rangle \otimes \Gamma_c(\LL_{\s_n}^{\FGV}\otimes\LL_{\s_n^*}^{\FGV})\to \Gamma_c(\LL_{\s_n}^{\FGV}\otimes\LL_{\s_n^*}^{\FGV}) .\end{equation}

Recall we have the canonical elements \[\ev_{\LL_{\s_n}^{\FGV,d}}\in H^0(\Gamma_c(\LL_{\s_n}^{\FGV,d}\otimes\LL_{\s_n^*}^{\FGV,d}))^*\] defined in \eqref{eq:canonicalelement}. We also have the fundamental class $\coev_{\s_n}([C])\in H^0(\Gamma(\s_n\otimes\s_n^*)\langle 2\rangle)$. 

The following result completely describes $\Gamma_c(\LL_{\s_n}^{\FGV}\otimes\LL_{\s_n^*}^{\FGV})$:

\begin{thm}\label{thm:geoLMgln}
For every $d\in \ZZ$, we have $\dim H^0(\Gamma_c(\LL_{\s_n}^{\FGV,d}\otimes\LL_{\s_n^*}^{\FGV,d}))=1$, and the action map \eqref{eq:glndiagheckeaction} induces an isomorphism of graded vector spaces \begin{equation}\label{eq:geoLMgln}H^*(\Sym^{\bullet}(\Gamma(\s_n\otimes\s_n^*)\langle 2\rangle))[\frac{1}{\coev_{\s_n}([C])}]\cdot H^0(\Gamma_c(\LL_{\s_n}^{\FGV,d}\otimes\LL_{\s_n^*}^{\FGV,d}))\isom H^*(\Gamma_c(\LL_{\s_n}^{\FGV}\otimes\LL_{\s_n^*}^{\FGV})).\end{equation}
\end{thm}

\begin{proof}[Proof of Theorem  \ref{thm:geoLMgln}]
First note that both sides of \eqref{eq:geoLMgln} carry bi-gradings defined as follows: on the right-hand side, we set \[\deg H^i(\Gamma_c(\LL_{\s_n}^{\FGV,e}\otimes\LL_{\s_n^*}^{\FGV,e}))=(i,e).\] On the left hand side, we have \[\deg H^i(\Sym^{e}(\Gamma(\s_n\otimes\s_n^*)\langle 2\rangle))=(i,e).\] 

We first prove the following:
\begin{lemma}
    Both sides of \eqref{eq:geoLMgln} have the same bi-graded dimension. 
\end{lemma}
\begin{proof}[Proof of lemma]
This follows from \cite[Main Global Theorem]{lysenko2001globalgeometrisedrankinselbergmethod}. For the convenience of readers, we reproduce the proof here.

For each $d\in\ZZ$, there exists a quasi-compact open substack $j:U\hookrightarrow\Bun_{\GL_n}^d$ and an integer $e\in\ZZ$ such that \begin{itemize}
    \item $\LL_{\s_n}^{\FGV,d}$ is a clean extension from $U\sub\Bun_{\GL_n}^d$;
    \item For any $\cE\in U$ and $\cL\in\Pic$ with $\deg \cL\leq e$ , we have $H^1(\Hom(\cL\otimes\Om^{(n-1)/2},\cE))=0$.
\end{itemize}
Take $U$,$e$ as above. For any $\cL$ as above, after restricting to $U$, the map $\rho_{n}^{\cL,d}$ is the complementary of zero section in a smooth vector bundle of rank $H^0(\Hom(\cL\otimes\Om^{(n-1)/2},\cE))=d_{\cL}-n(g-1)$. Therefore, we have a fiber sequence of objects in $\Shv(U)$: \begin{equation}
   \cF\to  (\rho_{n,!}^{\cL,d}\uk_{\Bun_{\GL_n}'^{\cL,d}})|_{U}\langle 2(d_{\cL}-n(g-1)) \rangle\to \uk_U.
\end{equation} where $\cF\in\Shv(U)^{\leq -(d_{\cL}-n(g-1))}$. Applying the functor $\Gamma_c(\LL_{\s_n}^{\FGV,d}\otimes\LL_{\s_n^{*}}^{\FGV,d}\otimes j_!(-))$ to this fiber sequence, we obtain a fiber sequence
\begin{equation}
    \Gamma_c(j^*(\LL_{\s_n}^{\FGV,d}\otimes\LL_{\s_n^{*}}^{\FGV,d})\otimes\cF)\to \Gamma_c(\rho_{n}^{\cL,d,*}(\LL_{\s_n}^{\FGV,d}\otimes\LL_{\s_n^{*}}^{\FGV,d}))\langle 2(d_{\cL}-n(g-1))\rangle\to \Gamma_c(\LL_{\s_n}^{\FGV,d}\otimes\LL_{\s_n^{*}}^{\FGV,d})
\end{equation}
By Theorem  \ref{thm:equalrkrs}, we have \[\Gamma_c(\rho_{n}^{\cL,d,*}(\LL_{\s_n}^{\FGV,d}\otimes\LL_{\s_n^{*}}^{\FGV,d}))\langle 2(d_{\cL}-n(g-1))\rangle\cong \Sym^{d_{\cL}}(\Gamma(\s_n\otimes\s_n^*)\langle 2\rangle).\] Therefore, the fiber sequence above is identified with \[\Gamma_c(j^*(\LL_{\s_n}^{\FGV,d}\otimes\LL_{\s_n^{*}}^{\FGV,d})\otimes\cF)\to \Sym^{d_{\cL}}(\Gamma(\s_n\otimes\s_n^*)\langle 2\rangle)\to \Gamma_c(\LL_{\s_n}^{\FGV,d}\otimes\LL_{\s_n^{*}}^{\FGV,d})\]
Since the functor $\Gamma_c:\Shv(U)\to\Vect$ has bounded above cohomological dimension, by taking $\deg\cL$ sufficiently small, we can make the first term arbitrarily connective. This implies that for any $i\in\ZZ$, there are isomorphisms \[H^i(\Sym^{k}(\Gamma(\s_n\otimes\s_n^*)\langle 2\rangle))\isom H^i(\Gamma_c(\LL_{\s_n}^{\FGV,d}\otimes\LL_{\s_n^{*}}^{\FGV,d})) \] for sufficiently large $k\in\ZZ$. One easily sees that this implies that both sides of \eqref{eq:geoLMgln} have the same bi-graded dimension. This ends the proof of the lemma.
\end{proof}

In particular, we know that $\dim H^0(\Gamma_c(\LL_{\s_n}^{\FGV,d}\otimes\LL_{\s_n^*}^{\FGV,d}))=1$.

When $g(C)=1$ and $n>1$, such $\s_n$ does not exist, and there is nothing to prove. When $g(C)=1$ and $n=1$, one can verify the statement by hand. Therefore, we can assume $g(C)\neq 1$. In this case, by the proof of Lemma \ref{lem:diagcccom}, the action map \eqref{eq:glndiagheckeaction} induces a map \begin{equation}\label{eq:glndiagheckeactionsym} a_{\frc_{\Std_n\boxtimes\Std_n},\s_n\boxtimes\s_n^*}^{\Sym}:H^*(\Sym^{\bullet}(\Gamma(\s_n\otimes\s_n^*)\langle 2\rangle))\otimes H^*(\Gamma_c(\LL_{\s_n}^{\FGV}\otimes\LL_{\s_n^*}^{\FGV}))\to H^*(\Gamma_c(\LL_{\s_n}^{\FGV}\otimes\LL_{\s_n^*}^{\FGV})).\end{equation} This equips $H^*(\Gamma_c(\LL_{\s_n}^{\FGV}\otimes\LL_{\s_n^*}^{\FGV}))$ with a structure of $H^*(\Sym^{\bullet}(\Gamma(\s_n\otimes\s_n^*)\langle 2\rangle))$-module. We have the following lemma:

\begin{lemma}
    The induced map \begin{equation}\label{eq:glndiagheckeactioninj} H^*(\Sym^{\bullet}(\Gamma(\s_n\otimes\s_n^*)\langle 2\rangle))\cdot H^0(\Gamma_c(\LL_{\s_n}^{\FGV,d}\otimes\LL_{\s_n^*}^{\FGV,d}))\to H^*(\Gamma_c(\LL_{\s_n}^{\FGV}\otimes\LL_{\s_n^*}^{\FGV}))\end{equation} is an injection.
\end{lemma}
\begin{proof}[Proof of lemma]
    We claim that there exists another bi-graded action map \begin{equation}
        a_{\frc_{\Std_n^*\boxtimes\Std_n^*},\s_n\boxtimes\s_n^*}:\Gamma(\s_n^*\otimes\s_n)\otimes \Gamma_c(\LL_{\s_n}^{\FGV}\otimes\LL_{\s_n^*}^{\FGV})\to \Gamma_c(\LL_{\s_n}^{\FGV}\otimes\LL_{\s_n^*}^{\FGV})
    \end{equation} in which $\deg H^i(\Gamma(\s_n^*\otimes\s_n))=(i,-1)$ such that there exists a non-zero number $\kappa\in k^{\times}$ satisfying \begin{equation}\label{eq:glnLMclifford}
        a_{\frc_{\Std_n\boxtimes\Std_n},\s_n\boxtimes\s_n^*}\circ  a_{\frc_{\Std_n^*\boxtimes\Std_n^*},\s_n\boxtimes\s_n^*}-a_{\frc_{\Std_n^*\boxtimes\Std_n^*},\s_n\boxtimes\s_n^*}\circ a_{\frc_{\Std_n\boxtimes\Std_n},\s_n\boxtimes\s_n^*}=\kappa\cdot \ev_{\Gamma(\s_n\otimes\s_n^*)}\otimes\id
    \end{equation} as maps \[\Gamma(\s_n\otimes\s_n^*)\langle 2\rangle \otimes \Gamma(\s_n^*\otimes\s_n)\otimes \Gamma_c(\LL_{\s_n}^{\FGV}\otimes\LL_{\s_n^*}^{\FGV})\to \Gamma_c(\LL_{\s_n}^{\FGV}\otimes\LL_{\s_n^*}^{\FGV}).\] Here, the map $\ev_{\Gamma(\s_n\otimes\s_n^*)}: \Gamma(\s_n\otimes\s_n^*)\langle 2\rangle \otimes \Gamma(\s_n^*\otimes\s_n)\to k$ is the evaluation map of the duality induced by cup product. 
    
    This is a particular case of the automorphic commutator relation \cite[Corollary\,6.12]{liu2025higherperiodintegralsderivatives}. To explain this, we will freely use notations in \emph{loc.cit}. Consider the spherical variety $X'=\AA^n\times^{\GL_n}(\GL_n\times\GL_n)$ where $\GL_n$ acts on $\AA^n$ via the standard action. Note that we have a closed embedding $i:X=\GL_n\to X'$ given by the zero-section. This gives a functor $i^*:\Shv(LX'/L^+G\rtimes\Aut(D))\to\Shv(LX/L^+G\rtimes\Aut(D))$ as $\Sat_{G,\hbar}$-module categories. Since $i^*\d_{X'}=\d_{X}$, this induces a map between Plancherel algebras $\PL_{X',\hbar}\to\PL_{X,\hbar}$ as algebra objects in $\Rep(\GL_n\times\GL_n)$. By \cite{BFGT}, there are canonical local special cohomological correspondences for $X'$ \[\frc^{X',\loc}_{\Std_n\boxtimes\Std_n}\in \Hom^0(\Std_n\boxtimes\Std_n,\PL_{X',\hbar})\]  and \[\frc^{X',\loc}_{\Std_n^*\boxtimes\Std_n^*}\in \Hom^0(\Std_n^*\boxtimes\Std_n^*\langle -2\rangle,\PL_{X',\hbar})\] such that \[[\frc^{X',\loc}_{\Std_n\boxtimes\Std_n},\frc^{X',\loc}_{\Std_n^*\boxtimes\Std_n^*}]=\kappa\hbar\cdot \frc^{X',\loc}_{\triv}\] for some $\kappa\in k^{\times}$. Here $\frc^{X',\loc}_{\triv}\in\Hom^0(\triv,\PL_{X',\hbar})$ is the trivial cohomological correspondence. By composing with the map $\PL_{X',\hbar}\to\PL_{X,\hbar}$, we get local cohomological correspondences for $X$  \[\frc^{\loc}_{\Std_n\boxtimes\Std_n}\in \Hom^0(\Std_n\boxtimes\Std_n,\PL_{X,\hbar})\]  and \[\frc^{\loc}_{\Std_n^*\boxtimes\Std_n^*}\in \Hom^0(\Std_n^*\boxtimes\Std_n^*\langle -2\rangle,\PL_{X,\hbar})\] such that \[[\frc^{\loc}_{\Std_n\boxtimes\Std_n},\frc^{\loc}_{\Std_n^*\boxtimes\Std_n^*}]=\kappa\hbar\cdot \frc^{\loc}_{\triv}.\] Note that the local-to-global procedure in \cite[\S4.4]{liu2025higherperiodintegralsderivatives} produces from $\frc_{\Std_n\boxtimes\Std_n}^{\loc}$ the diagonal cohomological correspondence $\frc_{\Std_n\boxtimes\Std_n}$. Therefore, we can take $\frc_{\Std_n^*\boxtimes\Std_n^*}$ to be the globalization of $\frc_{\Std_n^*\boxtimes\Std_n^*}^{\loc}$, and take $a_{\frc_{\Std_n^*\boxtimes\Std_n^*},\s_n\boxtimes\s_n^*}$ to be the induced Hecke action. The identity \eqref{eq:glnLMclifford} follows from \cite[Corollary\,6.12]{liu2025higherperiodintegralsderivatives}.

    Now we prove the injectivity of the map in the lemma. Write \[H^0(\Gamma(\s_n\otimes\s_n^*)\langle 2\rangle)=k\cdot c_0\] \[H^{-1}(\Gamma(\s_n\otimes\s_n^*)\langle 2\rangle)=\bigoplus_{j}k\cdot c_{-1,j}\] \[H^{-2}(\Gamma(\s_n\otimes\s_n^*)\langle 2\rangle)=k\cdot c_{-2}\] where $c_0=\coev_{\s_n}([C])$. 
    Assuming the contrary, suppose $c_0^{h_0}f(c_{-1,1},c_{-1,2},\cdots)c_{-2}^{h_{-2}}\cdot H^0(\Gamma_c(\LL_{\s_n}^{\FGV,d}\otimes\LL_{\s_n^*}^{\FGV,d}))=0$ for some $h_0,h_{-2}\in\ZZ_{\geq 0}$ and a polynomial $f(c_{-1,1},c_{-1,2},\cdots)\in k[c_{-1,1},c_{-1,2},\cdots]$. Since the action of $H^{i}(\Gamma(\s_n^*\otimes\s_n))$ annihilates $H^0(\Gamma_c(\LL_{\s_n}^{\FGV,d}\otimes\LL_{\s_n^*}^{\FGV,d}))$ for $i=1,2$ for degree reasons, we can apply the relation \eqref{eq:glnLMclifford} and get $c_0^{h_0}\cdot H^0(\Gamma_c(\LL_{\s_n}^{\FGV,d}\otimes\LL_{\s_n^*}^{\FGV,d})) =0$. However, Lemma \ref{lem:diagcccalh0} implies that $(c_0\cdot-)^*\ev_{\LL_{\s_n}^{\FGV,d}}=\ev_{\LL_{\s_n}^{\FGV,d-1}}$ for all $d\in\ZZ$. Therefore, we have $c_0^{h_0}\cdot H^0(\Gamma_c(\LL_{\s_n}^{\FGV,d}\otimes\LL_{\s_n^*}^{\FGV,d}))=H^0(\Gamma_c(\LL_{\s_n}^{\FGV,d+h_0}\otimes\LL_{\s_n^*}^{\FGV,d+h_0}))\neq 0$. This gives a contradiction and proves the lemma.

\end{proof}

During the proof above, one sees that \[\coev_{\s_n}([C])\cdot -:H^0(\Gamma_c(\LL_{\s_n}^{\FGV,d}\otimes\LL_{\s_n^*}^{\FGV,d}))\isom H^0(\Gamma_c(\LL_{\s_n}^{\FGV,d+1}\otimes\LL_{\s_n^*}^{\FGV,d+1})).\] This implies that the map \eqref{eq:glndiagheckeactioninj} naturally extends to a map \eqref{eq:geoLMgln}, which has to be an isomorphism since both sides have the same bi-graded dimension. This concludes the proof of Theorem  \ref{thm:geoLMgln}.

\end{proof}

\subsubsection{Proof of Theorem  \ref{thm:geoLMglntrsht}}\label{sec:app:proofdiag}
\begin{proof}[Proof of Theorem  \ref{thm:geoLMglntrsht}]
By Theorem  \ref{thm:geoLMgln}, Assumption \ref{assumption:diagonalgenerate} is satisfied. Therefore, Conjecture \ref{conj:intersection} is true by Proposition \ref{prop:intersectioncc}. Let $\alpha_1,\cdots,\alpha_{2n^2(g-1)+2}$ be the Frobenius eigenvalues for $H^1(\Gamma(\s_n\otimes\s_n^*))$. Again by Theorem  \ref{thm:geoLMgln}, we know \[\begin{split}\tr(\Frob,\Gamma_c(\LL_{\s_n}^d\otimes\LL_{\s_n^*}^d))&= \tr(\Frob, \Sym^{\bullet}\tau^{\leq -1}(\Gamma(\s_n\otimes\s_n^*)\langle 2\rangle)) \\
&=\frac{\prod_{i=1}^{2n^2(g-1)+2}(1-\alpha_iq^{-1})}{1-q^{-1}} \\
&=(\ln q)\Res_{s=1}\frac{\prod_{i=1}^{2n^2(g-1)+2}(1-\alpha_iq^{-s})}{(1-q^{-s})(1-q^{1-s})}\\
&=(\ln q)\Res_{s=1} L(\s_n\otimes\s_n^*,s)
\end{split}.\] This concludes the proof of Theorem  \ref{thm:geoLMglntrsht} by Conjecture \ref{conj:intersection}.

\end{proof}

\subsection{Computing intersection number}\label{sec:app:compute}
We are now ready to wrap up the proof.
\begin{proof}[Proof of Theorem  \ref{thm:intro:main}]
The first claim about the finiteness of non-zero terms has been addressed in \S\ref{sec:app:rscycle}. We only need to prove the identity \eqref{eq:intro:main}. Consider the pairing \begin{equation}
    \langle-,-\rangle_{\s,r}^{(d_n,d_{n-1})}=\sum_{\underline{\e}\in\{\pm 1\}^r_{0}}\langle-,-\rangle_{\underline{\e},\s}^{(d_n,d_{n-1})}\in \Hom^0((M^{\otimes r})_0\otimes (M^{\otimes r})_0,k)
.\end{equation}
By Proposition \ref{prop:app:geo=cyclers} and Proposition \ref{prop:app:geo=fakers}, the desired identity \eqref{eq:intro:main} translates to \begin{equation}\label{eq:app:desired}\begin{split}
    \langle z_{\s,r}, z_{\s^*,r}\rangle^{(d_n,d_{n-1}),*}_{\s,r}\\=\\ q^{\dim\Bun_{\GL_{n-1}}}(\ln q)^{-r-2}\frac{\left(\frac{d}{ds}\right)^{r}\Big|_{s=1/2}\widetilde{L}(\s_n\otimes\s_{n-1}\oplus\s_n^*\otimes\s_{n-1}^*,s)}{\Res_{s=1}\widetilde{L}(\s_n\otimes\s_n^*,s)\Res_{s=1}\widetilde{L}(\s_{n-1}\otimes\s_{n-1}^*,s)}
\end{split}\end{equation} where $z_{\s,r}\in (M^{\otimes r})_0^*$ is introduced in \eqref{eq:app:faketotal}.
We want to compare this with Theorem  \ref{thm:app:kolydermain}. For this purpose, we first compare $z_{\s,r}\in (M^{\otimes r})_0^*$ with $z_{\s^*,r}\in (M^{\otimes r})_0^*$. By comparing Corollary \ref{cor:app:highestwt} and Remark \ref{rmk:app:lowestwt} for $\s$ and $\s^*$, we know that \begin{equation}\label{eq:app:fakevsfakedual}\begin{split}
    z_{\s^*,r}&=\frac{q^{-n^2(g-1)/2}\chi^{-n}_{\det\s_{n-1}^*}(\Om^{1/2})\chi^{-n+1}_{\det\s_n^*}(\Om^{1/2})}{q^{-n^2(g-1)/2}\chi^{-n}_{\det\s_{n-1}}(\Om^{1/2})\chi^{-n+1}_{\det\s_n}(\Om^{1/2})\e(\s_n\otimes\s_{n-1})} z_{\s,r}\\&=\chi^{n}_{\det\s_{n-1}}(\Om)\chi^{n-1}_{\det\s_n}(\Om)\e(\s_n\otimes\s_{n-1})^{-1}z_{\s,r}
\end{split}.\end{equation} 
Then we compare $\om_{M^{\otimes r}}|_{(M^{\otimes r})_0^{\otimes 2}}$ with $\langle -, -\rangle^{(d_n,d_{n-1})}_{\s,r}$. By combining Proposition \ref{prop:app:geo=cyclediaggln} and Theorem \ref{thm:geoLMglntrsht} (with identity replaced by \eqref{eq:app:nn-1LM}), we get \begin{equation}\begin{split}
    \langle -, -\rangle^{(d_n,d_{n-1})}_{\s,r}\\=\\(-1)^{r/2}(\ln q)^2\cdot \Res_{s=1} L(\s_n\otimes\s_n^*,s)\Res_{s=1} L(\s_{n-1}\otimes\s_{n-1}^*,s)\om_{M^{\otimes r}}|_{(M^{\otimes r})_0^{\otimes 2}}
\end{split}.\end{equation} Here the sign $(-1)^{r/2}$ arises from our definition of $\om_{M}$, which can be tracked back to the negative sign in \eqref{eq:app:omk}.\footnote{This is because $\om_{M^{\otimes r}}|_{(M^{\otimes r})_0^{\otimes 2}}$ involves $r/2$-times $\ev_{\Std_n\otimes\Std_{n-1}}$ and $r/2$-times $-\ev_{\Std_n^*\otimes\Std_{n-1}^*}$} By passing to the pairing on the dual space, one has \begin{equation}\label{eq:app:fakevsint}\begin{split}
    \langle -, -\rangle^{(d_n,d_{n-1}),*}_{\s,r}\\=\\(-1)^{r/2}\frac{1}{(\ln q)^2\cdot \Res_{s=1} L(\s_n\otimes\s_n^*,s)\Res_{s=1} L(\s_{n-1}\otimes\s_{n-1}^*,s)}\om_{(M^{\otimes r})^*}|_{(M^{\otimes r})_0^{*\otimes 2}}
\end{split}.\end{equation}

Combining Theorem  \ref{thm:app:kolydermain}, \eqref{eq:app:fakevsfakedual}, \eqref{eq:app:fakevsint}, we get \[\begin{split}
     \langle z_{\s,r}, z_{\s^*,r}\rangle^{(d_n,d_{n-1}),*}_{\s,r}\\=\\ q^{-n^2(g-1)}(\ln q)^{-r-2}\frac{\left(\frac{d}{ds}\right)^{r}\Big|_{s=1/2}\widetilde{L}(\s_n\otimes\s_{n-1}\oplus\s_n^*\otimes\s_{n-1}^*,s)}{\Res_{s=1}L(\s_n\otimes\s_n^*,s)\Res_{s=1}L(\s_{n-1}\otimes\s_{n-1}^*,s)}
\end{split}
.\] 

Then the desired equality \eqref{eq:app:desired} follows from the equality above and
\[
    \widetilde{L}(\s_n\otimes\s_n^*,s)=q^{n^2(g-1)s}L(\s_n\otimes\s_n^*,s)
\] \[
    \widetilde{L}(\s_{n-1}\otimes\s_{n-1}^*,s)=q^{(n-1)^2(g-1)s}L(\s_{n-1}\otimes\s_{n-1}^*,s)
\]\[\dim\Bun_{\GL_{n-1}}=(n-1)^2(g-1).\] This concludes the proof of Theorem  \ref{thm:intro:main}.

\end{proof}

\bibliographystyle{amsalpha}
\bibliography{Bibliography}
\end{document}